\documentclass[12pt,oneside,a4papre]{amsart}

\usepackage{amssymb}
\usepackage{amsmath}
\usepackage{amsthm}
\usepackage{amscd}
\usepackage[all]{xy}

\usepackage{longtable}
\usepackage{mathrsfs}

\usepackage[dvipdfmx]{xcolor}
\usepackage[dvipdfmx]{pict2e}
\usepackage[dvipdfmx]{graphicx}
\usepackage{comment}
\usepackage{todonotes}

\usepackage{tikz}

\usepackage[all]{xy}

\usetikzlibrary{cd}

\theoremstyle{plain}
\newtheorem{theorem}{Theorem}[section]
\newtheorem{lemma}[theorem]{Lemma}
\newtheorem{corollary}[theorem]{Corollary}
\newtheorem{proposition}[theorem]{Proposition}
\newtheorem{remark}[theorem]{Remark}
\newtheorem{conjecture}[theorem]{Conjecture}

\theoremstyle{definition}
\newtheorem{definition}[theorem]{Definition}

\newcommand{\kah}{K\"{a}hler }
\newcommand{\idd}{i\partial\overline{\partial}}


\begin{document}
\title
[singular $\omega$-trace positivity]
{$\omega$-trace and Griffiths positivity for singular Hermitian metrics}
\author{Yuta Watanabe}
\address{Graduate School of Mathematical Sciences, The University of Tokyo, 3-8-1 Komaba, Meguro-Ku, Tokyo 153-8914, Japan}
\email{wyuta.math@gmail.com}
\date{}

\begin{abstract}
   In this paper, we investigate various positivity for singular Hermitian metrics such as Griffiths, $\omega$-trace and RC, where $\omega$ is a Hermitian metric, and show that these quasi-positivity notions induce $0$-th cohomology vanishing, rational conected-ness, etc.
   Here, $\omega$-trace positivity of smooth Hermitian metrics $h$ on holomorphic vector bundles $E$ represents the positivity of $tr_\omega i\Theta_{E,h}$.
\end{abstract}

\vspace{-10mm}

\maketitle

\vspace*{-4mm}

\setcounter{tocdepth}{1}
\tableofcontents

\vspace{-8mm}

\section{Introduction}

In several complex variables and complex algebraic geometry, positivity notions for holomorphic vector bundles have played an important role. 
The concept of singular Hermitian metrics has been introduced, and its positivity and properties have been studied (cf. \cite{deC98,BP08}).
However, the curvature currents cannot always be defined with measure coefficients \cite{Rau15}.
Therefore, it is necessary to define the notion of positivity without using curvature currents.  
For Griffiths semi-positivity and semi-negativity, there is already such a characterization (see Proposition \ref{characterization of Grif posi for smooth}) using plurisubharmonicity, 
which is used to define Griffiths positivity for singular Hermitian metrics (see Definition \ref{def Griffiths semi-posi sing}).
Throughout this paper, let $X$ be an $n$-dimensional complex manifold and $E\to X$ be a holomorphic vector bundle, %of finite rank $r$ 
and $\omega$ is not necessarily \kah unless otherwise noted.

In this paper, we consider quasi-positivity for singular Hermitian metrics and its properties such as vanishing theorems derived from it.
In the smooth case, quasi-positivity on $X$ is defined by that it is semi-positive on $X$ and there exists at least one point $x_0\in X$ such that it is positive at $x_0$. 
%Let $\omega$ be a Hermitian metric on $X$ and $h$ be a smooth Hermitian metric on $E$. 
For a Hermitian metric $\omega$ on $X$, we say that a smooth Hermitian metric $h$ on $E$ is $\omega$-trace (semi)-positive if the operator $tr_\omega i\Theta_{E,h}$ $\in \Gamma(X,\mathrm{End}(E))$ is (semi)-positive definite.
Here, this positivity is equivalent to the positivity of the curvature operator $[i\Theta_{E,h},\Lambda_\omega]$ on $\bigwedge^{n,n}T^*_X\otimes E$ (see Proposition \ref{characterization of tr-posi and A^(n,n) posi}).
First, we introduce the definition of $\omega$-subharmonicity for locally integrable functions and investigate its properties in $\S2$. 
Then, using the characterization of $\omega$-trace positivity without curvature currents in $\S4$, we define $\omega$-trace positivity for singular Hermitian metrics and obtain the following characterization to Griffiths semi-positivity in $\S5$.

\begin{theorem}$($\textnormal{=\,Theorem \ref{characterization of sing Grif semi-posi and tr-omega semi-posi}}$)$
    Let $X$ be a complex manifold and $E$ be a holomorphic vector bundle on $X$. %equipped with a singular Hermitian metric $h$.
    Then the following conditions are equivalent
    \begin{itemize}
        \item [$(a)$] $E$ has a singular Hermitian metric $h_G$ with Griffiths semi-positivity.
        \item [$(b)$] $E$ has a singular Hermitian metric $h_{tr}$ which %such that $h_{tr}$
         is $\omega$-trace semi-positive $($with approximation$)$ for any Hermitian metric $\omega$.
    \end{itemize}
    In particular, if $(a)$ holds then we can take $h_G=h_{tr}$ in $(b)$, and if $(b)$ holds then $h_G$ in $(a)$ can be taken to coincide with $h_{tr}$ almost everywhere.
\end{theorem}

Moreover, we show not pseudo-effectiveness of the dual vector bundle $E^*$ if $E$ has a singular Hermitian metric with $\omega$-trace positivity in $\S6$, and in $\S3$ when $\mathrm{rank}\,E=1$.

Second, we introduce the notion of $\mathrm{deg}\,_\omega$-strictly maximal (see Definition \ref{def of deg-max}) for a torsion-free coherent sheaf $\mathcal{F}$ on $X$, where $X$ is compact \kah and $\omega$ is a \kah metric on $X$. 
This means that $\mathrm{deg}\,_\omega(\mathcal{F})$ is greater than the degree of any coherent subsheaf. 
We give $0$-th cohomology vanishing for the dual sheaf of $\mathcal{F}$ which is $\mathrm{deg}\,_\omega$-strictly maximal, and show that a holomorphic vector bundle is $\mathrm{deg}\,_\omega$-strictly maximal if it is $\omega$-trace quasi-positive with approximation (see Theorem \ref{tr-posi then deg max}). %in $\S7$.

Using these, we provide the following vanishing theorem for quasi-positivity of singular Hermitian metrics in $\S8$.
This theorem is a generalization of the known $n$-th cohomology vanishing for Griffiths positivity (see Theorem \ref{V-thm of Inayama}, \cite{Ina20}), which requires conditions for the Lelong number of $\log\mathrm{det}\,h$.

\begin{theorem}$($\textnormal{=\,Theorem \ref{V-thm for tr-quasi-posi with appro}}$)$
    Let $X$ be a compact \kah manifold and $E$ be a holomorphic vector bundle over $X$ equipped with a singular Hermitian metric $h$. 
    If there exists a \kah metric $\omega$ such that $h$ is $\omega$-trace quasi-positive with approximation, 
    then %$E$ is generically $\omega^{n-1}$ strictly positive and 
    %for any $m\in\mathbb{N}$ 
    we have the following cohomology vanishing
    \begin{align*}
        H^0(X,(E^*)^{\otimes m})&=0,\\
        H^0(X,\Lambda^pE^*)&=0,
    \end{align*}
    for any $m\in\mathbb{N}$ and any $1\leq p\leq \mathrm{rank}\,E$.
\end{theorem}

Since positivity is weakened to quasi-positivity, it is no longer possible to use the known $L^2$-estimate method, and the Hodge decomposition cannot be used as in the smooth case.
Recently, a Fujita Conjecture type theorem involving multiplier ideal sheaves was presented in \cite{SY19} using vanishing theorems.
Finally, we obtain the Fujita Conjecture type theorem (see. Theorem \ref{Fujita Conj for Grif quai-posi}) involving the $L^2$-subsheaf for Griffiths quasi-positivity as an application. %of our vanishing theorems.

Yau's conjecture (see Conjecture \ref{Yau's conjecture}, \cite[Problem\,47]{Yau82}) requiring positive holomorphic sectional curvature for rational connected-ness was solved in \cite{HW20,Mat22}, and in \cite{Yan18} by introducing the notion of RC-positivity.
More recently, the conditions of Yau's Conjecture have been generalized to quasi-positivity and solved in \cite{ZZ23}.
Here, it is a sufficient condition for the \kah metric $\Omega$ on $T_X$ to be RC-positive that $\Omega$ has positive holomorphic sectional curvature, or is $\omega$-trace positive for a Hermitian metric $\omega$ on $X$.
%Here, a \kah metric $\Omega$ on $T_X$ is RC-positive if $\Omega$ has positive holomorphic sectional curvature or is $\omega$-trace positive for a Hermitian metric $\omega$ on $X$.
As a generalization to singular Hermitian metrics and quasi-positivity with respect to $\omega$-trace positivity,
%For $\omega$-trace quasi-positivity of singular Hermitian metrics, 
we obtain the following theorem by using the characterization of rational conected-ness in \cite{CDP14}.

\begin{theorem}$($\textnormal{=\,Theorem \ref{rationally conected if tr-quasi posi}}$)$
    Let $X$ be a compact \kah manifold. %and $h$ be a singular Hermitian metric on $T_X$.
    If there exist a Hermitian metric $\omega$ on $X$ and a singular Hermitian metric $h$ on $T_X$ such that $h$ is $\omega$-trace quasi-positive with approximation, then $X$ is projective and rationally connected.
\end{theorem}

Here, this theorem is already known if $h$ is smooth and $\omega$-trace positive (see \cite[Corollary\,1.5]{Yan18}).
Finally, it is shown that generically $\omega^{n-1}$ positivity (see Definition \ref{def of generically omega posi}, \cite[Section\,6]{Miy87}) yields from $\omega$-trace positivity with approximation. %in $\S9$.

\section{$\omega$-subharmonicity of locally integrable functions}

In this section, we introduce the definition of $\omega$-subharmonicity for locally integrable functions and investigate its properties.

Let $X$ be a Hermitian manifold of dimension $n$ with a Hermitian metric $\omega$. 
Let $L^2_{p,q}(X,\omega)$ be a Hilbert space of $(p,q)$-forms $u$ with measurable coefficients such that 
\begin{align*}
    ||u||^2_\omega=\int_X|u|^2_\omega dV_\omega<+\infty.
\end{align*}
We denote by $\langle\langle\bullet,\bullet\rangle\rangle_\omega$ the global inner product of %a Hilbert space 
$L^2_{p,q}(X,\omega)$ with respect to $\omega$, i.e.
\begin{align*}
    \langle\langle u,v\rangle\rangle_\omega=\int_X u\wedge \ast_\omega v=\int_X\langle u,v\rangle_\omega dV_\omega,
\end{align*}
for any $u,v\in L^2_{p,q}(X,\omega)$, where $\ast_\omega$ is the Hodge-star operator $L^2_{p,q}(X,\omega)\to L^2_{n-p,n-q}(X,\omega)$.

%Let $(X,\omega)$ be a Hermitian metric. 
For any $(1,1)$-form $\alpha\in\mathcal{E}^{1,1}(X)$, we have that 
\begin{align*}
    tr_\omega\alpha=\frac{n\alpha\wedge\omega^{n-1}}{\omega^n},
\end{align*}
where $tr_\omega\alpha=\sum g^{jk}\alpha_{jk}$ if $\alpha$ and $\omega$ given in local coordinates by $\alpha=i\sum\alpha_{jk}dz_j\wedge d\overline{z}_k$ and $\omega=i\sum g_{jk}dz_j\wedge d\overline{z}_k$.
%Choosing local coordinates $z=(z_1,\ldots,z_n)$ on $X$, 
We define the \textit{Chern Laplacian} for smooth functions
\begin{align*}
    \Delta^{Ch}_\omega f:%=+\sum g^{jk}\frac{\partial^2f}{\partial z_j\partial \overline{z}_k}%=-\Lambda_\omega\idd f
    =2\sqrt{-1}tr_\omega\overline{\partial}\partial f=-2\frac{n\idd f\wedge\omega^{n-1}}{\omega^n},
\end{align*}
where $\Delta^{Ch}_\omega=-2\sum g^{jk}\partial_{z_j}\partial_{\overline{z}_k}$ for local coordinates $(z_1,\ldots,z_n)$. 
The Chern Laplacian has several names in literature such as canonical-Laplacian or complex-Laplacian.
In this paper, we define the symbol $\Delta^C_\omega$ by $-\Delta^{Ch}_\omega/2$, i.e. 
\begin{align*}
    \Delta^C_\omega f:=tr_\omega \idd f=\frac{n\idd f\wedge \omega^{n-1}}{\omega^n}
\end{align*}
for any smooth functions, according to the usual Laplacian $\Delta_{\mathbb{C}^n}=\sum \partial_{z_j}\partial_{\overline{z}_j}$ on $\mathbb{C}^n$.

In \cite{Gau84}, Gauduchon made explicit the relation between the Chern Laplacian $\Delta^{Ch}_\omega$ and Laplace-Beltrami operator $\Delta_\omega$
on smooth functions through the torsion $1$-form, which we recall is defined by the equation
\begin{align*}
    d\omega^{n-1}=\theta\wedge\omega^{n-1}.
\end{align*}

%Denote $(\Delta^C_\omega)^*$ by the adjoint of $\Delta^C_\omega$ with respect to the inner product $\langle\langle\bullet,\bullet\rangle\rangle_\omega$.

\begin{lemma}\label{Gau84 lemma}$($\textnormal{cf.\,\cite{Gau84}}$)$
    Let $X$ be a compact manifold endowed with a Hermitian metric $\omega$ with torsion $1$-form $\theta$. 
    The Chern Laplacian on smooth functions $f$ has %the forms 
    \begin{align*}
        \Delta^{Ch}_\omega f=\Delta_\omega f+\langle df,\theta\rangle_\omega, \quad
        (\Delta^{Ch}_\omega)^* f=\Delta_\omega f-\langle df,\theta\rangle_\omega+d_\omega^*\theta\cdot f,
    \end{align*}
    where $(\Delta^{Ch}_\omega)^*$ is the adjoint of $\Delta^{Ch}_\omega$ with respect to the inner product $\langle\langle\bullet,\bullet\rangle\rangle_\omega$.
\end{lemma}
%Lemma の compact は必要!!

Here, $(\Delta^C_\omega)^*$ is also the adjunct of $\Delta^C_\omega$ similarly.
Thus, the Chern Laplacian and the Laplacian-Beltrami operator on smooth functions coincide when $\omega$ is balanced, i.e. $d\omega^{n-1}=0$. %,???らしいけどなぜ?
In particular, if $\omega$ is \kah then %$\omega$ is balanced, 
$\Delta^{Ch}_\omega$ is self-adjoint and $\Delta^{Ch}_\omega=\Delta_\omega=-2\Delta^C_\omega$. %$\Delta_\omega^{\overline{\partial}}=\Delta_\omega^\partial=-2\Delta^C_\omega$.

%And %if $X$ is compact or the acting function has a compact support, 
Moreover, $(\Delta^C_\omega)^*$ can also be written as follows. For any smooth function $g\in\mathcal{C}^{\infty}(X)$,
\begin{align*}
    (\Delta^C_\omega)^*g=\frac{n\idd(g\omega^{n-1})}{\omega^n},
\end{align*}
if $X$ is compact or $g$ has a compact support, i.e. $g\in\mathscr{D}(X)$.
In fact, for any $f\in\mathcal{C}^{\infty}(X)$ %we have that 
\begin{align*}
    \langle\langle\Delta^C_\omega f,g\rangle\rangle_\omega=\int_X \Delta^C_\omega f\cdot g \,dV_\omega
    =\int_X\frac{n\idd f\wedge g\omega^{n-1}}{n!}=\int_Xf\cdot \frac{n\idd(g\omega^{n-1})}{\omega^n}dV_\omega.
\end{align*}

Here, $(X,\omega)$ is a Hermitian manifold which is not necessarily compact and $U\subset X$ is an open subset contained in a local chart. 
We introduce definitions of $\omega$-subharmonicity for locally integrable functions.

\begin{definition}\label{Def of omega-subharmonic}$($\textnormal{cf. \cite[Definition\,2.5]{Chi13}}$)$
    A locally integrable function $\varphi\in\mathscr{L}^1_{loc}(U)$ is called $\omega$-\textit{subharmonic} if 
    \begin{align*}
        \idd\varphi\wedge\omega^{n-1}\geq0
    \end{align*}
    in the sense of currents.
\end{definition}

\begin{definition}
    For any locally integrable function $\varphi\in\mathscr{L}^1_{loc}(U)$, we define $\Delta^C_\omega\varphi\geq0$ in the sense of currents by for any test functions $\phi\in\mathscr{D}(U)_{\geq0}$
    \begin{align*}
        0\leq\langle\langle\Delta^C_\omega\varphi,\phi\rangle\rangle_\omega:=\langle\langle\varphi,(\Delta^C_\omega)^*\phi\rangle\rangle_\omega.
    \end{align*}
\end{definition}

%近い定義に[Chi13,\,Definition\,2.5]などがある.

Littman's results on approximation in \cite{Lit63} are presented below.
Let $D$ be a bounded domain in $\mathbb{R}^n$ and $L$ be the partial differential operator which is uniformly elliptic.
For any locally integrable function $u$ in $D$, we say that $Lu\geq0$ \textit{in the weak sense} if the inequality 
\begin{align*}
    \langle Lu,\phi\rangle:=\int u(x)L^*\phi(x)dx\geq0
\end{align*}
holds for all non-negative test function $\phi\in\mathscr{D}(D)_{\geq0}$ with compact support, where $L^*$ is the formal adjoint of $L$. %(It is weakened to $\mathcal{C}^2_0$ within \cite{Lit63}).
Littman shown the pproximation property that is the existence of a smooth decreasing approximate sequence $(u_\nu)_{\nu\in\mathbb{N}}$ preserving positivity $Lu_\nu\geq0$.

There is a definition close to Definition \ref{Def of omega-subharmonic}, \cite[Definition\,2.7]{Chi13}, which is defined in the limit of approximation using Littman's result.
In this paper, we adopt Definition \ref{Def of omega-subharmonic} which is easier to handle, although the upper semi-continuity is lost.

\begin{proposition}\label{equivalence of currents semi-posi}
    For any locally integrable function $\varphi\in\mathscr{L}^1_{loc}(U)$, the following conditions are equivalent.
    \begin{itemize}
        \item [($a$)] $\varphi$ is $\omega$-subharmonic,
        \item [($b$)] $\Delta^C_\omega\varphi\geq0$ in the sense of currents,
        \item [($c$)] $\Delta^C_\omega\varphi\geq0$ in the weak sense.
    \end{itemize}
\end{proposition}

\begin{proof}
    We fiexd any non-negative test function $\phi\in\mathscr{D}(U)_{\geq0}$. 
    
    $(a) \iff (b)$. We have that
    \begin{align*}
        \langle\idd\varphi\wedge\omega^{n-1},\phi\rangle&=\int_U\idd\varphi\wedge\omega^{n-1}(\phi)=\int_U\idd\varphi\wedge\phi\omega^{n-1}=\int_U\varphi\cdot\idd(\phi\omega^{n-1})\\
        &=n!\int_U\varphi\cdot\frac{\idd(\phi\omega^{n-1})}{\omega^n}dV_\omega=(n-1)!\int_U\varphi\cdot(\Delta^C_\omega)^*\phi \,dV_\omega\\
        &=(n-1)!\langle\langle\Delta^C_\omega\varphi,\phi\rangle\rangle_\omega.
    \end{align*}

    $(b) \iff (c)$. Let $(\Delta^C_\omega)^{\dag}$ be the formal adjoint of $\Delta^C_\omega$ with respect to the inner product $\int \Delta^C_\omega u(z)\cdot v(z)dV_{\mathbb{C}^n}$ for any $u,v\in\mathcal{C}^\infty(U)$.
    There exist a positive function $\gamma\in\mathcal{C}^\infty(U)_{>0}$ such that $dV_\omega=\gamma dV_{\mathbb{C}^n}$ on $U$. 
    Uesing mollifier $(\rho_\varepsilon)_{\varepsilon>0}$, we define the smooth function $\varphi_\varepsilon:=\varphi\ast\rho_\varepsilon$ then $\varphi_\varepsilon\to\varphi$ a.e. ($\varepsilon\to+0$). 
    Therefore, we have
    \begin{align*}
        \langle\langle\Delta^C_\omega\varphi,\phi\rangle\rangle_\omega&=\langle\langle\varphi,(\Delta^C_\omega)^*\phi\rangle\rangle_\omega=\int_U\varphi(\Delta^C_\omega)^*\phi\,dV_\omega\\
        &=\lim_{\varepsilon\to+0}\int_U\varphi_\varepsilon(\Delta^C_\omega)^*\phi\,dV_\omega=\lim_{\varepsilon\to+0}\int_U\Delta^C_\omega\varphi_\varepsilon\cdot\phi\,dV_\omega=\lim_{\varepsilon\to+0}\int_U\Delta^C_\omega\varphi_\varepsilon\cdot\phi\gamma\,dV_{\mathbb{C}^n}\\
        &=\lim_{\varepsilon\to+0}\int_U\varphi_\varepsilon(\Delta^C_\omega)^{\dag}(\phi\gamma)\,dV_{\mathbb{C}^n}=\int_U\varphi(\Delta^C_\omega)^{\dag}(\phi\gamma)\,dV_{\mathbb{C}^n}\\
        &=\langle \Delta^C_\omega\varphi,\phi\gamma\rangle,
    \end{align*}
    by Lebesgue dominated convergence theorem.
\end{proof}

Thanks to Littman \cite[Theorem\,A]{Lit63}, we have the following approximation property.
In addition, examining the $\omega$-subharmonicity.

\begin{proposition}\label{smoothing of omega-subharmonic}
    Let $\varphi$ be a $\omega$-subharmonic function on $U$. 
    Then there exists a sequence of smooth $\omega$-subharmonic functions $(\varphi_\nu)_{\nu\in\mathbb{N}}$ decreasing to $\varphi$ a.e. on every compact subset $U'\subset U$.
    
    Moreover, $\varphi_\nu$ is given explicitly as $\varphi_\nu(z)=\int_U K_\nu(\zeta,z)\varphi(\zeta)dV_\zeta$, %the result of an integral operator acting on $\varphi$:
    %\begin{align*}
    %    \varphi_\nu(z)=\int K_\nu(\zeta,z)\varphi(\zeta)dV_\zeta,
    %\end{align*}
    where the kernel $K_\nu (\zeta, z)$ is smooth non-negative and is constructed explicitly from the Green's function of $\Delta^C_\omega$ and $U$, and $\int_UK_\nu(\zeta,z)dV_\zeta\to1$ uniformly in $z\in U'$.
\end{proposition}

%ここで$\Delta^C_\omega$はちゃんと楕円的である.

\begin{theorem}\label{omega-SH for any omega then PSH}
    Let $\varphi\in\mathscr{L}^1_{loc}(U)$ be a locally integrable function on $U$. If $\varphi$ is $\omega$-subharmonic for any Hermitian metric $\omega$, then $\idd\varphi\geq0$ in the sense of currents. 
    In particular, there exists a plurisubharmonic function $\widetilde{\varphi}$ on $U$ such that $\varphi=\widetilde{\varphi}$ a.e.
\end{theorem}

\begin{proof}
    The condition $\idd\varphi\geq0$ in the sense of currents is equivalent to that  
    \begin{align*}
        \sum\frac{\partial^2\varphi}{\partial z_j\partial\overline{z}_k}\xi_j\overline{\xi}_k
    \end{align*}
    is positive distribution for any $0\ne\xi\in\mathbb{C}^n$. Assume that $\varphi$ is not plurisubharmonic. Then there exists $\xi=\sum\xi_j\frac{\partial}{\partial z_j}\in T_U\cong\mathbb{C}^n$ such that $\idd\varphi(\xi)=\sum\frac{\partial^2\varphi}{\partial z_j\partial\overline{z}_k}\xi_j\overline{\xi}_k$ is not positive distribution, 
    i.e. there exists a test function $\phi\in\mathscr{D}(U)_{\geq0}$ such that 
    \begin{align*}
        C:=\int_U\sum\frac{\partial^2\varphi}{\partial z_j\partial\overline{z}_k}\xi_j\overline{\xi}_k\cdot\phi\, dV_{\mathbb{C}^n}<0.
    \end{align*}

    Here, there exists local coordinates $(w_1,\cdots,w_n)$ such that $\xi$ and $\partial/\partial w_1$ are parallel, i.e. $\xi=\zeta\frac{\partial}{\partial w_1}$ for a non-zero number $0\ne\zeta\in\mathbb{R}$.
    Then we get 
    \begin{align*}
        \sum\frac{\partial^2\varphi}{\partial z_j\partial\overline{z}_k}\xi_j\overline{\xi}_k=\idd\varphi(\xi)=|\zeta|^2\idd\varphi\Bigl(\frac{\partial}{\partial w_1}\Bigr)=|\zeta|^2\frac{\partial^2\varphi}{\partial w_1\partial\overline{w}_1}.
    \end{align*}

    For any $\varepsilon>0$, let $\omega_\varepsilon$ be a Hermitian metric by
    \begin{align*}
        \omega_\varepsilon:=\frac{i}{|\zeta|^2}dw_1\wedge d\overline{w}_1+\frac{i}{\varepsilon}\sum_{2\leq j\leq n} dw_j\wedge d\overline{w}_j,
    \end{align*}
    and there exists a positive function $\gamma_\varepsilon\in\mathcal{C}^{\infty}(U)_{>0}$ such that $\gamma_\varepsilon\, dV_{\omega_\varepsilon}=dV_{\mathbb{C}^n}$.
    By $\omega$-subharmonicity of $\varphi$, for any $\varepsilon>0$ and any test function $\phi\gamma_\varepsilon\in\mathscr{D}(U)_{\geq0}$ we have %that
    \begin{align*}
        0&\leq\frac{1}{(n-1)!}\langle\idd\varphi\wedge\omega^{n-1}_\varepsilon,\phi\gamma_\varepsilon\rangle=\frac{1}{(n-1)!}\int_U\idd\varphi\wedge \phi\gamma_\varepsilon\omega^{n-1}_\varepsilon\\
        &=\int_U\frac{n\idd\varphi\wedge\omega^{n-1}_\varepsilon}{\omega^n_\varepsilon}\phi\gamma_\varepsilon\, dV_{\omega_\varepsilon}=\int_Utr_{\omega_\varepsilon}\idd\varphi\cdot\phi\,dV_{\mathbb{C}^n}\\
        &=\int_U\Bigl(|\zeta|^2\frac{\partial^2\varphi}{\partial w_1\partial\overline{w}_1}+\varepsilon\sum_{2\leq j\leq n}\frac{\partial^2\varphi}{\partial w_j\partial\overline{w}_j}\Bigr)\cdot\phi\,dV_{\mathbb{C}^n}\\
        &=C+\varepsilon\int_U\sum_{2\leq j\leq n}\frac{\partial^2\varphi}{\partial w_j\partial\overline{w}_j}\cdot\phi\,dV_{\mathbb{C}^n}.
    \end{align*}
    
    Therefore, for the following real number 
    \begin{align*}
        I:=\int_U\sum_{2\leq j\leq n}\frac{\partial^2\varphi}{\partial w_j\partial\overline{w}_j}\cdot\phi\,dV_{\mathbb{C}^n},
    \end{align*}
    there is $\varepsilon_0>0$ such that $C+\varepsilon_0 I<0$. This contradicts $\omega_{\varepsilon_0}$-subharmonicity of $\varphi$.
\end{proof}

%Furthermore, 
We define strictly $\omega$-subharmonicity by reference to strictly plurisubharmonicity.

\begin{definition}
    A locally integrable function $\varphi\in\mathscr{L}^1_{loc}(U)$ is called \textit{strictly} $\omega$-\textit{subharmonic} if 
    for any $x\in U$, there exist a neighborhood $V\subset U$ of $x$ and $c_x>0$ such that 
    \begin{align*}
        \idd\varphi\wedge\omega^{n-1}\geq c_x\omega^n
    \end{align*}
    on $V$ in the sense of currents.
\end{definition}

\begin{proposition}\label{smoothing of strictly omega-subharmonic}
    Let $\varphi\in\mathscr{L}^1_{loc}(U)$ be a locally integrable function on $U$. If $\varphi$ is strictly $\omega$-subharmonic then for any $x\in U$
    there exist a neighborhood $V$ of $x$, $c_x>0$ and a sequence of smooth strictly $\omega$-subharmonic functions $(\varphi_\nu)_{\nu\in\mathbb{N}}$ decreasing to $\varphi$ a.e. on $V$ 
    such that for any $\nu\in\mathbb{N}$, $\idd\varphi_\nu\wedge\omega^{n-1}\geq c_x\omega^n$ on $V$ in the sense of currents, which is called a uniformly positive condition. 
\end{proposition}

\begin{proof}
    By the assumption, there exist a neighborhood $W$ of $x$ and $C_x>0$ such that $\idd\varphi\wedge\omega^{n-1}\geq C_x\omega^n$ on $W$ in the sense of currents.
    Here, there exist a \kah metric $\widetilde{\omega}$ and $\delta>0$ such that $\omega\geq\widetilde{\omega}\geq \delta\omega$ on $U$, and then there exists a smooth \kah potential $\psi$ of $\widetilde{\omega}$, i.e. $\idd\psi=\widetilde{\omega}$ on $W$.
    From the inequality
    \begin{align*}
        \idd\varphi\wedge\omega^{n-1}\geq C_x\omega^n=C_x\omega\wedge\omega^{n-1}\geq C_x\widetilde{\omega}\wedge\omega^{n-1}=C_x\idd\psi\wedge\omega^{n-1},
    \end{align*}
    i.e. $\idd(\varphi-C_x\psi)\wedge\omega^{n-1}\geq0$, the function $\varphi-C_x\psi$ is weakly $\omega$-subharmonic on $W$.

    By Proposition \ref{smoothing of omega-subharmonic}, there exists a sequence of smooth $\omega$-subharmonic functions $(\varPhi_\nu)_{\nu\in\mathbb{N}}$ decreasing to $\varphi-C_x\psi$ a.e. on a open subset $V\subset W$.
    Define the sequence of smooth functions $(\varphi_\nu)_{\nu\in\mathbb{N}}$ on $V$ by $\varphi_\nu:=\varPhi_\nu+C_x\psi$, then $\varphi_\nu$ decreasing to $\varphi$ a.e. and satisfying a uniformly positive condition;
    \begin{align*}
        %0\leq\idd\varPhi_\nu\wedge\omega^{n-1}=\idd(\varphi_\nu-c_x\psi)\wedge\omega^{n-1}\\
        \idd\varphi_\nu\wedge\omega^{n-1}\geq C_x\idd\psi\wedge\omega^{n-1}=C_x\widetilde{\omega}\wedge\omega^{n-1}\geq C_x\delta\omega^n
    \end{align*}
    on $V$ for any $\nu\in\mathbb{N}$, by $\omega$-subharmonicity of $\varPhi_\nu$.
\end{proof}

This \textit{uniformly positive condition} is necessary for strictly positivity to be inherited from an approximation, and even if an approximation exists without this condition, it is not necessarily inherited.
In fact, for any locally holomorphic function $f\in \mathcal{O}_U$, there exists a sequence of smooth strictly plurisubharmonic functions $(\log(|f|+\varepsilon))_{\varepsilon>0}$ decreasing to $\log|f|$, but its limit $\log|f|$ is not strictly plurisubharmonic.

\begin{proposition}\label{exists smoothing then (strictly) omega-SH as currents}
    Let $\omega$ be a Hermitian metric on $X$ and $c$ be a non-negative number.
    For any locally integrable function $\varphi\in\mathscr{L}^1_{loc}(U)$, if there exists a sequence of smooth functions $(\varphi_\nu)_{\nu\in\mathbb{N}}$ decreasing to $\varphi$ a.e. on $U$ such that $\idd\varphi_\nu\wedge\omega^{n-1}\geq c\omega^n$ for any $\nu\in\mathbb{N}$ on $U$,
    then $\idd\varphi\wedge\omega^{n-1}\geq c\omega^n$ on $U$ in the sense of currents.
\end{proposition}

\begin{proof}
    By the inequality $\Delta^C_\omega\varphi_\nu\geq c$, i.e. $\idd\varphi_\nu\wedge\omega^{n-1}\geq c\omega^n$, for any $\phi\in\mathscr{D}(U)$ we get 
    \begin{align*}
        \langle c\omega^n,\phi\rangle=c\int\phi\cdot\omega^n&\leq\langle\idd\varphi_\nu\wedge\omega^{n-1},\phi\rangle%=\int \phi\cdot\idd\varphi_\nu\wedge\omega^{n-1}
        =\int id\overline{\partial}\varphi_\nu\wedge(\phi\omega^{n-1})
        %&=\int d\Bigl(i\overline{\partial}\varphi_\nu\wedge(\phi\omega^{n-1})\Bigr)+i\overline{\partial}\varphi_\nu\wedge d(\phi\omega^{n-1})\\
        %&=\int i\overline{\partial}\varphi_\nu\wedge \partial(\phi\omega^{n-1})=\int id\varphi_\nu\wedge \partial(\phi\omega^{n-1})\\
        %&=\int d\Bigl(i\varphi_\nu\cdot\partial(\phi\omega^{n-1})\Bigr)-i\varphi_\nu d\partial(\phi\omega^{n-1})\\
        %&=-i\int\varphi_\nu\overline{\partial}\partial(\phi\omega^{n-1})
        =\int\varphi_\nu\idd(\phi\omega^{n-1}).
    \end{align*}
    Hence, we have that 
    \begin{align*}
        \langle c\omega^n,\phi\rangle&\leq\lim_{\nu\to+\infty}\int\varphi_\nu\idd(\phi\omega^{n-1})\leq\int\limsup_{\nu\to+\infty}\varphi_\nu\idd(\phi\omega^{n-1})\\
        &=\int\lim_{\nu\to+\infty}\varphi_\nu\idd(\phi\omega^{n-1})=\int\varphi\idd(\phi\omega^{n-1})=\langle\varphi,\idd(\phi\omega^{n-1})\rangle\\
        &=\langle\idd\varphi,\phi\omega^{n-1}\rangle=\langle\idd\varphi\wedge\omega^{n-1},\phi\rangle,
    \end{align*}
    i.e. $\idd\varphi\wedge\omega^{n-1}\geq c\omega^n$ on $U$ in the sense of currents, by Fatou's lemma.
\end{proof}

\begin{proposition}\label{psh then omega-SH}
    Let $\omega$ be a Hermitian metric on $X$.
    If a function $\varphi$ is $($resp. strictly$)$ plurisubharmonic %on a locally open subset $U\subset\mathbb{C}^n$ 
    then $\varphi$ is $($resp. strictly$)$ $\omega$-subharmonic. %, i.e. $\Delta^C_\omega\varphi\geq0$ on $U$ in the sense of currents.
\end{proposition}

\begin{proof}
    First, we assume that $\varphi$ is strictly plurisubharmonic, i.e. for any fiexd $x\in X$, there exist a neighborhood $V$ of $x$ and $C_x>0$ such that $\varphi-C_x|z|^2$ is plurisubharmonic on $V$.
    For an approximate identity $(\rho_\varepsilon)_{\varepsilon>0}$, the smooth function $\varphi_\varepsilon:=\varphi\ast\rho_\varepsilon$ on $V_\varepsilon:=\{z\in V\mid d_{\mathbb{C}^n}(z,\partial V)<\varepsilon\}$ is also strictly plurisubharmonic, i.e. $\idd\varphi_\varepsilon\geq C_x\idd(|z|^2\ast\rho_\varepsilon)=C_x\idd|z|^2$.
    Here, there exists $\delta>0$ such that $\idd|z|^2\geq\delta\omega$ on $V$.
    Then we have that $\idd\varphi_\varepsilon\wedge\omega^{n-1}\geq C_x\idd|z|^2\wedge\omega^{n-1}\geq\delta C_x\omega^n=c_x\omega^n$ on $V$, where $c_x=\delta C_x>0$. 
    If $\varphi$ is plurisubharmonic, it can be taken as $c_x=0$. Therefore, the proof concludes from Proposition \ref{exists smoothing then (strictly) omega-SH as currents}.
\end{proof}

\begin{proposition}\label{pull back omega-SH for hol immersion map}
    Let $f:X\to Y$ be a holomorphic immersion mapping between complex manifolds, $U$ be an open subset of $Y$ and $\omega$ be a Hermitian metric on $Y$.
    If a locally integrable function $\varphi\in\mathscr{L}^1_{loc}(U)$ is $($resp. strictly$)$ $\omega$-subharmonic then the pull back function $f^*\varphi$ is $($resp. strictly$)$ $f^*\omega$-subharmonic on $f^{-1}(U)$.
\end{proposition}

\begin{proof}
    We show strictly $f^*\omega$-subharmonicity. Fiexd any point $x\in f^{-1}(U)$. Due to the local property, it is sufficient to show this in some neighborhood of $x$.
    Let $y=f(x)$. By Proposition \ref{smoothing of strictly omega-subharmonic}, there exist a neighborhood $W$ of $y$, $c>0$ and a sequence of smooth strictly $\omega$-subharmonic functions $(\varphi_\nu)_{\nu\in\mathbb{N}}$ decreasing to $\varphi$ a.e.\! on $W$ 
    satisfying %the uniformly positive condition, i.e. 
    $\idd\varphi_\nu\wedge\omega^{n-1}\geq c\omega^n$ for any $\nu\in\mathbb{N}$.
    There exists an open neighborhood $V$ of $x$ satisfying $f(V)\subset W$ and $f|_V$ is an embedding. For any $\nu\in\mathbb{N}$, we have that 
    \begin{align*}
        \idd f^*\varphi_\nu\wedge f^*\omega^{n-1}=f^*(\idd\varphi_\nu\wedge\omega^{n-1})\geq cf^*\omega^n,
    \end{align*}
    i.e. $\Delta^C_{f^*\omega}f^*\varphi_\nu\geq c/n$ on $V$. Hance, $f^*\varphi$ is strictly $f^*\omega$-subharmonic by Proposition \ref{exists smoothing then (strictly) omega-SH as currents}. 
    And $f^*\omega$ subharmonicity is also shown in the same way.
\end{proof}

\begin{proposition}\label{convex increasing of omega-SH is also omega-SH}
    For any smooth convex increasing function $\chi$, if a locally integrable function $\varphi$ on $U$, i.e. $\varphi\in\mathscr{L}^1_{loc}(U)$, is $\omega$-subharmonic then $\chi\circ\varphi$ is also $\omega$-subharmonic.

    %Let $\varphi\in\mathscr{L}^1_{loc}(U)$ be a locally integrable function on $U$ and $\chi$ be a smooth convex increasing function.
    %If $\varphi$ is $\omega$-subharmonic then $\chi\circ\varphi$ is also $\omega$-subharmonic.
\end{proposition}

\begin{proof}
    By Proposition \ref{smoothing of omega-subharmonic}, there exists a sequence of smooth $\omega$-subharmonic functions $(\varphi_\nu)_{\nu\in\mathbb{N}}$ decreasing to $\varphi$ a.e. on every compact subset $U'\subset U$.
    For any $\nu\in\mathbb{N}$, we get 
    \begin{align*}
        \Delta^C_\omega\chi\circ\varphi_\nu(z)=\chi'\circ\varphi_\nu(z)\Delta^C_\omega\varphi_\nu(z)+\chi''\circ\varphi_\nu(z)|\partial\varphi_\nu|^2_\omega(z)\geq0,
    \end{align*}
    i.e. $\chi\circ\varphi_\nu$ is also $\omega$-subharmonic. Here $\chi\circ\varphi_\nu$ decreasing to $\chi\circ\varphi$ a.e. 
    Hance, $\chi\circ\varphi$ is also $\omega$-subharmonic by Proposition \ref{exists smoothing then (strictly) omega-SH as currents}.
\end{proof}

However, strictly $\omega$-subharmonicity is generally not inherited unless the function is smooth.
This is a similar phenomenon to strictly plurisubharmonicity.

\section{Singular $\omega$-trace positivity of holomorphic line bundles}

First, we introduce singular Hermitian metrics on %holomorphic 
line bundles and its positivity.

\begin{definition}$($\textnormal{cf.\,\cite{Dem93},\,\cite[Chapter\,3]{Dem10}}$)$ %$(\mathrm{cf.~[5],\,[7,\,Chapter\,3]})$ %[Dem93][Dem10]
    A $\it{singular}$ $\it{Hermitian}$ $\it{metric}$ $h$ on a line bundle $L$ is a metric which is given in any trivialization $\tau:L|_U\xrightarrow{\simeq} U\times\mathbb{C}$ by 
    \begin{align*}
        ||\xi||^2_h=|\tau(\xi)|^2e^{-\varphi}, \qquad x\in U,\,\,\xi\in L_x
    \end{align*}
    where $\varphi\in\mathcal{L}^1_{loc}(U)$ is an arbitrary function, called the weight of the metric with respect to the trivialization $\tau$.
\end{definition}

\begin{definition}\label{def of psef & big}$($\textnormal{cf.\,\cite[Definition\,3.2]{Wat23}}$)$
    Let $L$ be a holomorphic line bundle on a complex manifold $X$ equipped with a singular Hermitian metric $h$. % equipped with a singular Hermitian metric $h$.
    \begin{itemize}
        \item $h$ is $\it{singular}$ $\it{semi}$-$\it{positive}$ if $i\Theta_{L,h}\geq0$ in the sense of currents,
        i.e. the weight of $h$ with respect to any trivialization coincides with some plurisubharmonic function almost everywhere.
        \item $h$ is $\it{singular}$ $\it{positive}$ if the weight of $h$ with respect to any trivialization coincides with some strictly plurisubharmonic function almost everywhere.
        %\item [($c$)] Let $\omega$ be a \kah metric on $X$. Then $h$ is $\it{strictly}$ $\delta_\omega$-$\it{positive}$ if for any open subset $U$ and any \kah potential $\varphi$ of $\omega$ on $U$, $he^{\delta\varphi}$ is singular semi-positive.
    \end{itemize}
\end{definition}

Clearly, singular semi-positivity is coincides with pseudo-effective on compact complex manifolds. 
Singular positivity %and strictly $\delta_\omega$-positivity 
also coincides with big on compact \kah manifolds by Demailly's definition and characterization (see \cite{Dem93},\,\cite[Chapter\,6]{Dem10}). %, where $\omega$ is a \kah metric. %[Dem10] [7,\,Chapter\,6]

Using $\omega$-subharmonicity, we define $\omega$-trace positivity for singular Hermitian metrics.

\begin{definition}\label{def of tr-posi for line bdl case}
    Let $L$ be a holomorphic line bundle on a complex manifold $X$ and $\omega$ be a Hermitian metric on $X$. %equipped with a singular Hermitian metric $h$. % equipped with a singular Hermitian metric $h$.
    We say that a singular Hermitian metric $h$ on $L$ is 
    \begin{itemize}
        \item $\omega$-\textit{trace} \textit{semi}-\textit{positive} if the weight of $h$ with respect to any trivialization coincide with some $\omega$-subharmonic function almost everywhere,
        %\item $\omega$-\textit{trace} \textit{positive} if the weight of $h$ with respect to any trivialization coincide with some strictly $\omega$-subharmonic function almost everywhere,
        \item $\omega$-\textit{trace} \textit{quasi}-\textit{positive} if $h$ is $\omega$-trace semi-positive 
        and there exist at least one point $x_0\in X$ and an open neighborhood $U$ of $x_0$ such that the weight of $h$ with respect to any trivialization coincide with some strictly $\omega$-subharmonic function almost everywhere,
        \item $\omega$-\textit{trace} \textit{semi}-\textit{negative} (resp. %\textit{negative}, 
        \textit{quasi}-\textit{negative}) if the dual metric $h^*$ on $L^*$ is $\omega$-trace semi-positive (resp. %positive, 
        quasi-positive).
    \end{itemize}
\end{definition}

This $\omega$-trace quasi-positivity is a weaker condition than already known in \cite[Theorem\,1.5]{Yan19}, such that the dual line bundle is not pseudo-effective.

\begin{theorem}\label{characterization of tr-posi, not psef}
    Let $L$ be a holomorphic line bundle on a compact complex manifold $X$.
    If there exist a Hermitian metric $\omega$ and a singular Hermitian metric $h$ on $L$ such that $h$ is $\omega$-trace quasi-positive, 
    then the dual line bundle $L^*$ is not pseudo-effective.
\end{theorem}

\begin{proof}
    For any $x\in X$, there exists a neighborhood $U_x$ of $x$ such that the weight function $\varphi^x$ of $h$ with respect to a trivialization is $\omega$-subharmonic on $U_x$.
    And there exist $x_0\in X$ and a neighborhood $U_0$ of $x_0$ such that the weight function $\varphi^{x_0}$ of $h$ as above is strictly $\omega$-subharmonic on $U_0$
    %and $c>0$ such that $\idd\varphi^{x_0}\wedge\omega^{n-1}\geq c\omega^n$ in the sense of currents by strictly $\omega$-subharmonicity. 
    
    By Proposition \ref{smoothing of omega-subharmonic}, there exists a sequence of smooth $\omega$-subharmonic functions $(\varphi_\nu^x)_{\nu\in\mathbb{N}}$ decreasing to $\varphi^x$ a.e. on a subset $V_x$. Let $h_\nu^x:=e^{-\varphi_\nu^x}$ be a smooth $\omega$-trace semi-positive Hermitian metric on $L$ over $V_x$.
    Let $\omega_G=e^f\omega$ be a Gauduchon metric in the conformal class of $\omega$ \cite{Gau77}, then we obtain
    \begin{align*}
        tr_{\omega_G}i\Theta_{L,h^x_\nu}=e^{-f}tr_\omega i\Theta_{E,h^x_\nu}=e^{-f}tr_\omega \idd\varphi_\nu^x\geq0
    \end{align*}
    on $V_x$. Similar to above, by Proposition \ref{smoothing of strictly omega-subharmonic} there exists a sequence of smooth strictly $\omega$-subharmonic functions $(\varphi_\nu^{x_0})_{\nu\in\mathbb{N}}$ decreasing to $\varphi^{x_0}$ a.e. on a subset $V_{x_0}$ which satisfies a uniformly positive condition, 
    i.e. there exists $c>0$ such that for any $\nu\in\mathbb{N}$, $\Delta_\omega^C\varphi_\nu\geq c$ on $V_{x_0}$. 
    Then $h_\nu^{x_0}:=e^{-\varphi_\nu^x}$ is a smooth $\omega$-trace positive Hermitian metric on $L$ over $V_{x_0}$. 
    We also obtain on $V_{x_0}$ for any $\nu\in\mathbb{N}$,
    \begin{align*}
        tr_{\omega_G}i\Theta_{L,h^{x_0}_\nu}=e^{-f}tr_\omega i\Theta_{E,h^{x_0}_\nu}=e^{-f}tr_\omega \idd\varphi_\nu^{x_0}\geq ce^{-f}.
    \end{align*}
    %on $V_{x_0}$. %from the uniformly positive condition.

    By compact-ness of $X$, in addition to point $x_0$ there exists a finite points $x_1,\cdots,x_N$ such that the sets $\{V_{x_j}\}_{j\in I_N}$ is the open cover of $X$ where $I_N:=\{0,1,\cdots,N\}$, i.e. $\bigcup_{j\in I_N}V_{x_j}=X$.
    Let $\{\rho_j\}_{j\in I_N}$ be a partitions of unity subordinate to the open cover $\{V_{x_j}\}_{j\in I_N}$, i.e. $\rho_j\in\mathscr{D}(V_{x_j})_{\geq0}$ satisfies $\mathrm{supp}\,\rho_j\subset\subset V_{x_j}$,
    then we have that 
    \begin{align*}
        \int_X i\Theta_{L,h}\wedge\omega^{n-1}_G=\int_Xi\Theta_{L,h}\wedge\sum_{j\in I_N}\rho_j\omega^{n-1}_G=\sum_{j\in I_N}\int_{V_{x_j}}i\Theta_{L,h}\wedge\rho_j\omega^{n-1}_G.
    \end{align*}

    For any $j\in I_N\setminus\{0\}$, we get 
    \begin{align*}
        \int_{V_{x_j}}i\Theta_{L,h}\wedge\rho_j\omega^{n-1}_G&=\int_{V_{x_j}}\idd\varphi^{x_j}\wedge\rho_j\omega^{n-1}_G=\int_{V_{x_j}}\varphi^{x_j}\cdot\idd(\rho_j\omega^{n-1}_G)\\
        &=\int_{V_{x_j}}\lim_{\nu\to+\infty}\varphi^{x_j}_\nu\cdot\idd(\rho_j\omega^{n-1}_G)\\
        &\geq\limsup_{\nu\to+\infty}\int_{V_{x_j}}\varphi^{x_j}_\nu\cdot\idd(\rho_j\omega^{n-1}_G)\\
        &\geq0.
    \end{align*}
    In fact, each integrable is semi-positive by the following inequality,
    \begin{align*}
        \int_{V_{x_j}}\varphi^{x_j}_\nu\wedge\idd(\rho_j\omega^{n-1}_G)&=\int_{V_{x_j}}\idd\varphi^{x_j}_\nu\wedge\rho_j\omega^{n-1}_G=(n-1)!\int_{V_{x_j}}\rho_j\frac{n\idd\varphi^{x_j}_\nu\wedge\omega^{n-1}_G}{\omega^n_G}dV_{\omega_G}\\
        &=(n-1)!\int_{V_{x_j}}\rho_j\cdot tr_{\omega_G}\idd\varphi^{x_j}_\nu \,dV_{\omega_G}\\
        &=(n-1)!\int_{V_{x_j}}\rho_j\cdot e^{-f}\cdot tr_\omega\idd\varphi^{x_j}_\nu \,dV_{\omega_G}\\
        &\geq0.
    \end{align*}

    Moreover, in the $j=0$ case, similar to above we get 
    \begin{align*}
        \int_{V_{x_0}}i\Theta_{L,h}\wedge\rho_0\omega^{n-1}_G&\geq\liminf_{\nu\to+\infty}\int_{V_{x_0}}\varphi^{x_0}_\nu\wedge\idd(\rho_0\omega^{n-1}_G)\\
        &=(n-1)!\liminf_{\nu\to+\infty}\int_{V_{x_0}}\rho_0\cdot e^{-f}\cdot tr_{\omega}\idd\varphi^{x_0}_\nu \,dV_{\omega_G}\\
        &\geq c(n-1)!\int_{V_{x_0}}\rho_0\cdot e^{-f}\, dV_{\omega_G}\\
        &>0.
    \end{align*}

    Here, it is well-known that there exist a smooth Hermitian metric $h_0$ on $L$ and a real valued function $\psi\in\mathscr{L}^1(X,\mathbb{R})$ such that $i\Theta_{L,h}=i\Theta_{L,h_0}+\idd\psi$. 
    Hance, we have that 
    \begin{align*}
        \int_X c^{BC}_1(L)\wedge\omega^{n-1}_G&=\int_X i\Theta_{L,h_0}\wedge\omega^{n-1}_G=\int_X i\Theta_{L,h}\wedge\omega^{n-1}_G=\sum_{j\in I_N}\int_{V_{x_j}}i\Theta_{L,h}\wedge\rho_j\omega^{n-1}_G\\
        &\geq\int_{V_{x_0}}i\Theta_{L,h}\wedge\rho_0\omega^{n-1}_G
        \geq c(n-1)!\int_{V_{x_0}}\rho_0\cdot e^{-f}\, dV_{\omega_G}
        >0,
    \end{align*}
    by $\partial\overline{\partial}\omega^{n-1}_G=0$.
    By following lemma, the dual line bundle $L^*$ is not pseudo-effective.
\end{proof}

\begin{lemma}\label{characterization of not psef}$($\textnormal{cf.\,\cite[Proposition\,3.2]{Yan19}}$)$
    Let $L$ be a holomorphic line bundle on a compact complex manifold $X$.
    The following statements are equivalent.
    \begin{itemize}
        \item [$(a)$] the dual line bundle $L^*$ is not pseudo-effective,
        \item [$(b)$] there exists a Gauduchon metric $\omega_G$ such that 
        \begin{align*}
            \int_X c^{BC}_1(L)\wedge\omega^{n-1}_G>0.
        \end{align*}
    \end{itemize}
\end{lemma}

\begin{corollary}
    Let $L$ be a holomorphic line bundle on a compact complex manifold $X$.
    If there exist a Hermitian metric $\omega$ and a singular Hermitian metric $h$ on $L$ such that $h$ is $\omega$-trace quasi-positive $($resp. semi-positive$)$, 
    then we have that 
    \begin{align*}
        \int_Xc_1(L)\wedge\omega^{n-1}>0,\quad(\mathrm{resp}.\,\geq0).
    \end{align*}

    In particular, if $\omega$ is \kah then $\mathrm{deg}\,_\omega(L)>0$ $($resp. $\geq0)$.
\end{corollary}

From \cite[Theorem\,1.5]{Yan19}, we immediately obtain the following corollary.

\begin{corollary}
    Let $L$ be a holomorphic line bundle on a compact complex manifold $X$.
    The following statements are equivalent.
    \begin{itemize}
        \item there exist a Hermitian metric $\omega$ and a singular Hermitian metric $h$ on $L$ such that $h$ is $\omega$-trace quasi-positive, 
        %\item there exist a Hermitian metric $\omega$ and a smooth Hermitian metric $h$ on $L$ such that $h$ is $\omega$-trace positive, 
        \item $L$ is $(n-1)$-positive which means that there exists a smooth Hermitian metric $h$ on $L$ such that $i\Theta_{L,h}$ has at least one positive eigenvalue at any point on $X$,
        \item $L^*$ is not pseudo-effective. 
    \end{itemize}
\end{corollary}

Here, a holomorphic line bundle $L\to X$ is called $q$-\textit{positive} if there exists a smooth Hermitian metric $h$ on $L$ such that $i\Theta_{L,h}$ has at least $(n-q)$ positive eigenvalues at every point on $X$, where $\mathrm{dim}\,X=n$.

\section{Positivity of smooth Hermitian metric on holomorphic vector bundles}

In this section, we give characterizations of various positivity without using Chern curvature for smooth Hermitian metrics.

Let $\omega$ be a Hermitian metric on a complex manifold $X$ of dimension $n$. 
For any $x\in X$, there exists a local coordinates $(z_1,\cdots,z_n)$ centered $x$ such that $\omega=i\sum dz_j\wedge d\overline{z}_j+O(|z|)$. 
We say that this coordinates is the \textit{standard coordinates of} $\omega$ \textit{at} $x$.
Let $E$ be a holomorphic vector bundle and $h$ be a smooth Hermitian metric on $E$.
For any point $x\in X$ and any locally coordinates $z=(z_1,\ldots,z_n)$, we can written
\begin{align*}
    i\Theta_{E,h}=i\sum c_{jk\lambda\mu}dz_j\wedge d\overline{z}_k\otimes e^*_\lambda\otimes e_\mu,
\end{align*}
at $x$, where $e=(e_1,\ldots,e_r)$ is an orthonormal frame on a neighborhood. Therefore,
\begin{align*}
    \omega(x)=i\sum dz_j\wedge d\overline{z}_j,\,\,
    tr_\omega i\Theta_{E,h}(x)=\sum c_{jj\lambda\mu}e^*_\lambda\otimes e_\mu,\,\, h(x)=\sum e^*_\lambda\otimes e_\lambda,
\end{align*}
where $tr_\omega i\Theta_{E,h}\in\Gamma(X,\mathrm{End}(E))$.
We have that $tr_\omega i\Theta_{E,h}<0$ at $x$ if and only if there exists $c_x>0$ such that $tr_\omega i\Theta_{E,h}\leq-c_xh$ at $x$.

%Here, $tr_\omega i\Theta_{E,h}\in\Gamma(X,\mathrm{End}(E))$ and $tr_\omega i\Theta_{E,h}(x)\in\mathrm{End}(E_x)\cong E^*_x\otimes E_x$. 
%Using natural isomorphism $J_x:E_x\to E^*_x:e_\lambda\to e^*_\lambda$変える, the inequality $tr_\omega i\Theta_{E,h}\leq-c_xh$ at $x$ represent $tr_\omega i\Theta_{E,h}(v):=tr_\omega i\Theta_{E,h}(J_xv,v)\leq-c_x|v|^2_h$ for any $0\ne v\in E_x$.

\begin{definition}
    Under the above setting, we say that $h$ is 
    \begin{itemize}
        \item $\omega$-\textit{trace semi}-\textit{positive} on $X$ if $tr_\omega i\Theta_{E,h}\geq0$ at any point of $X$,
        \item $\omega$-\textit{trace semi}-\textit{negative} if the dual metric $h^*$ on $E^*$ is $\omega$-trace semi-positive,
    \end{itemize}
    where $tr_\omega i\Theta_{E,h}=-tr_\omega i\Theta_{E^*,h^*}$.
    
    Moreover, we define $\omega$-\textit{trace positivity} and \textit{negativity} in the similar manner.
\end{definition}

We consider characterizations of $\omega$-trace positivity without using curvature.

\begin{proposition}\label{characterization of tr-posi at point}
    Let $\omega$ be a Hermitian metric and $h$ be a smooth Hermitian metric on $E$. Then for any fiexd point $x\in X$, the following conditions are equivalent;
    \begin{itemize}
        \item [($a$)] $tr_\omega i\Theta_{E,h}<0$ at $x$, i.e. there exists $c>0$ such that $tr_\omega i\Theta_{E,h}\leq-ch$ at x,
        \item [($b$)] There exists $c>0$ such that for any $s\in\mathcal{O}(E)_x$, $\Delta^C_\omega|s|^2_h\geq c|s|^2_h$ at $x$, %for a standard coordinates of $\omega$ at $x$,
        \item [($c$)] There exists $c>0$ such that for any $s\in\mathcal{O}(E)_x$, $\Delta^C_\omega\log|s|^2_h\geq c$ at $x$. %for a standard coordinates of $\omega$ at $x$,
    \end{itemize}
    Moreover, the case $tr_\omega i\Theta_{E,h}\leq0$ at $x$ corresponds to $c=0$.
    %where $\Delta_z:=\Delta_{\idd|z|^2}=\sum\partial_{z_j}\partial_{\overline{z}_j}$ and $\Delta^C_\omega=\Delta_z$ at $x$ for standard coordinates of $\omega$ at $x$.
\end{proposition}

\begin{proof}
    Define $\Delta_{\idd|z|^2}=\sum\partial_{z_j}\partial_{\overline{z}_j}$ for a locally holomorphic coordinates $z=(z_1,\ldots,z_n)$.
    For any standard coordinates of $\omega$ at $x$, we have that $\Delta^C_\omega=\Delta_{\idd|z|^2}$ at $x$.
    $(a) \Longrightarrow (b)$. Let $z=(z_1,\cdots,z_n)$ be a standard coordinate of $\omega$ at centered $x$.
    From the equality $\idd|s|^2_h=\sum(D'^h_{z_j}s,D'^h_{z_k}s)_hdz_j\wedge d\overline{z}_k-\sum(\Theta^h_{jk}s,s)_hdz_j\wedge d\overline{z}_k$ for any $s\in\mathcal{O}(E)_x$, we have %that 
    \begin{align*}
        \Delta_{\idd|z|^2}|s|^2_h=tr_{\idd|z|^2}\idd|s|^2_h&=\sum||D'^h_{z_j}s||^2_h-\sum(\Theta^h_{jj}s,s)_h\\
        &=\sum||D'^h_{z_j}s||^2_h-tr_{\idd|z|^2}i\Theta_{E,h}(s,\overline{s})\\
        &\geq-tr_{\idd|z|^2}i\Theta_{E,h}(s,\overline{s})\geq c|s|^2_h.
    \end{align*}
    %where $tr_z:=tr_{\idd|z|^2}$.

    $(b) \Longrightarrow (a)$. For any $0\ne a\in E_x$, there exists $0\ne s\in\mathcal{O}(E)_x$ such that $s(x)=a$ and $D'^hs(x)=0$.
    By the assumption and the equality $\Delta_{\idd|z|^2}|s|^2_h=tr_{\idd|z|^2}\idd|s|^2_h=\sum||D'^h_{z_j}s||^2_h-\sum(\Theta^h_{jj}s,s)_h$, we have that 
    \begin{align*}
        c|a|^2_h\leq\Delta_{\idd|z|^2}|a|^2_h=\Delta_{\idd|z|^2}|s|^2_h|_{z=0}=-\sum(\Theta^h_{jj}a,a)_h=-tr_{\idd|z|^2}i\Theta_{E,h}(a,\overline{a})
    \end{align*}
    at $x$, i.e. $tr_\omega i\Theta_{E,h}\leq-ch$ at $x$.
    
    $(a) \Longrightarrow (c)$. Let $z=(z_1,\cdots,z_n)$ be a standard coordinate of $\omega$ at centered $x$.
    Here, for any $s\in\mathcal{O}(E)_x$ and any $\xi\in T_X$ we obtain 
    \begin{align*}
        \idd\log|s|^2_h(\xi)&=i(|s|^2_h|D'^hs\cdot \xi|^2_h-|\langle D'^hs\cdot \xi,s\rangle_h|^2)/|s|^4_h-\langle i\Theta_{E,h}(\xi,\overline{\xi})s,s\rangle_h/|s|^2_h\\
        &\geq-\langle i\Theta_{E,h}(\xi,\overline{\xi})s,s\rangle_h/|s|^2_h,
    \end{align*}
    by Chauchy-Schwarz. Then we have that
    \begin{align*}
        \Delta_{\idd|z|^2}\log|s|^2_h=tr_{\idd|z|^2}\idd\log|s|^2_h&=\sum\partial\overline{\partial}\log|s|^2_h(\partial/\partial z_j)\\
        &\geq-\sum\langle \Theta_{E,h}(\partial/\partial z_j,\partial/\partial \overline{z}_j)s,s\rangle_h/|s|^2_h\\
        &=-tr_{\idd|z|^2}i\Theta_{E,h}(s,\overline{s})/|s|^2_h\geq c_x.
    \end{align*}

    $(c) \Longrightarrow (a)$. This is shown similar to $(b) \Longrightarrow (a)$.
\end{proof}

\begin{proposition}\label{characterization of tr-posi on nbd}
    Let $\omega$ be a Hermitian metric and $h$ be a smooth Hermitian metric on $E$. Then the following conditions are equivalent;
    \begin{itemize}
        \item [($a$)] $tr_\omega\Theta_{E,h}<0$, %i.e. $\exists\,0<c\in\mathcal{C}^\infty(X,\mathbb{R}_{>0})$ such that $tr_\omega\Theta_{E,h}\leq -ch$.
        \item [($b$)] For any $x\in X$, there exist a neighborhood $U$\! of $x$, $c_x>0$ and a sequence of smooth Hermitian metrics $(h_\nu)_{\nu\in\mathbb{N}}$ on $U$ decreasing to $h$
        such that for any $\nu\in\mathbb{N}$ and any holomorphic section $s\!\in\! H^0(U,E)$, we have that $\Delta^C_\omega|s|^2_{h_\nu}\!\geq\! c_x|s|^2_{h_\nu}$ on $U$\!,
        \item [($c$)] For any $x\in X$, there exist a neighborhood $U$\! of $x$, $c_x>0$ and a sequence of smooth Hermitian metrics $(h_\nu)_{\nu\in\mathbb{N}}$ on $U$ decreasing to $h$ 
        such that for any $\nu\in\mathbb{N}$ and any holomorphic section $s\in\! H^0(U,E)$, we have that $\Delta^C_\omega\log|s|^2_{h_\nu}\!\geq\! c_x$ on $U$\!.
    \end{itemize}
\end{proposition}

\begin{proof}
    $(a) \Longrightarrow (b)$. For any $x\in X$, let $z=(z_1,\cdots,z_n)$ be a standard coordinate of $\omega$ at centered $x$.
    Then there exists $\widetilde{C}>0$ such that $tr_\omega i\Theta_{E,h}=tr_{\idd|z|^2}i\Theta_{E,h}\leq-\widetilde{C}h$ at $x$. By smooth-ness of $h$, there exist a neighborhood $U$ of $x$ and $C_x>0$ such that $tr_{\idd|z|^2}i\Theta_{E,h}\leq-C_xh$ on $U$.
    %From the equality $\idd|s|^2_h=\sum(D'^h_{z_j}s,D'^h_{z_k}s)_hdz_j\wedge d\overline{z}_k-\sum(\Theta^h_{jk}s,s)_hdz_j\wedge d\overline{z}_k$, 
    Similar to the proof of Proposition \ref{characterization of tr-posi at point}, for any $y\in U$ and any $s\in\mathcal{O}(E)_y$ we have that 
    \begin{align*}
        \Delta_{\idd|z|^2}|s|^2_h=tr_{\idd|z|^2}\idd|s|^2_h&=\sum||D'^h_{z_j}s||^2_h-\sum(\Theta^h_{jj}s,s)_h\\
        &=\sum||D'^h_{z_j}s||^2_h-tr_{\idd|z|^2}i\Theta_{E,h}(s,\overline{s})\\
        &\geq-tr_{\idd|z|^2}i\Theta_{E,h}(s,\overline{s})\geq C_x|s|^2_h.
    \end{align*}

    Let $h_\nu:=h\ast\rho_\nu$, where $(\rho_\nu)_{\nu\in\mathbb{N}}$ is an approximation to the identity. %Here, for any $y\in U$ and any $0\ne s\in\mathcal{O}(E)_y$, $\Delta|s|_h^2\geq c_x|s|^2_h$ on $U$ for a standard coordinate of $\omega$ at $x$.
    Then we have $\Delta_{\idd|z|^2}|s|^2_{h\ast\rho_\nu}\geq C_x|s|^2_{h\ast\rho_\nu}$ on $U_\nu=\{x\in U\mid d_{\mathbb{C}^n}(x,\partial U)>1/\nu\}$. In fact, by the equality
    \begin{align*}
        |s|^2_{h\ast\rho_\nu}=\int|s|^2_{h^{(w)}}(z)\rho_\nu(w)dV_w,
    \end{align*}
    where $h^{(w)}(z):=h(z-w)$, we have that 
    %for any $\phi\in\mathcal{D}(U_\varepsilon)_{\geq0}$, 
    \begin{align*}
        \Delta_{\idd|z|^2}|s|^2_{h\ast\rho_\nu}-C_x|s|^2_{h\ast\rho_\nu}&=(\Delta_{\idd|z|^2}-C_x)|s|^2_{h\ast\rho_\nu}\\
        &=\int(\Delta_{\idd|z|^2}-C_x)|s|^2_{h^{(w)}}(z)\rho_\nu(w)dV_w
        \geq0,
        %&=\int\Biggl(\sum\frac{\partial^2}{\partial z_j \partial \overline{z}_j}-\tilde{C}_x\Biggr)|s|^2_{h^(w)}(z)\rho_\varepsilon(w)dV_w\\
    \end{align*}
    where $(\Delta_{\idd|z|^2}-C_x)|s|^2_{h^{(w)}}(z)=(\Delta_{\idd|z|^2}-C_x)|s^{(-w)}|^2_h(z-w)\geq0$ by $s^{(w)}(z)=s(z-w)$ is also holomorphic.
    Similar to the proof of Proposition \ref{characterization of tr-posi at point} in $(b) \Longrightarrow (a)$, we obtain that $tr_{\idd|z|^2}i\Theta_{E,h_\nu}\leq-C_x h_\nu$ on $U$. 
    By smooth-ness of $h_\nu$ and $h$ and $tr_{\idd|z|^2}=tr_\omega+O(|z|)$ at $x$, there exist a neighborhood $W\subset\subset U$ of $x$ such that $tr_\omega i\Theta_{E,h}\leq-c_x h$ on $V$ and a neighborhood $V_\nu\subset\subset U$ of $x$ such that $tr_\omega i\Theta_{E,h_\nu}\leq-c_x h_\nu$ on $V_\nu$, where $c_x:=C_x/2$.
    Since $h_\nu$ smoothly converges to $h$ ($\nu\to+\infty$), $V_\nu$ approaches $W$. Then there exist $\nu_0\in\mathbb{N}$ and an open subset $V\subset W\cap\bigcap_{\nu\geq\nu_0}V_\nu\ne\emptyset$ such that 
    for any $y\in V$ and any $s\in\mathcal{O}(E)_y$ we have that $\Delta^C_\omega|s|^2_{h_\nu}\geq c_x|s|^2_{h_\nu}$ at $y$ from Proposition \ref{characterization of tr-posi at point}.
    %In fact, there exists a standard coordinates $z$ of $\omega$ centered $y$ such that $\Delta^C_\omega=\Delta_{\idd|z|^2}$ at $y$.

    Finally, we show decreasing-ness. 
    Here, for any constant section $\sigma\in H^0(U,E)$, we get $\Delta_{\idd|z|^2}|\sigma|^2_h\geq C_x|\sigma|^2_h$ for the standard coordinates $z$ of $\omega$ centered $x$. 
    By subharmonicity of $|\sigma|^2_h$, %and constant-ness of $\sigma$, 
    we have that $|\sigma|^2_h\ast\rho_\nu=|\sigma|^2_{h_\nu}$ decreases to $|\sigma|^2_h$, i.e. $h_\nu\searrow h$, if $\nu\to+\infty$.

    $(b) \Longrightarrow (a)$. This follows immediately from Proposition \ref{exists smoothing then (strictly) omega-SH as currents}.

    $(a) \iff (c)$. It follows as in the proof of $(a) \iff (b)$ using Proposition \ref{characterization of tr-posi at point}.
\end{proof}

We denote the  \textit{curvature operator} $[i\Theta_{E,h},\Lambda_\omega]$ on $\bigwedge^{p,q}T^*_X\otimes E$ by $A^{p,q}_{E,h,\omega}$.
The fact that the curvature operator $[i\Theta_{E,h},\Lambda_\omega]$ is positive (resp. semi-positive) definite on $\bigwedge^{p,q}T^*_X\otimes E$ is simply written as $A^{p,q}_{E,h,\omega}>0$ (resp. $\geq 0$).

\begin{proposition}\label{characterization of tr-posi and A^(n,n) posi}
    Let $h$ be a smooth Hermitian metric on $E$. For any Hermitian metric $\omega$, the following conditions are equivalent.
    \begin{itemize}
        \item $h$ is $\omega$-trace semi-positive, i.e. $tr_\omega i\Theta_{E,h}\geq0$,
        \item the curvature operator $A^{n,n}_{E,h,\omega}$ for $(n,n)$-forms is semi-positive, i.e. $A^{n,n}_{E,h,\omega}\geq0$.
    \end{itemize}
    The same equivalence holds for positivity.
\end{proposition}

\begin{proof}
    For any $x\in X$, 
    %there exists a locally coordinates $(z_1,\cdots,z_n)$ centered $x$ such that $\omega=i\sum dz_j\wedge d\overline{z}_j+O(|z|)$.
    %つまり, 
    there exists a standard coordinates $z=(z_1,\cdots,z_n)$ of $\omega$ at $x$. 
    Let $e=(e_1,\cdots,e_r)$ be an orthonormal frame of $E$ with respect to $h$ on a neighborhood of $x$. 
    Then we can write $i\Theta_{E,h}=i\sum c_{jk\lambda\mu}dz_j\wedge d\overline{z}_k\otimes e^*_\lambda\otimes e_\mu$,
    %\begin{align*}
    %    i\Theta_{E,h}=i\sum c_{jk\lambda\mu}dz_j\wedge d\overline{z}_k\otimes e^*_\lambda\otimes e_\mu,
    %\end{align*}
    at $x$, then we get $tr_\omega i\Theta_{E,h}=\sum c_{jj\lambda\mu}e^*_\lambda\otimes e_\mu$ at $x$. 
    Thus, $\omega$-trace semi-positivity of $h$ represents the condition that $tr_\omega i\Theta_{E,h}(s)=\sum c_{jj\lambda\mu}s_\lambda\overline{s}_\mu\geq0$ for any $s=\sum s_\lambda e_\lambda\in E_x$.

    In other hand, for any $(n,n)$-forms $u=\sum u_\lambda dz_N\wedge d\overline{z}_N\otimes e_\lambda\in\bigwedge^{n,n}T_{X,x}\otimes E_x$ where $dz_N=dz_1\wedge\cdots\wedge dz_n$, we have that 
    \begin{align*}
        A^{n,n}_{E,h,\omega}u&=i\Theta_{E,h}\wedge\Lambda_\omega u\\
        &=i\Theta_{E,h}\wedge i(-1)^n\sum u_\lambda\Bigl(\frac{\partial}{\partial z_s} \lrcorner\, dz_N\Bigr)\wedge\Bigl(\frac{\partial}{\partial \overline{z}_s} \lrcorner\, d\overline{z}_N\Bigr)\otimes e_\lambda\\
        &=\sum c_{jk\lambda\mu}u_\lambda dz_j\wedge\Bigl(\frac{\partial}{\partial z_s} \lrcorner\, dz_N\Bigr)\wedge d\overline{z}_k\wedge\Bigl(\frac{\partial}{\partial \overline{z}_s} \lrcorner\, d\overline{z}_N\Bigr)\otimes e_\lambda\\
        &=\sum c_{jj\lambda\mu}u_\lambda dz_N\wedge d\overline{z}_N\otimes e_\lambda.
    \end{align*}
    Therefore, we get $\langle A^{n,n}_{E,h,\omega}u,u\rangle_{h,\omega}=\sum c_{jj\lambda\mu}u_\lambda \overline{u}_\mu$.
\end{proof}

Let $(z_1,\cdots,z_n)$ be local coordinates, then Chern curvature tensor $\Theta_{E,h}$ is written as 
\begin{align*}
    \Theta_{E,h}=\sum_{1\leq j,k\leq n}\Theta^h_{jk}dz_j\wedge d\overline{z}_k,
\end{align*}
where the coefficients are written as $\Theta^h_{jk}=[D'^h_{z_j},\overline{\partial}_{z_k}]$ and $\overline{\partial}_{z_j}=\partial/\partial\overline{z}_j$.

\begin{definition}$(\mathrm{cf.}$~\cite[ChapterVII]{Dem-book}$)$
    Let $E$ be a holomorphic vector bundle on a complex manifold $X$.
    We say that a smooth Hermitian metric $h$ on $E$ is \textit{Griffiths positive} (resp. \textit{negative}) if for any $x\in X$, $0\ne\xi\in T_{X,x}$ and any $0\ne s\in E_x$, we have
        \begin{align*}
            \Theta_{E,h}(u\otimes\xi)=\sum(\Theta^h_{jk}u,u)_h\xi_j\overline{\xi}_k>0 \quad (\mathrm{resp}. <0).
        \end{align*}
\end{definition}

Let $h$ be a smooth Hermitian metric on $E$ and $u=(u_1,\cdots,u_n)$ be an $n$-tuple of locally holomorphic sections of $E$. We define $T^h_u$, an $(n-1,n-1)$-form through
\begin{align*}
    T^h_u=\sum(u_j,u_k)_h\widehat{dz_j\wedge d\overline{z}_k}
\end{align*}
where $(z_1,\cdots,z_n)$ are local coordinates on $X$ and $\widehat{dz_j\wedge d\overline{z}_k}$ %denotes the differential form 
satisfying $idz_j\wedge d\overline{z}_k\wedge\widehat{dz_j\wedge d\overline{z}_k}=dV_{\mathbb{C}^n}$.
%denotes the wedge product of all $dz_i$ and $d\overline{z}_i$ expect $dz_j$ and $d\overline{z}_k$, multiplied by a constant of absolute value $1$, chosen so that $T_u^h$ is a positive form. 
Then a short computation yields that %$(E,h)$ is Nakano semi-negative if and only if $T^h_u$ is plurisubharmonic in the sense that $\idd T^h_u\geq0$ (see \cite{Ber09,Rau15}). %[1,\,33]). %\cite{Ber09},\,\cite{Rau15}).
$h$ is Griffiths semi-negative if and only if $T^h_{\xi u}$ is plurisubharmonic in the sense that $\idd T^h_{\xi u}\geq0$ where $u_j=\xi_ju$ for any $\xi\in\mathbb{C}^n$ and any $u\in\mathcal{O}(E)_x$ (see \cite{Ber09,Rau15}). %[1,\,33]). %\cite{Ber09},\,\cite{Rau15}).
%In the case of $u_j=\xi_ju$ for any $\xi\in\mathbb{C}^n$, %$=(\xi_1,\cdots,\xi_n)\in\mathbb{C}^n$,      
In fact, if $D'^hu=0$ at $x$ then we get 
\begin{align*}
   \idd T^h_{\xi u}=-\sum(\Theta^h_{jk}u,u)_h\xi_j\overline{\xi}_kdV_{\mathbb{C}^n},
\end{align*}
at $x$ by the following equations that for any $u,v\in\mathcal{O}(E)$ 
\begin{align*}
    \frac{\partial^2}{\partial z_j\partial \overline{z}_k}(u,v)_h&=(D'^h_{z_j}u,D'^h_{z_k}v)_h-(\Theta^h_{jk}u,v)_h,\\
    \idd T^h_u%&=\sum(D'^h_{z_j}u_j,D'^h_{z_k}u_k)_hdV_{\mathbb{C}^n}-\sum(\Theta^h_{jk}u_j,u_k)_hdV_{\mathbb{C}^n}\\
    &=||\sum D'^h_{z_j}u_j||^2_h-\sum(\Theta^h_{jk}u_j,u_k)_hdV_{\mathbb{C}^n}.
\end{align*}

We already know the following characterization.

\begin{proposition}\label{characterization of Grif posi for smooth}
    For a smooth Hermitian metric $h$ on $E$ and any fiexd point $x\in X$, the following conditions are equivalent.
    \begin{itemize}
        \item $h$ is Griffiths semi-negative $($resp. negative$)$ at $x$, %for any $0\ne\xi\in \mathbb{C}^n$ and any $0\ne s\in E_x$, we have that $\Theta_{E,h}(s\otimes\xi)<0$ $($resp. $\leq0)$,
        \item $\log|u|^2_h$ is $($resp. strictly$)$ plurisubharmonic for any $0\ne u\in\mathcal{O}(E)_x$,
        \item there exists $\delta>0$ such that 
        \begin{align*}
            \idd|u|^2_h\geq\delta|u|^2_h\idd|z|^2 \quad (\mathit{resp}.\, \geq0)
        \end{align*}
        for any $0\ne u\in\mathcal{O}(E)_x$, where $(z_1,\cdots,z_n)$ is local coordinates.
    \end{itemize}
\end{proposition}

In other words, the case $\delta=0$ is simply $|u|^2_h$ is plurisubharmonic. 
The inequality $\idd|u|^2_h\geq\delta|u|^2_h\idd|z|^2$ can be rewritten as $\idd T^h_{\xi u}\geq\delta|u|^2_hdV_{\mathbb{C}^n}$ for any $0\ne\xi\in\mathbb{C}^n$.
%\begin{align*}
%    \idd T^h_{\xi u}\geq\delta|u|^2_hdV_{\mathbb{C}^n},
%\end{align*}

\section{Definitions and relationships for singular $\omega$-trace and Griffiths positivity}
\subsection{Definition of singular Griffiths positivity}

In this subsection, we introduce the definition of singular Hermitian metrics and singular Griffiths positivity for holomorphic vector bundles.
%the $L^2$-subsheaf $\mathscr{E}(h)$ of $\mathcal{O}(E)$ analogous to the multiplier ideal sheaf.
Notions of singular Hermitian metrics for holomorphic vector bundles were introduced and investigated (cf. \cite{BP08,deC98}). 
However, it is known that we cannot always define the curvature currents with measure coefficients (see \cite{Rau15}). 
Thus, Griffiths semi-negativity or semi-positivity (see \cite{BP08,Rau15}) is defined without using the curvature currents by using plurisubharmonicity.

\begin{definition}\label{I def of sHm on vect bdl}$($\textnormal{cf.\,\cite[Definition\,2.2.1]{PT18}}$)$ 
    %$(\mathrm{cf.}$ \cite[Definition\,2.2.1]{PT18}, \cite[Definition\,1.1]{Rau15}$)$ %~[2,~Section~3],~~[31,~Definition,~2.2.1]~and~[33,~Definition~1.1]} %[BP08][PT18][Rau15]
    We say that $h$ is a \textit{singular Hermitian metric} on $E$ if $h$ is a measurable map from the base manifold $X$ to the space of non-negative Hermitian forms on the fibers satisfying $0<\mathrm{det}\,h<+\infty$ almost everywhere.
\end{definition}

We already know the following definition and characterization for singular Griffiths semi-positivity.

\begin{definition}\label{def Griffiths semi-posi sing}$($\textnormal{cf.\,\cite[Definition\,2.2.2]{PT18},\,\cite[Definition\,1.2]{Rau15}}$)$ %\cite[Definition~3.1]{BP08},\,
    %\cite[Definition~2.2.2]{PT18},\,\cite[Definition~1.2]{Rau15}$)$ %[BP08][PT18][Rau15]
    We say that a singular Hermitian metric $h$ on $E$ is 
    \begin{itemize}
        \item \textit{Griffiths semi-negative} if $\log|u|_h$ is plurisubharmonic for any local holomorphic section $u\in\mathcal{O}(E)$ of $E$.
        \item \textit{Griffiths semi-positive} if the dual metric $h^*$ on $E^*$ is Griffiths semi-negative.
    \end{itemize}
\end{definition}

\begin{proposition}\label{characterization of sing Grif semi-posi}$($\textnormal{cf.\,\cite[Proposition\,3.1]{BP08}}$)$
    Let $h$ be a singular Hermitian metric on $E$. Then the following conditions are equivalent.
    \begin{itemize}
        \item [($a$)] $h$ is Griffiths semi-negative, i.e. $\log|u|_h$ is plurisubharmonic for any local holomorphic section $u\in\mathcal{O}(E)$,
        \item [($b$)] $|u|^2_h$ is plurisubharmonic for any local holomorphic section $u\in\mathcal{O}(E)$,
        \item [($c$)] $T^h_{\xi u}$ is plurisubharmonic for any local holomorphic section $u\in\mathcal{O}(E)$ and any $\xi\in\mathbb{C}^n$, where $\xi u=(u_1,\cdots,u_n)$ and $u_j=\xi_ju$,
        \item [($d$)] for any Stein subset $S$ which satisfies $E|_S$ is trivial, there exists a sequence of smooth Griffiths semi-negative Hermitian metrics $(h_\nu)_{\nu\in\mathbb{N}}$ decreasing to $h$ pointwise a.e. on any relatively compact Stein subset of $S$, where $h_\nu:=h\ast\rho_\nu$.
    \end{itemize} 
    Here, $(\rho_\nu)_{\nu\in\mathbb{N}}$ is an approximate identity on $S$.
\end{proposition}

Recently, this triviality of $E$ within the condition $(d)$ was omitted in \cite{DNWZ23}.
Using Proposition \ref{characterization of Grif posi for smooth}, we introduce a definition of Griffiths positivity with respect to singular Hermitian metric.

\begin{definition}\label{def of strictly Grif posi as sing}
    Let $X$ be a complex manifold and $E$ be a holomorphic vector bundle on $X$. % equipped with a singular Hermitian metric $h$. 
    We say that a singular Hermitian metric $h$ on $E$ is 
    \begin{itemize}
        \item \textit{Griffiths negative at point} $x\in X$ if there exist an open neighborhood $U$ of $x$ and $\delta>0$ such that for any local holomorphic section $u\in H^0(U,E)$, we have %$\idd|u|^2_h\geq\delta|u|^2_h\idd|z|^2$ 
        \begin{align*}
            \idd|u|^2_h\geq\delta|u|^2_h\idd|z|^2
        \end{align*}
        on $U$ in the sense of currents, where $(z_1,\cdots,z_n)$ is local coordinates of $U$,
        \item \textit{Griffiths negative} if $h$ is Griffiths negative at any point $x\in X$,
        \item \textit{Griffiths positive} (resp. \textit{at point} $x\in X$) if the dual metric $h^*$ on $E^*$ is Griffiths negative (resp. at point $x\in X$).
    \end{itemize}
\end{definition}

There is a definition close to Definition \ref{def of strictly Grif posi as sing}, \cite[Definition\,6.1]{Rau15}, 
which assumes that $F=\{z\mid\mathrm{det}\,h(z)=0\}$ is a closed set and that there exists an exhaustion of open sets $\{U_j\}_{j\in\mathbb{N}}$ of $F^c$ such that $\mathrm{det}\,h>1/j$ on $U_j$.
In this paper, we adopt Definition \ref{def Griffiths semi-posi sing} as a more general setting and show that this definition induces singular positivity for the tautological line bundle.
We obtain an equivalence between positivity and existence of the approximation, as in Proposition \ref{characterization of sing Grif semi-posi}.

\begin{proposition}\label{characterization of Grif nega by smoothing}
    Let $h$ be a singular Hermitian metric on $E$. Then the following conditions are equivalent.
    \begin{itemize}
        \item [($a$)] $h$ is Griffiths negative, %i.e. $\log|u|_h$ is plurisubharmonic for any local holomorphic section $u\in\mathcal{O}(E)$,
        %\item [($b$)] $|u|^2_h$ is plurisubharmonic for any local holomorphic section $u\in\mathcal{O}(E)$,
        \item [($b$)] for any $x\in X$ there exist an open neighborhood $U$ of $x$ and $\delta>0$ such that for any local holomorphic section $u\in H^0(U,E)$ and any $\xi\in\mathbb{C}^n$, 
        \begin{align*}
            \idd T^h_{\xi u}\geq\delta|u|^2_hdV_{\mathbb{C}^n},
        \end{align*}
        where $\xi u=(u_1,\cdots,u_n)$ and $u_j=\xi_ju$,
        \item [($c$)] for any Stein subset $S$ which satisfies $E|_S$ is trivial, 
        there exists a sequence of smooth Griffiths negative Hermitian metrics $(h_\nu)_{\nu\in\mathbb{N}}$ on $S$ % any relatively compact Stein subset $\widetilde{S}$ of $S$ %where $h_\nu:=h\ast\rho_\nu$.
        %which satisfies that this sequence 
        decreasing to $h$ pointwise a.e. %and %satisfies %the uniformity condition 
        which satisfies there exists $\delta>0$ such that 
        \begin{align*}
            %\Theta_{E,h_\nu}(u\otimes \xi)&=\sum(\Theta^{h_\nu}_{jk}u,u)_{h_\nu}\xi_j\overline{\xi}_k\leq-\delta|u|^2_{h_\nu}|\xi|^2, \\
            \quad \idd|u|^2_{h_\nu}\geq\delta|u|^2_{h_\nu}\idd|z|^2,  \quad
            i.e. \quad \idd T^{h_\nu}_{\xi u}\geq\delta|u|^2_{h_\nu}dV_{\mathbb{C}^n},
        \end{align*}
        for any $\nu\in\mathbb{N}$, any $x\in\widetilde{S}$, any $0\ne\xi\in T_{X,x}$ and any $0\ne u\in E_x$.
    \end{itemize}
    Furthermore, if $h$ is Griffiths negative then the following condition $(d)$ holds, 
    \begin{itemize}
        \item [($d$)] $\log|u|_h$ is strictly plurisubharmonic for any local holomorphic section $u\in\mathcal{O}(E)$.
    \end{itemize}
\end{proposition}

\begin{proof}
    $(a) \iff (b)$ is trivial and just rewording.

    $(c) \Longrightarrow (b)$. It is sufficient to show that for any fiexd test function $\phi\in\mathscr{D}(U)_{\geq0}$, 
    \begin{align*}
        0\leq(\idd T^h_{\xi u}-\delta|u|^2_hdV_{\mathbb{C}^n})(\phi)&=\int_U T^h_{\xi u}(z)\wedge\idd\phi(z)-\delta|u|^2_h(z)\phi(z)dV_{\mathbb{C}^n}\\
        &=\int_U|u|^2_h(z)\Bigl(\sum \xi_j\overline{\xi}_k\phi_{jk}(z)-\delta\phi(z)\Bigr)dV_{\mathbb{C}^n},
    \end{align*}
    where $\phi_{jk}=\frac{\partial^2\phi}{\partial z_j\partial \overline{z}_k}$,
    for any local section $0\ne u\in H^0(U,E)$ and any $0\ne\xi\in\mathbb{C}^n$. 
    By the assumption, $\idd T^{h_\nu}_{\xi u}\geq\delta|u|^2_{h_\nu}dV_{\mathbb{C}^n}\geq\delta|u|^2_hdV_{\mathbb{C}^n}$ on $U$ in the sense of currents, i.e. 
    \begin{align*}
        \delta\int_U|u|^2_h(z)\phi(z)dV_{\mathbb{C}^n}\leq\int_U|u|^2_{h_\nu}(z)\sum \xi_j\overline{\xi}_k\phi_{jk}(z)dV_{\mathbb{C}^n}.
    \end{align*}

    Let $K:=\{x\in U\mid \lim_{\nu\to+\infty}|u|^2_{h_\nu}(x)=|u|^2_h(x)\}$, then the Lebesgue measure of $U\setminus K$ is zero by $(h_\nu)_{\nu\in\mathbb{N}}$ decreasing to $h$ pointwise a.e.
    By Fatou's lemma, we have that 
    \begin{align*}
        \delta\int_U|u|^2_h\phi\, dV_{\mathbb{C}^n}&\leq\limsup_{\nu\to+\infty}\int_U|u|^2_{h_\nu}(z)\sum \xi_j\overline{\xi}_k\phi_{jk}(z)dV_{\mathbb{C}^n}\\
        &\leq\int_U\limsup_{\nu\to+\infty}|u|^2_{h_\nu}(z)\sum \xi_j\overline{\xi}_k\phi_{jk}(z)dV_{\mathbb{C}^n}
        \!=\!\!\int_K\lim_{\nu\to+\infty}|u|^2_{h_\nu}(z)\sum \xi_j\overline{\xi}_k\phi_{jk}(z)dV_{\mathbb{C}^n}\\
        &=\int_K|u|^2_h(z)\sum \xi_j\overline{\xi}_k\phi_{jk}(z)dV_{\mathbb{C}^n}
        =\int_U|u|^2_h(z)\sum \xi_j\overline{\xi}_k\phi_{jk}(z)dV_{\mathbb{C}^n}.
    \end{align*}

    $(b) \Longrightarrow (c)$. This is already known in \cite[Proposition\,6.2]{Rau15} and can construct the sequence $(h_\nu)_{\nu\in\mathbb{N}}$ by defining $h_\nu:=h\ast\rho_\nu$, where $(\rho_\nu)$ is an approximate identity on $S$.
    Here, for any constant section $\sigma\in H^0(U,E)$, we have that $|\sigma|^2_{h_\nu}=|\sigma|^2_h\ast\rho_\nu$ decreasing to $|\sigma|^2_h$ pointwise a.e. %, i.e. $h_\nu$ decreasing to $h$ pointwise a.e., 
    by plurisubharmonicity of $|\sigma|^2_h$.

    $(c) \Longrightarrow (d)$. By the assumption, we get $-\langle i\Theta_{E,h_\nu}a,a\rangle_{h_\nu}\geq\delta|a|^2_{h_\nu}\idd|z|^2$ for any point $x\in X$ and any $a\in E_x$.
    Therefore, we have that 
    \begin{align*}
        \idd\log|u|^2_{h_\nu}&=i\frac{|s|^2_{h_\nu}|D'^{h_\nu}s|^2_{h_\nu}-\langle D'^{h_\nu}s,s\rangle_{h_\nu}\wedge\langle s,D'^{h_\nu}s\rangle_{h_\nu}}{|s|^4_{h_\nu}}-\frac{\langle i\Theta_{E,h_\nu}s,s\rangle_{h_\nu}}{|s|^2_{h_\nu}}\\
        &\geq-\frac{\langle i\Theta_{E,h_\nu}s,s\rangle_{h_\nu}}{|s|^2_{h_\nu}}
        \geq\delta\idd|z|^2,
    \end{align*}
    by Chauchy-Schwarz inequality. 
    Hance, $\log|s|^2_{h_\nu}-\delta|z|^2$ are smooth plurisubharmonic functions decreasing to $\log|s|^2_h-\delta|z|^2$, which is also plurisubharmonic. 
\end{proof}

\begin{remark}
    In general, singular Griffiths negativity does not hold from condition $(d)$. 
    However, the equivalence holds in the case of holomorphic line bundles.
\end{remark}

In fact, even if $\varphi$ is strictly plurisubharmonic, $e^\varphi$ is not necessarily strictly plurisubharmonic.
For example, if we take $\log|z_1|+|z|^2$ as a strictly plurisubharmonic function on $\mathbb{C}^n$ then $|z_1|e^{|z|^2}$ is not strictly plurisubharmonic.

Finally, we introduce the following definition of Griffiths quasi-positivity.

\begin{definition}\label{def of Grif quai-posi as sing}
    Let $X$ be a complex manifold and $E$ be a holomorphic vector bundle on $X$. % equipped with a singular Hermitian metric $h$. 
    We say that a singular Hermitian metric $h$ on $E$ is 
    \begin{itemize}
        \item \textit{Griffiths quasi-negative} if $h$ is Griffiths semi-negative on $X$ and there exists at least one point $x_0\in X$ such that $h$ is Griffiths negative at $x_0$,
        \item \textit{Griffiths quasi-positive} if the dual metric $h^*$ on $E^*$ is Griffiths quasi-negative.
    \end{itemize}
\end{definition}

\subsection{Definition of singular $\omega$-trace positivity and its relationship to singular Griffiths positivity}

%In this subsection, 
We define $\omega$-trace positivity for singular Hermitian metric using Proposition \ref{characterization of tr-posi at point} and \ref{characterization of tr-posi on nbd}.
Moreover, we establish the relation to Griffiths positivity.

\begin{definition}
    Let $\omega$ be a Hermitian metric and $E$ be a holomorphic vector bundle on $X$. 
    We say that a singular Hermitian metric $h$ on $E$ is 
    \begin{itemize}
        \item $\omega$-\textit{trace semi}-\textit{negative} if for any $x\in X$ and any local holomorphic section $s\in\mathcal{O}(E)_x$, $\log|s|^2_h$ is $\omega$-subharmonic on a neighborhood of $x$,
        \item $\omega$-\textit{trace semi}-\textit{positive} if the dual metric $h^*$ on $E^*$ is $\omega$-trace semi-negative.
    \end{itemize}
\end{definition}

This definition is clearly consistent with the case of line bundles, i.e. Definition \ref{def of tr-posi for line bdl case}.

\begin{definition}
    Let $\omega$ be a Hermitian metric and $E$ be a holomorphic vector bundle on $X$. 
    We say that a singular Hermitian metric $h$ on $E$ is 
    \begin{itemize}
        \item $\omega$-\textit{trace} \textit{negative at point} $x\in X$ if there exist a neighborhood $U$ of $x$, $c_x>0$ and a sequence of smooth Hermitian metrics $(h_\nu)_{\nu\in\mathbb{N}}$ on $U$ decreasing to $h$ pointwise
        such that for any $\nu\in\mathbb{N}$ and any local holomorphic section $s\in H^0(U,E)$, we have that $\Delta^C_\omega\log|s|^2_{h_\nu}\geq c_x$ on $U$,
        \item $\omega$-\textit{trace} \textit{negative} if $h$ is $\omega$-trace negative at any point $x\in X$,
        \item $\omega$-\textit{trace} \textit{positive} (resp. \textit{at point} $x\in X$) if the dual metric $h^*$ on $E^*$ is $\omega$-trace negative (resp. at point $x\in X$).
    \end{itemize}
\end{definition}

By Theorem \ref{omega-SH for any omega then PSH} and Proposition \ref{characterization of Grif nega by smoothing}, we have the following.

\begin{theorem}\label{characterization of sing Grif semi-posi and tr-omega semi-posi} 
    Let $X$ be a complex manifold and $E$ be a holomorphic vector bundle on $X$. %equipped with a singular Hermitian metric $h$.
    Then the following conditions are equivalent
    \begin{itemize}
        \item [$(a)$] $E$ has a singular Hermitian metric $h_G$ with Griffiths semi-positivity.
        \item [$(b)$] $E$ has a singular Hermitian metric $h_{tr}$ which %such that $h_{tr}$
         is $\omega$-trace semi-positive $($with approximation$)$ for any Hermitian metric $\omega$.
    \end{itemize}
    In particular, if $(a)$ holds then we can take $h_G=h_{tr}$ in $(b)$, and if $(b)$ holds then $h_G$ in $(a)$ can be taken to coincide with $h_{tr}$ almost everywhere.
\end{theorem}

\begin{proof}
    $(a) \Longrightarrow (b)$. This follows immediately by Proposition \ref{psh then omega-SH}.

    $(b) \Longrightarrow (a)$. Let $h:=h_{tr}$ for short. By the assumption and Theorem \ref{omega-SH for any omega then PSH}, we obtain $\idd\log|u|^2_{h^*}\geq0$ in the sense of currents on $U$ for any local holomorphic section $u\in H^0(U,E^*)$. 
    From the following lemma, we get $\idd|u|^2_{h^*}\geq0$ in the sense of currents on $U$.
    In other words, $\idd T^{h^*}_{\xi u}\geq0$ in the sense of currents on $U$ for any $0\ne\xi\in\mathbb{C}^n$, i.e. 
    \begin{align*}
        0\leq\idd T^{h^*}_{\xi u}(\phi)=\int_U T^{h^*}_{\xi u}\wedge\idd\phi
    \end{align*}
    for any test function $\phi\in\mathscr{D}(U)_{\geq0}$. 
    Here, we get 
    \begin{align*}
        |u|^2_{h^*_\nu}(z)=\int|u|^2_{h^*_{(w)}}(z)\rho_\nu(w)dV_w, \quad T^{h^*_\nu}_{\xi u}(z)=\int T^{h^*_{(w)}}_{\xi u}(z)\rho_\nu(w)dV_w,
    \end{align*}
    where $h^*_{(w)}(z):=h^*(z-w)$. We have that 
    \begin{align*}
        \idd T^{h_\nu^*}_{\xi u}(\phi)&=\int_U T^{h_\nu^*}_{\xi u}(z)\wedge\idd\phi(z)
        =\int_z\Bigl\{\int_wT^{h^*_{(w)}}_{\xi u}(z)\rho_\nu(w)dV_w\Bigr\}\wedge\idd\phi(z)\\
        &=\int_{w\in\mathrm{supp}\,\rho_\nu} \rho_\nu(w)\Bigl\{\int_z T^{h^*_{(w)}}_{\xi u}(z)\wedge\idd\phi(z)\Bigr\}dV_w\\
        &\geq0,
    \end{align*}
    where the translation of $u$ is also holomorphic. Hance, $h^*_\nu$ is smooth Griffiths semi-positive and decreasing to $h$ pointwise a.e.
    
    Define the singular Hermitian metric $h^*_G$ by $h^*_G(z):=\lim_{\nu\to+\infty}h^*_\nu(z)$. %, i.e. $h^*_G$ is the limit of convergence pointwise.
    In detail, $h^*_G$ is defined as follows. Let $h^*_G(z):=(H^*_{jk}(z))$ and $h^*(z):=(h^*_{jk}(z))$ be local matrix representations for the same holomorphic frame. 
    For constant vectors $v_1={}^t(1,0,\cdots,0),\cdots,v_r$ $={}^t(0,\cdots,0,1)$, define $H^*_{jj}$ as the decreasing limit of $|v_j|^2_{h^*_\nu}$ converging pointwise, i.e. $H^*_{jj}(z):=\lim_{\nu\to+\infty}|v_j|^2_{h^*_\nu}(z)$, then $H^*_{jj}$ is a function with plurisubharmonicity and coincides with $h^*_{jj}$ a.e. % almost everywhere.
    For constant vectors $v={}^t(1,1,\cdots,0)$ and $v'={}^t(1,i,\cdots,0)$, we have that $\lim_{\nu\to+\infty}|v|^2_{h^*_\nu}:=|v|^2_{h^*_G}=H^*_{11}+H^*_{22}+2\mathrm{Re}H^*_{12}$ and $\lim_{\nu\to+\infty}|v'|^2_{h^*_\nu}:=|v'|^2_{h^*_G}=H^*_{11}+H^*_{22}+2\mathrm{Im}H^*_{12}.$
    %\begin{align*}
    %    \lim_{\nu\to+\infty}|v|^2_{h^*_\nu}:=|v|^2_{h^*_G}=H^*_{11}+H^*_{22}+2\mathrm{Re}H^*_{12}, \quad \lim_{\nu\to+\infty}|v'|^2_{h^*_\nu}:=|v'|^2_{h^*_G}=H^*_{11}+H^*_{22}+2\mathrm{Im}H^*_{12}.
    %\end{align*} 
    From this, $H^*_{12}$ can be defined as a function and coincides with $h^*_{12}$ a.e. By finitely repeating this process, $h^*_G$ is defined and coincides with $h^*$ a.e.

    Since %$\log|u|^2_{h^*_\nu}$ is plurisubharmonic and
    $\log|u|^2_{h^*_G}=\lim_{\nu\to+\infty}\log|u|^2_{h^*_\nu}$ is also plurisubharmonic by decreasing-ness, we have that %$h=h_G$ almost everywhere and 
    $h_G$ is Griffiths semi-positive.
\end{proof}

\begin{lemma}
    Let $\varphi\in\mathscr{L}^1_{loc}(U)$ be a locally integrable function on $U$. If $\varphi$ satisfies $\idd\varphi\geq0$ in the sense of currents on $U$, then $\idd\chi\circ\varphi\geq0$ in the sense of currents on $U$ for any smooth increasing convex function $\chi$.
\end{lemma}

\begin{proof}
    The assumption is equivalent to that $\varphi_\nu:=\varphi\ast\rho_\nu$ is plurisubharmonic for any $\nu>0$. %, where $(\rho_\nu)_{\nu\in\mathbb{N}}$ is an approximate identity.
    Therefore, we show that $(\chi\circ\varphi)_\varepsilon:=(\chi\circ\varphi)\ast\rho_\varepsilon$ is plurisubharmonic for any $\varepsilon>0$. %を示せばよい. 
    By smooth plurisubharmonicity of $\chi\circ\varphi_\nu$, the functions $(\chi\circ\varphi)^\nu_\varepsilon:=(\chi\circ\varphi_\nu)\ast\rho_\varepsilon$ are also smooth plurisubharmonic and decreasing to $(\chi\circ\varphi)_\varepsilon$ pointwise.
    Hance, this function $(\chi\circ\varphi)_\varepsilon$ is plurisubharmonic.
\end{proof}

\begin{corollary}
    If $E$ has a singular Hermitian metric $h$ with Griffiths positivity, then there exists a singular Hermitian metric on $E$ which is $\omega$-trace positive for any Hermitian metric $\omega$ and coincides with $h$ almost everywhere.  
    %For a singular Hermitian metric $h$ on $E$, if $h$ is Griffiths positive then $h$ is $\omega$-trace positive for any Hermitian metric $\omega$.
\end{corollary}

\begin{remark}
    For singular Hermitian metrics $h$ with a.e. Griffiths semi-positivity $($cf. \cite[Def\,\,2.2.2]{PT18}$)$, the proof of Theorem \ref{characterization of sing Grif semi-posi and tr-omega semi-posi} gives a construction of a singular Hermitian metric which is Griffiths semi-positive and coincides with $h$ almost everywhere. 
\end{remark}

Using Littman's approximation method, we obtain the following proposition.

\begin{proposition}\label{smoothing of tr-omega semi-posi for constant section}
    Let $\omega$ be a Hermitian metric and $E$ be a holomorphic vector bundle equipped with a singular Hermitian metric $h$. 
    If $h$ is $\omega$-trace semi-negative then for any open subset $U$ contained in a local chart, 
    there exists a sequence of smooth Hermitian metrics $(h_\nu)_{\nu\in\mathbb{N}}$ decreasing to $h$ pointwise a.e. on every compact subset $U'\subset U$ such that 
    for any $x\in U'$ and any local constant section $\sigma$ of $E$, we have that $\Delta^C_\omega|\sigma|^2_{h_\nu}\geq0$.
\end{proposition}

\begin{proof}
    Let $(\Delta^C_\omega)^\dagger$ be the formal adjoint of $\Delta^C_\omega$ on $U$ with respect to the inner product $\int_U\Delta^C_\omega u(z)\cdot v(z)dV_{U,z}$ for any $u,v\in\mathcal{C}^\infty(U)$.
    For the Green's function $g(\zeta,z)$ of the canonical Laplacian operator $\Delta^C_\omega$ with respect to the domain $U$, we let $G_\nu(\zeta,z)$ be the function constructed in \cite{Lit63} which satisfy $(\Delta^C_\omega)_zG_\nu(\zeta,z)\geq0$. 
    Therefore, we have %that 
    \begin{align*}
        F_{G_\nu,\phi}(\zeta):=\int_UG_\nu(\zeta,z)(\Delta^C_\omega)^\dagger_z\phi(z)dV_z\geq0
    \end{align*}
    for any test function $\phi\in\mathscr{D}(U)_{\geq0}$. 
    Here, $G_\nu(\zeta,z)=0$ for $|\zeta-z|\geq \rho_\nu$ and $\rho_\nu$ decreasing to $0$ if $\nu\to+\infty$ by the definition in \cite{Lit63}.
    Then the function $F_{G_\nu,\phi}(\zeta)$ has compact support for enough large $\nu$, i.e. $F_{G_\nu,\phi}\in\mathscr{D}(U)_{\geq0}$.

    Put the function $K_\nu(\zeta,z):=(\Delta^C_\omega)^\dagger_\zeta G_\nu(\zeta,z)$ as in \cite{Lit63}, then this functions satisfy $K_\nu(\zeta,z)\geq0$ and $\lim_{\nu\to+\infty}\int_UK_\nu(\zeta,z)dV_{\zeta}=1$ uniformly in $z\in U'$.
    We may assume that $U$ is small enough to make $E$ trivial, and let $e=(e_1,\cdots,e_r)$ be a holomorphic frame of $E$ on $U$ and $(h_{\lambda\mu})$ be a matrix representation of $h$ with respect to $e$.
    We define smooth Hermitian metrics $h_\nu$ by $h\ast K_\nu$. In other words, when $(h_{\nu,\lambda\mu})$ is the representation matrix of $h_\nu$ with respect to $e$, we define by
    \begin{align*}
        h_{\nu,\lambda\mu}(z):=\int_UK_\nu(\zeta,z)h_{\lambda\mu}(\zeta)dV_\zeta.
    \end{align*}

    For any fiexd local constant section $\sigma$ of $E$, i.e. $\sigma\in H^0(U,E)$, we have that 
    \begin{align*}
        |\sigma|^2_{h_\nu}(z)=\int_UK_\nu(\zeta,z)|\sigma|^2_h(\zeta)dV_\zeta\geq0
    \end{align*}
    and $h_\nu$ decreasing to $h$ a.e. pointwise by $\lim_{\nu\to+\infty}\int_UK_\nu(\zeta,z)dV_{\zeta}=1$.

    Finally, we show that $\Delta^C_\omega|\sigma|^2_{h_\nu}(\phi)\geq0$ for any test function $\phi\in\mathscr{D}(U)_{\geq0}$. 
    We have %that 
    \begin{align*}
        \Delta^C_\omega|\sigma|^2_{h_\nu}(\phi)&=\int_U(\Delta^C_\omega)_z|\sigma|^2_{h_\nu}(z)\phi(z)dV_z\\
        &=\int_U\phi(z)(\Delta^C_\omega)_z\Bigl\{\int_UK_\nu(\zeta,z)|\sigma|^2_h(\zeta)dV_\zeta\Bigr\}dV_z\\
        &=\int_U\phi(z)(\Delta^C_\omega)_z\Bigl\{\int_U|\sigma|^2_h(\zeta)(\Delta^C_\omega)^\dagger_\zeta G_\nu(\zeta,z)dV_\zeta\Bigr\}dV_z\\
        &=\int_U|\sigma|^2_h(\zeta)\Bigl\{\int_U(\Delta^C_\omega)_z(\Delta^C_\omega)^\dagger_\zeta G_\nu(\zeta,z)\cdot\phi(z)dV_z\Bigr\}dV_\zeta\\
        &=\int_U|\sigma|^2_h(\zeta)(\Delta^C_\omega)^\dagger_\zeta\Bigl\{\int_UG_\nu(\zeta,z)(\Delta^C_\omega)^\dagger_z\phi(z)dV_z\Bigr\}dV_\zeta\\
        &=\int_U|\sigma|^2_h(\zeta)(\Delta^C_\omega)^\dagger_\zeta F_{G_\nu,\phi}(\zeta) dV_\zeta.
    \end{align*}
    
    Here, $\log|\sigma|^2_h$ is $\omega$-subharmonic by the assumption, then $|\sigma|^2_h$ is also $\omega$-subharmonic from Proposition \ref{convex increasing of omega-SH is also omega-SH}. 
    By Proposition \ref{equivalence of currents semi-posi}, we obtain $\Delta^C_\omega|\sigma|^2_h\geq0$ in the weak sense. Hence, for a test function $F_{G_\nu,\phi}\in\mathscr{D}(U)_{\geq0}$, we have that 
    \begin{align*}
        0\leq\langle\Delta^C_\omega|\sigma|^2_h,F_{G_\nu,\phi}\rangle=\int_U|\sigma|^2_h(\zeta)(\Delta^C_\omega)^\dagger_\zeta F_{G_\nu,\phi}(\zeta) dV_\zeta=\Delta^C_\omega|\sigma|^2_{h_\nu}(\phi).
    \end{align*}
    This concludes the proof.
\end{proof}

From this proposition, we provide the following conjecture.

\begin{conjecture}\label{Conj 1}
    Under the situation in Proposition \ref{smoothing of tr-omega semi-posi for constant section}, can $\sigma$ be holomorphic instead of constant?
    In other words, is there an approximation $(h_\nu)_{\nu\in\mathbb{N}}$ at locally for the singular Hermitian metric $h$ which is $\omega$-trace semi-positive?
    Here, smooth Hermitian metrics $h_\nu$ must be $\omega$-trace semi-positive and decreasing to $h$.
\end{conjecture}

Obviously, this conjecture is correct in the case of holomorphic line bundles, and even if $h$ is more strongly Griffiths semi-positive.
Therefore, we define a new situation with such approximation below.

\begin{definition}
    Let $\omega$ be a Hermitian metric and $E$ be a holomorphic vector bundle on $X$. 
    We say that a singular Hermitian metric $h$ on $E$ is 
    \begin{itemize}
        \item $\omega$-\textit{trace semi}-\textit{negative with approximation} if for any $x\in X$, there exist a neighborhood $U$ of $x$ and a sequence of smooth $\omega$-trace semi-negative Hermitian metrics $(h_\nu)_{\nu\in\mathbb{N}}$ on $U$ decreasing to $h$ almost everywhere,
        \item $\omega$-\textit{trace quasi}-\textit{negative with approximation} if $h$ is $\omega$-trace semi-negative with approximation and there exists at least one point $x_0\in X$ such that $h$ is $\omega$-trace negative at $x_0$,
        \item $\omega$-\textit{trace semi}-\textit{positive} (resp. \textit{quasi}-\textit{positive}) \textit{with approximation} if the dual metric $h^*$ on $E^*$ is $\omega$-trace semi-negative (resp. quasi-negative) with approximation.
    \end{itemize}
\end{definition}

Clearly, if a singular Hermitian metric $h$ is $\omega$-trace semi-positive with approximation then $\omega$-trace semi-positive.
Paraphrasing Conjecture \ref{Conj 1}, it says the converse of this.
Moreover, %we have that 
if $h$ is Griffiths semi-positive (resp. quasi-positive) then $\omega$-trace semi-positive (resp. quasi-positive) with approximation for any Hermitian metric $\omega$. %by Proposition \ref{characterization of sing Grif semi-posi} and \ref{characterization of Grif nega by smoothing}.

\section{Properties of various singular positivity}
\subsection{singular RC-positive and a vanishing theorem}

The notion of \textit{RC}-\textit{positivity} was introduced in \cite{Yan18} to resolve Yau's conjecture (see Conjecture \ref{Yau's conjecture}).
%Yau's conjecture (see Conjecture \ref{Yau's conjecture}) was solved in \cite{HW20,Mat22}, and in \cite{Yan18} by introducing \textit{RC}-\textit{positivity}.
Let $E\to X$ be an RC-positive holomorphic vector bundle, then this positivity is such that the Serre line bundle $\mathcal{O}_E(1)\to\mathbb{P}(E)$ is $(n-1)$-positive with respect to the naturally induced metric, where $\mathrm{dim}\,X=n$.
In this subsection, we extend and define PC-positivity to singular Hermitian metric and investigate its properties.

\begin{definition}$($\textnormal{cf.\,\cite[Definition\,1.2]{Yan18}}$)$
    Let $E$ be a holomorphic vector bundle of $\mathrm{rank}\,r$ over a complex manifold $X$ equipped with a smooth Hermitian metric $h$. 
    We say that $h$ is \textit{RC}-\textit{positive} (resp. \textit{negative}) if for any $x\in X$ and any $0\ne a\in E_x$, there exists $0\ne\xi\in T_{X,x}$ such that 
    $\Theta_{E,h}(a\otimes\xi)=\sum(\Theta^h_{jk}a,a)_h\xi_j\overline{\xi}_k>0 \quad (\mathrm{resp}.\,<0)$.
    %\begin{align*}
    %    \Theta_{E,h}(a\otimes\xi)=\sum(\Theta^h_{jk}a,a)_h\xi_j\overline{\xi}_k>0 \quad (\mathrm{resp}.\,<0).
    %\end{align*}
\end{definition}

Here, if Griffiths positive then RC-positive, and if $\mathrm{rank}\,E\!=\!1$ then RC-positive and $(n-1)$-positive are equivalent,
and that $h$ is RC-positive if and only if $h^*$ is RC-negative.

\begin{proposition}
    Let $h$ be a smooth Hermitian metric on $E$. Then $h$ is RC-negative if and only if for any $x\in X$ and any $s\in\mathcal{O}(E)_x$ there exists $0\ne v\in\mathbb{C}^n$ such that  
    $\log|s|^2_h|_{L_v}$ is strictly subharmonic on a neighborhood of $x$, where $L_v:=v\cdot\mathbb{C}$ centered $x$.
\end{proposition}

\begin{proof}
    %First 
    We assume RC-negativity of $h$. For any $s\in\mathcal{O}(E)_x$ and any $\xi\in T_{X,x}$, we have %that  
    \begin{align*}
        %\idd\log |s|^2_h&=i\frac{|D'^hs|^2_h}{|s|^2_h}-i\frac{\langle D'^hs,s\rangle_h\wedge\langle s,D'^hs\rangle_h}{|s|^4_h}-\frac{\langle i\Theta_{E,h}s,s\rangle_h}{|s|^2_h},\\
        \idd\log|s|^2_h(\xi)=i\frac{|s|^2_h|D'^hs\cdot \xi|^2_h-|\langle D'^hs\cdot \xi,s\rangle_h|^2}{|s|^4_h}-\frac{\langle i\Theta_{E,h}(\xi,\overline{\xi})s,s\rangle_h}{|s|^2_h}
        \geq-\frac{\langle i\Theta_{E,h}(\xi,\overline{\xi})s,s\rangle_h}{|s|^2_h}
    \end{align*}
    by Chauchy-Schwarz inequality. Here, we write $i\Theta_{E,h}(s\otimes\xi)=\langle i\Theta_{E,h}(\xi,\overline{\xi})s,s\rangle_h$.
    
    From the assumption and smooth-ness of $h$, for any $0\ne s\in\mathcal{O}(E)_x$ there exists $0\ne v\in\mathbb{C}^n$ such that $\langle i\Theta_{E,h}(v,\overline{v})s,s\rangle_h<0$ on a neighborhood.
    Then we get $\idd\log|s|^2_h(v)\geq-\langle i\Theta_{E,h}(v,\overline{v})s,s\rangle_h/|s|^2_h>0$, i.e. $\log|s|^2_h|_{L_v}$ is strictly subharmonic.  

    Finally, we show the converse of this. %By assumption, for any $0\ne s\in\mathcal{O}(E)_x, \exists\,v\in\mathbb{C}^n$ s.t. $\log|s|^2_h|_{L_v}$ is strictly subharmonic. 
    Here, for any $0\ne a\in E_x$ there exists $0\ne s\in\mathcal{O}(E)_x$ such that $s(x)=a$ and $D'^hs(x)=0$. Therefore, we get 
    \begin{align*}
        0<\idd\log|s|^2_h|_{L_v}=\idd\log|s|^2_h(v)=-\langle i\Theta_{E,h}(v,\overline{v})a,a\rangle_h/|a|^2_h
    \end{align*}
    at $x$. Hence, $h$ is RC-negative.
\end{proof}

Using this characterization, RC-positivity is extended to singular Hermitian metric without curvature currents.

\begin{definition}
    Let $E$ be a holomorphic vector bundle on a complex manifold $X$. 
    We say that a singular Hermitian metric $h$ on $E$ is 
    \begin{itemize}
        \item \textit{RC-negative} if for any $x\in X$ and any $s\in\mathcal{O}(E)_x$, the function $|s|^2_h$ is upper semi-continuous and there exists $v\in\mathbb{C}^n$ such that $\log|s|^2_h|_{L_v}$ is strictly subharmonic, where $L_v=v\cdot\mathbb{C}$ centered $x$,
        \item \textit{RC-positive} if the dual metric $h^*$ on $E^*$ is RC-negative. 
    \end{itemize}
\end{definition}

We obtain the following vanishing and properties.

\begin{proposition}\label{RC-posi then H^0=0}
    Let $E$ be a holomorphic vector bundle on a compact complex manifold $X$ equipped with a singular Hermitian metric $h$.
    If $h$ is RC-positive then $H^0(X,E^*)=0$.
    %Let $X$ be a compact complex manifold and $E$ be a holomorphic vector bundle over $X$. 
    %If $E$ has a singular Hermitian metric $h$ which is RC-positive then $H^0(X,E^*)=0$.
\end{proposition}

\begin{proof}
    Let $0\ne s\in H^0(X,E^*)$ then $|s|^2_{h^*}\geq0$ is upper semi-continuous. Then this has maximum at some point $p\in X$ and $|s|^2_{h^*}(p)>0$, i.e. $\log|s|^2_{h^*}$ has maximum at $p$.
    By RC-negativity of $h^*$, for any $q\in X$, there exists $0\ne v\in\mathbb{C}^n$ such that $\log|s|^2_{h^*}|_{L_v}$ is strictly subharmonic, where $L_v=v\cdot\mathbb{C}$ centered $p$.
    Thus, we take $p=q$ then there exists $r>0$ such that $\log|s|^2_{h^*}|_{L_v}$ has no maximum on $L_v\cap\mathbb{B}(p,r)$ and exists $x_0\in L_v\cap\partial\mathbb{B}(p,r)$
    such that $\log|s|^2_{h^*}(x)<\log|s|^2_{h^*}(x_0)$ for any $x\in L_v\cap\mathbb{B}(p,r)$.
    This is contradiction for $\log|s|^2_{h^*}(p)$ is maximum. Hence, $s=0$ and $H^0(X,E^*)=0$.
\end{proof}

\begin{theorem}\label{tr-posi then RC-posi}
    Let $\omega$ be a Hermitian metric and $E$ be a holomorphic vector bundle on $X$ equipped with a singular Hermitian metric $h$. 
    If $h$ is $\omega$-trace positive then $\Lambda^ph$ and $h^{\otimes k}$ are also $\omega$-trace positive, and in particular RC-positive. 
\end{theorem}

\begin{proof}
    Suppose that $\Lambda^ph$ is not RC-positive. 
    Therefore there exist $x\in X$ and $0\ne \sigma\in\mathcal{O}(\Lambda^pE^*)_x$ such that  
    for any $0\ne v\in\mathbb{C}^n$, $\log|\sigma|^2_{\Lambda^ph^*}|_{L_v}$ is not strictly subharmonic on a neighborhood of $x$, i.e. $-\idd\log|\sigma|^2_{\Lambda^ph^*}|_{L_v}\geq0$ in the sense of currents. %$-\partial_v\overline{\partial}_v\log|\sigma|^2_{\Lambda^ph^*}|_{L_v}$
    Then $-\idd\log|\sigma|^2_{\Lambda^ph^*}\geq0$ in the sense of currents on a neighborhood $W$ of $x$. %In particular, we get $-\Delta\log|\sigma|^2_{\Lambda^ph^*}\geq0$.
    %From Proposition \ref{psh then omega-SH}, $-\log|\sigma|^2_{\Lambda^ph^*}$ is $\omega$-subharmonic, i.e.  $-\Delta_\omega\log|\sigma|^2_{\Lambda^ph^*}\geq0$ on $W$ in the sense of currents.

    By the assumption, there exist a neighborhood $V$ of $x$, $c_x>0$ and a sequence of smooth Hermitian metrics $(h^*_\nu)_{\nu\in\mathbb{N}}$ decreasing to $h^*$ pointwise a.e. on $U$ such that $tr_\omega i\Theta_{E^*,h^*_\nu}\leq-c_xh_\nu^*$ on $V$.
    For smooth Hermitian metric $H$ on $E$, it is easy to see that 
    \begin{align*}
        i\Theta_{\Lambda^pE,\Lambda^pH}&=\Lambda^p i\Theta_{E,H}\in\Gamma(X,\Lambda^{1,1}T^*_X\otimes\mathrm{End}(\Lambda^pE)),\\
        tr_\omega i\Theta_{\Lambda^pE,\Lambda^pH}&=tr_\omega\Bigl(\Lambda^pi\Theta_{E,H}\Bigr)=\Lambda^p\Bigl(tr_\omega i\Theta_{E,H}\Bigr)\in\Gamma(X,\mathrm{End}(\Lambda^pE)).
    \end{align*}
    Therefore, we get $tr_\omega i\Theta_{\Lambda^pE^*,\Lambda^ph^*_\nu}\leq-c_x^p\Lambda^ph^*_\nu$ on $V$. Then $\Lambda^ph$ is also $\omega$-trace positive.

    Let $U:=W\cap V$ be a open subset. For any point $y\in U$, there exists a standard coordinate $(z_1,\ldots,z_n)$ of $\omega$ at $y$ %i.e. $\omega=i\sum dz_j\wedge d\overline{z}_j+O(|z|)$, 
    satisfying $\omega=\idd|z|^2$ at $y$.
    For any $\xi\in T_{X,y}$, we have %that 
    \begin{align*}
        \idd\log|\sigma|^2_{\Lambda^ph^*_\nu}(\xi)%&=i\frac{|\sigma|^2_{\Lambda^ph_\nu}|D'^{\Lambda^ph_\nu}s\cdot \xi|^2_{\Lambda^ph_\nu}-|\langle D'^{\Lambda^ph_\nu}\sigma\cdot \xi,\sigma\rangle_{\Lambda^ph_\nu}|^2}{|\sigma|^4_{\Lambda^ph_\nu}}-\frac{\langle i\Theta_{\Lambda^pE,\Lambda^ph_\nu}(\xi,\overline{\xi})\sigma,\sigma\rangle_{\Lambda^ph_\nu}}{|\sigma|^2_{\Lambda^ph_\nu}}\\
        \geq-\frac{\langle i\Theta_{\Lambda^pE^*,\Lambda^ph^*_\nu}(\xi,\overline{\xi})\sigma,\sigma\rangle_{\Lambda^ph^*_\nu}}{|\sigma|^2_{\Lambda^ph^*_\nu}},
    \end{align*}
    by Chauchy-Schwarz and that 
    \begin{align*}
        \Delta_{\idd|z|^2}\log|\sigma|^2_{\Lambda^ph^*_\nu}&=tr_{\idd|z|^2}\partial\overline{\partial}\log|\sigma|^2_{\Lambda^ph^*_\nu}=\sum\partial\overline{\partial}\log|\sigma|^2_{\Lambda^ph^*_\nu}(\partial/\partial z_j)\\
        &\geq-\sum\langle i\Theta_{\Lambda^pE^*,\Lambda^ph^*_\nu}(\partial/\partial z_j,\partial/\partial \overline{z}_j)\sigma,\sigma\rangle_{\Lambda^ph^*_\nu}/|\sigma|^2_{\Lambda^ph^*_\nu}\\
        &=-tr_{\idd|z|^2}\Theta_{\Lambda^pE^*,\Lambda^ph^*_\nu}(\sigma,\overline{\sigma})/|\sigma|^2_{\Lambda^ph^*_\nu}\geq c^p_x,
    \end{align*}
    at $y$, i.e. $\Delta^C_\omega\log|\sigma|^2_{\Lambda^ph^*_\nu}\geq c^p_x$ at any $y\in U$.
    Then $\idd\log|\sigma|^2_{\Lambda^ph^*_\nu}\wedge\omega^{n-1}\geq c^p_x\omega^n$ on $U$.

    Hence, we obtain $\Delta^C_\omega\log|\sigma|^2_{\Lambda^ph^*}\geq c^p_x$ on $U$ in the sense of currents by Proposition \ref{exists smoothing then (strictly) omega-SH as currents}.
    This is a contradiction to $-\idd\log|\sigma|^2_{\Lambda^ph^*}\geq0$ in the sense of currents on $U$.

    The case of $h^{\otimes k}$ is shown in the same way.
\end{proof}

\subsection{Propositions and induced metrics on the tautological line bundle} %$\mathcal{O}_{E}(1)\to\mathbb{P}(E)$}

In this subsection, we investigate the inheritance of various singular positivity to holomorphic subbundles and tautological line bundles.

\begin{proposition}\label{subbdl is also tr-negative}
    Let $X$ be a complex manifold, $\omega$ be a Hermitian metric on $X$ and $E$ be a holomorphic vector bundle equipped with a singular Hermitian metric $h$. We have %that 
    \begin{itemize}
        \item [($a$)] If $h$ is $\omega$-trace negative, then every subbundle $S$ of $E$ has the natural induced singular Hermitian metric $h_S$ which is $\omega$-trace negative,
        \item [($b$)] If $h$ is $\omega$-trace positive, then every quotient $Q$ of $E$ has the natural induced singular Hermitian metric $h_Q$ which is $\omega$-trace positive.
    \end{itemize}
    The similar results hold for $\omega$-trace semi-positive (resp. quasi-positive) with approximation, $\omega$-trace semi-positive and RC-positive, and their respective negativities.
\end{proposition}

\begin{proof}
    $(a)$ Let $j:S\hookrightarrow E$ be a holomorphic inclusion map. We define the natural singular Hermitian metric $h_S$ on $S$ induced from $h$ by $|u|_{h_S}:=|ju|_h$ for any section $u$ of $S$.
    Let $U$ be an open subset of $X$ and let $c\geq0$. If for any local holomorphic section $s\in H^0(U,E)$, we get $\Delta^C_\omega\log|s|^2_h\geq c$ in the sense of currents on $U$, 
    then for any local holomorphic section $u\in H^0(U,S)$ we have that $ju\in H^0(U,E)$ and $\Delta^C_\omega\log|u|^2_{h_S}$ $=\Delta^C_\omega\log|ju|^2_h\geq c$ in the sense of currents on $U$.
    This is obviously also holds for smooth Hermitian metrics.
    And, if there exists locally an approximate sequence $(h_\nu)_{\nu\in\mathbb{N}}$ decreases to $h$, then the induced approximate sequence $(h_{\nu,S})_{\nu\in\mathbb{N}}$ also decreases to $h_S$.

    $(b)$ There exists a holomorphic vector subbundle $S$ of $Q$ such that the sequence $0\longrightarrow S\longrightarrow E\longrightarrow Q\longrightarrow 0$ is exact.
    Thus, considering this dual sequence which is also exact, we can show by using $(a)$.
\end{proof}

Similar to the proof of Theorem \ref{tr-posi then RC-posi}, the following is obtained.

\begin{proposition}\label{tr-quasi-posi with appro then h^m is also tr-quasi-posi with appro}
    Let $X$ be a complex manifold, $\omega$ be a Hermitian metric on $X$ and $E$ be a holomorphic vector bundle equipped with a singular Hermitian metric $h$.
    If $h$ is $\omega$-trace semi-positive $($resp. quasi-positive$)$ with approximation, then $\Lambda^ph$ and $h^{\otimes k}$ are also $\omega$-trace semi-positive $($resp. quasi-positive$)$ with approximation.
\end{proposition}

Let $\pi$ be the projection $\mathbb{P}(E)\to X$ and $L:=\mathcal{O}_E(1)$.
Any smooth Hermitian metric $h$ on $E$ introduces a \textit{canonical metric} $h_L$ on $L$ as follows:

Let $e=(e_1,\ldots,e_r)$ be a local holomorphic frame with respect to a given trivialization on $E$.
The corresponding holomorphic coordinates on $E^*$ are denoted by $(w_1,\ldots,w_r)$. 
There is a local section of $e_{L^*}$ of $L^*$ defined by 
\begin{align*}
    e_{L^*}=\sum_{1\leq\lambda\leq r}w_\lambda e^*_\lambda.
\end{align*}
Define the induced quotient metric $h_L$ on $L$ by the morphism $(\pi^*E,\pi^*h)\to L$. 
In other words, the dual metric $h_L^*$ on $L^*$ is a metric induced by the natural mapping $\iota:L^*\hookrightarrow(\pi^*E)^*=\pi^*E^*$.
If $(h_{\lambda\mu})$ is the matrix representation of $h$ with respect to the basis $e$, then $h_L$ can be written as 
\begin{align*}
    h_L=\frac{1}{||e_{L^*}||^2_{h_{L^*}}}=\frac{1}{||\iota\circ e_{L^*}||^2_{\pi^*h^*}}=\frac{1}{||e_{L^*}||^2_{h^*}}=\frac{1}{\sum h^*_{\lambda\mu}w_\lambda\overline{w}_\mu},
\end{align*}
where $h^*_{\lambda\mu}=h^{\mu\lambda}$ and $(h^{\lambda\mu}),(h^*_{\lambda\mu})$ are the matrix representation of $h^{-1}, h^*$, respectively.

The curvature of $(L,h_L)$ is 
    \begin{align*}
        i\Theta_{L,h_L}=-\idd\log h_L=\idd\log||e_{L^*}||^2_{h^*}=\idd\log\Bigl(\sum h^{\mu\lambda}w_\lambda\overline{w}_\mu\Bigr).
    \end{align*}

    We fix a point $t\in X$. %$p=(s,[w])\in\mathbb{P}(E)$ where $[w]=[w_1:\cdots:w_r]$. 
    By \cite[ChapterV]{Dem-book}, %[Dem-book, ChapterV, Proposition 12.10], 
    for any standard coordinate $z=(z_1,\ldots,z_n)$ of $\omega$ centered at point $t$, i.e. $\omega=i\sum dz_j\wedge d\overline{z}_j+O(|z|)$,
    there exists a local holomorphic frame $e=(e_1,\ldots,e_r)$ of $E$ around $t$ such that 
    \begin{align*}
        h_{\lambda\mu}=\delta_{\lambda\mu}-\sum c_{jk\lambda\mu}z_j\overline{z}_k+O(|z|^3),%\\
        %h^{\mu\lambda}=\delta_{\lambda\mu}+c_{jk\mu\lambda}z_j\overline{z}_k+O(|z|^3).
    \end{align*}
    where these coefficients $c_{jk\lambda\mu}$ are of curvature $\Theta_{E,h}$ written as follows
    \begin{align*}
        i\Theta_{E,h}(t)=i\sum c_{jk\lambda\mu}dz_j\wedge d\overline{z}_k\otimes e^*_\lambda\otimes e_\mu.
    \end{align*}
    From the following inequality $i\Theta_{E^*,h^*}=-i\,{}^t\Theta_{E,h}$, 
    %\begin{align*}
    %    i\Theta_{E^*,h^*}(t)&=i\sum c^*_{jk\lambda\mu}dz_j\wedge d\overline{z}_k\otimes (e^*_\lambda)^*\otimes e^*_\mu \\
    %    =-i\,{}^t\Theta_{E,h}&=-i\sum c_{jk\mu\lambda}dz_j\wedge d\overline{z}_k\otimes (e^*_\lambda)^*\otimes e^*_\mu, 
    %\end{align*}
    i.e. $c^*_{jk\lambda\mu}=-c_{jk\mu\lambda}$, we have that 
    \begin{align*}
        h^{\mu\lambda}=h^*_{\lambda\mu}=\delta_{\lambda\mu}-\sum c^*_{jk\lambda\mu}z_j\overline{z}_k+O(|z|^3)
        =\delta_{\lambda\mu}+\sum c_{jk\mu\lambda}z_j\overline{z}_k+O(|z|^3).
    \end{align*}

    Therefore, for any point $w\in E_t$, i.e. any point $p=(t,[w])\in\mathbb{P}(E)$ where $[w]=[w_1:\cdots:w_r]$ and $\pi(p)=t$, we have that 
    \begin{align*}
        i\Theta_{L,h_L}(p)&=i\sum c_{jk\mu\lambda}\frac{w_\lambda\overline{w}_\mu}{|w|^2}dz_j\wedge d\overline{z}_k+i\sum_{1\leq \lambda,\mu\leq r}\frac{|w|^2\delta_{\lambda\mu}-\overline{w}_\lambda w_\mu}{|w|^4}dw_\lambda\wedge d\overline{w}_\mu\\
        &=i\sum c_{jk\lambda\mu}\frac{w_\mu\overline{w}_\lambda}{|\overline{w}|^2}dz_j\wedge d\overline{z}_k+\omega_{FS}([w])\\
        &=i\frac{\Theta_{E,h}(\overline{w})}{|\overline{w}|^2_h}+\omega_{FS}([w]),
    \end{align*}
    where $\omega_{FS}$ is Fubini-Study metric on $\mathbb{P}(E^*_t)\cong\mathbb{P}^{r-1}$.

We show that Griffiths positivity as in Definition \ref{def of strictly Grif posi as sing} and $\omega$-trace negativity induce similar singular positivity and negativity in the tautological line bundle, respectively.

\begin{theorem}\label{E Grif posi then O_E(1) is also big}%$($\textnormal{cf.\,\cite[Theorem\,3.3]{WZ24}}$)$
    Let $X$ be a complex manifold and $E$ be a holomorphic vector bundle on $X$ equipped with a singular Hermitian metric $h$.
    If $h$ is Griffiths positive then the tautological line bundle $\mathcal{O}_E(1)\to\mathbb{P}(E)$ has a natural induced singular Hermitian metric and this metric is singular positive.
    In particular, if $X$ is compact then $\mathcal{O}_E(1)$ is big.
\end{theorem}

\begin{proof}
    From Proposition \ref{characterization of sing Grif semi-posi}, for any $x\in X$, there exists a neighborhood $U$ of $x$, $\delta_x>0$ ans a sequence of smooth Griffiths negative Hermitian metrics $(h^*_\nu)_{\nu\in\mathbb{N}}$ on $U$ %any relatively compact Stein subset $\widetilde{S}$ of $S$ where $h_\nu:=h\ast\rho_\nu$ and $(\rho_\nu)_{\nu\in\mathbb{N}}$ is an approximate identity on $S$.
    decreasing to $h^*$ a.e. pointwise %and satisfies %the uniformity condition 
    %that there exists $\delta>0$ 
    such that $\idd|u|^2_{h^*_\nu}\geq\delta_x|u|^2_{h^*_\nu}\idd|z|^2$,
    %\begin{align*}
    %    \idd|u|^2_{h^*_\nu}\geq\delta_x|u|^2_{h^*_\nu}\idd|z|^2,  
    %\end{align*}
    for any $\nu\in\mathbb{N}$, any $y\in U$ and any $0\ne u\in E^*_y$. 
    By the inequality
    \begin{align*}
        \idd|v|^2_{h^*_\nu}=\sum_{j,k}(D'^{h^*_\nu}_{z_j}v,D'^{h^*_\nu}_{z_k}v)_{h^*_\nu}dz_j\wedge d\overline{z}_k-\sum_{j,k}(\Theta^{h^*_\nu}_{jk}v,v)_{h^*_\nu}dz_j\wedge d\overline{z}_k,
    \end{align*}
    for any local holomorphic section $v\in\mathcal{O}(E^*)_y$, we have that 
    \begin{align*}
        -\sum(\Theta^{h^*_\nu}_{jk}u,u)_{h^*_\nu}dz_j\wedge d\overline{z}_k\geq\delta_x|u|^2_{h^*_\nu}\idd|z|^2.
    \end{align*}
    In fact, for any $u\in E^*_y$ we can take $v\in\mathcal{O}(E^*)_y$ satisfying $v(y)=u$ and $D'^{h^*_\nu}v(y)=0$.

    There is a natural antilinear isometry between $E^*$ and $E$, which we will denote by $J_\nu$.
    Denote the pairing between $E^*$ and $E$ by $\langle\cdot,\cdot\rangle$ which satisfies that $\langle\xi,u\rangle=(u,J_\nu\xi)_{h_\nu}$
    for any local section $u$ of $E$ and any local section $\xi$ of $E^*$.
    Under the natural holomorphic structure on $E^*$, we obtain $\overline{\partial}_{z_j}\xi=J^{-1}_\nu D'^{h_\nu}_{z_j}J_\nu\xi$ and $D'^{h^*_\nu}_{z_j}\xi=J^{-1}_\nu\overline{\partial}_{z_j}J_\nu\xi$.
    %\begin{align*}
    %    \overline{\partial}_{z_j}\xi=J^{-1}_\nu D'^{h_\nu}_{z_j}J_\nu\xi, \quad D'^{h^*_\nu}_{z_j}\xi=J^{-1}_\nu\overline{\partial}_{z_j}J_\nu\xi.
    %\end{align*}
    Thus, for any local sections $\xi_j\in C^\infty(E^*)$ and $u_j\in C^\infty(E)$ satisfying $u_j=J_\nu\xi_j$, we get 
    \begin{align*}
        \sum(\Theta^{h^*_\nu}_{jk}\xi_j,\xi_k)_{h^*_\nu}=-\sum(\Theta^{h_\nu}_{jk}u_k,u_j)_{h_\nu},
    \end{align*}

    Hance, for any $\nu\in\mathbb{N}$, any $y\in U$ and any $0\ne \xi\in E_y$, we have that 
    \begin{align*}
        \Theta_{E,h_\nu}(\xi)=\sum(\Theta^{h_\nu}_{jk}\xi,\xi)_{h_\nu}dz_j\wedge d\overline{z}_k\geq\delta_x|\xi|^2_{h_\nu}\idd|z|^2.
    \end{align*}

    Let $h_L^\nu$ be canonical metrics on $L|_{\pi^{-1}(U)}$ induced by $h_\nu$ where %$\pi:\mathbb{P}(E)\to E$ and 
    $L:=\mathcal{O}_E(1)$.
    %For any $t\in U$, let $(z_1,\cdots,z_n)$ be a standard coordinate of $\omega$ at $t$ satisfying $\omega=\idd|z|^2+O(|z|)$ and 
    Let $(e_1,\ldots,e_r)$ be a orthonormal basis on $E$, then we can write 
    \begin{align*}
        i\Theta_{E,h_\nu}&=i\sum c^\nu_{jk\lambda\mu}dz_j\wedge d\overline{z}_k\otimes e^*_\lambda\otimes e_\mu%, \quad %\\
        %tr_\omega i\Theta_{E,h_\nu}=i\sum c^\nu_{jj\lambda\mu} e^*_\lambda\otimes e_\mu
    \end{align*}
    at $y\in U$.
    Then for any point $p=(y,[w])\in\mathbb{P}(E)$, i.e. any $w\in E_y$, we have 
    \begin{align*}
        i\Theta_{L,h_L^\nu}(p)=i\sum c^\nu_{jk\lambda\mu}\frac{w_\mu\overline{w}_\lambda}{|\overline{w}|^2_{h_\nu}}+\omega_{FS}([w])=i\frac{\Theta_{E,h_\nu}(\overline{w})}{|\overline{w}|^2_{h_\nu}}+\omega_{FS}([w])
        \geq \delta_x\idd|z|^2+\omega_{FS}([w]),
    \end{align*}
    i.e. $i\Theta_{L,h^\nu_L}\geq\delta_x\idd|z|^2+\widetilde{\omega}_{FS}$ on $\pi^{-1}(U)$, where
    there is a global metric $\widetilde{\omega}_{FS}$ on $\mathbb{P}(E)$ that is $\overline{\partial}$-closed and is the Fubini-Study metric when restricted to each fiber. 
    
    Here, $\delta_x\idd|z|^2+\widetilde{\omega}_{FS}$ is K\"ahler, and since the weights $\varphi_\nu$ of $h^\nu_L$ has a uniformaly positivity and is strictly plurisubharmonic and decreasing to the weight $\varphi$ of the canonical metric $h_L$ induced by $h$ a.e, then $\varphi$ coincides with some strictly plurisubharmonic function.
    Therefore, the singular Hermitian metric $h_L$ on $L$ is singular positive.
\end{proof}

\begin{theorem}\label{E tr-nega then O_E(1) is also tr-nega}
    Let $X$ be a compact complex manifold, $\omega$ be a Hermitian metric on $X$ and $E$ be a holomorphic vector bundle over $X$ equipped with a singular Hermitian metric $h$.
    Let $\mathcal{O}_E(1)\to\mathbb{P}(E)$ be the tautological line bundle of $E\to X$. 
    If $h$ is $\omega$-trace negative then there exists a Hermitian metric $\widetilde{\omega}$ on $\mathbb{P}(E)$ 
    such that the induced singular Hermitian metric $h_{O(1)}$ on $\mathcal{O}_E(1)$ is $\widetilde{\omega}$-trace negative.
\end{theorem}

\begin{proof}
    For any $x\in X$, there exist a neighborhood $U$ of $x$, $c_x>0$ and a sequence of smooth Hermitian metrics $(h_\nu)_{\nu\in\mathbb{N}}$ on $U$ decreasing to $h$ pointwise such that $tr_\omega i\Theta_{E,h_\nu}\leq -c_xh_\nu$ on $U$ for any $\nu\in\mathbb{N}$.
    Let $h_L^\nu$ be canonical metrics on $L|_{\pi^{-1}(U)}$ where %$\pi:\mathbb{P}(E)\to E$ and 
    $L:=\mathcal{O}_E(1)$.
    For any $t\in U$, let $(z_1,\cdots,z_n)$ be a standard coordinate of $\omega$ at $t$ satisfying $\omega=\idd|z|^2+O(|z|)$ and $(e_1,\ldots,e_r)$ be a orthonormal basis on $E$, then we can write 
    \begin{align*}
        i\Theta_{E,h_\nu}=i\sum c^\nu_{jk\lambda\mu}dz_j\wedge d\overline{z}_k\otimes e^*_\lambda\otimes e_\mu, \quad %\\
        tr_\omega i\Theta_{E,h_\nu}=i\sum c^\nu_{jj\lambda\mu} e^*_\lambda\otimes e_\mu
    \end{align*}
    at $t$.
    Then for any point $p=(t,[w])\in\mathbb{P}(E)$, i.e. any $w\in E_t$, we have 
    \begin{align*}
        i\Theta_{L,h_L^\nu}(p)=i\sum c^\nu_{jk\lambda\mu}\frac{w_\mu\overline{w}_\lambda}{|\overline{w}|^2_{h_\nu}}+\omega_{FS}([w])=i\frac{\Theta_{E,h_\nu}(\overline{w})}{|\overline{w}|^2_{h_\nu}}+\omega_{FS}([w]).
    \end{align*}

    From the above, there is a global metric $\widetilde{\omega}_{FS}$ on $\mathbb{P}(E)$ that is $\overline{\partial}$-closed and is the Fubini-Study metric when restricted to each fiber.
    Let $\gamma>0$ and $\widetilde{\omega}:=\pi^*\omega+\gamma\widetilde{\omega}_{FS}$ then $\widetilde{\omega}$ is a Hermitian metric on $\mathbb{P}(E)$ since $\widetilde{\omega}_{FS}$ is not dependent on $z$.
    There exists a local coordinate $(\zeta_1,\ldots,\zeta_{r-1})$ centered $[w]$ on $\mathbb{P}(E^*_s)$ such that $\gamma\widetilde{\omega}_{FS}=i\sum d\zeta_j\wedge d\overline{\zeta}_j+(|\zeta|^2)$.
    Therefore, we obtain $\widetilde{\omega}(p)=\pi^*\omega(s)+\gamma\widetilde{\omega}_{FS}([w])=i\sum dz_j\wedge d\overline{z}_j+i\sum d\zeta_j\wedge d\overline{\zeta}_j$ and
    \begin{align*}
        %\widetilde{\omega}(p)&=\pi^*\omega(s)+\gamma\widetilde{\omega}_{FS}([w])=i\sum dz_j\wedge d\overline{z}_j+i\sum d\zeta_j\wedge d\overline{\zeta}_j,\\
        i\Theta_{L,h_L^\nu}(p)\wedge\widetilde{\omega}^{n-1}(p)&=\Bigl(i\sum c^\nu_{jk\lambda\mu}\frac{w_\mu\overline{w}_\lambda}{|\overline{w}|^2_{\nu}}dz_j\wedge d\overline{z}_k+\frac{i}{\gamma}\sum d\zeta_j\wedge d\overline{\zeta}_j\Bigr)\wedge\widetilde{\omega}^{n-1}(p)\\
        &=\Bigl(i\sum c^\nu_{jj\lambda\mu}\frac{w_\mu\overline{w}_\lambda}{|\overline{w}|^2_{h_\nu}}dz_j\wedge d\overline{z}_j+\frac{i}{\gamma}\sum d\zeta_j\wedge d\overline{\zeta}_j\Bigr)\wedge\widetilde{\omega}^{n-1}(p)\\
        %&=i\sum c_{jj\lambda\mu}\frac{w_\mu\overline{w}_\lambda}{|\overline{w}|^2_h}dz_j\wedge d\overline{z}_j\wedge\widetilde{\omega}^{n-1}(p)+\frac{i}{\gamma}\sum d\zeta_j\wedge d\overline{\zeta}_j\wedge\widetilde{\omega}^{n-1}(p)\\
        &=\Bigl(\frac{\sum c^\nu_{jj\lambda\mu}\overline{w}_\lambda w_\mu}{|\overline{w}|^2_{h_\nu}}+\frac{r-1}{\gamma}\Bigr)\widetilde{\omega}^n
        =\Bigl(\frac{tr_\omega\Theta_{E,h_\nu}(\overline{w})}{|\overline{w}|^2_{h_\nu}}+\frac{r-1}{\gamma}\Bigr)\widetilde{\omega}^n\\
        &\leq\Bigl(-c_x+\frac{r-1}{\gamma}\Bigr)\widetilde{\omega}^n.
    \end{align*}
    Hance, we have $i\Theta_{L,h_L^\nu}\wedge\widetilde{\omega}^{n-1}\!\leq\!(-c_x+\frac{r-1}{\gamma})\widetilde{\omega}^n$ on $\pi^{-1}(U)$ for any $\nu\in\mathbb{N}$.
    By compact-ness of $X$, we can take $\gamma$ large enough such that $C_x:=c_x-\frac{r-1}{\gamma}>0$ for any $x\in X$.
    
    Let $h_L$ be the canonical metric $L$ over $\mathbb{P}(E)$ induced by $h$. 
    Then the weight function $\varphi_L$ of $h_L|_{\pi^{-1}(U)}$ satisfies that $\idd\varphi_L\wedge\widetilde{\omega}^{n-1}\!\leq\!(-c_x+\frac{r-1}{\gamma})\widetilde{\omega}^n<0$ on $\pi^{-1}(U)$ by Proposition \ref{exists smoothing then (strictly) omega-SH as currents}.
    Hance, $h_L$ is $\widetilde{\omega}$-trace negative.
\end{proof}

%And, the following follows immediately from this theorem.

\begin{corollary}
    Let $X$ be a compact complex manifold, $\omega$ be a Hermitian metric on $X$ and $E$ be a holomorphic vector bundle over $X$ equipped with a singular Hermitian metric $h$.
    If $h$ is $\omega$-trace positive then $E^*$ is not pseudo-effective.
\end{corollary}

\begin{proof}
    Suppose $E^*$ is pseudo-effective, then the tautological line bundle $\mathcal{O}_{E^*}(1)$ is also pseudo-effective.
    However, $\mathcal{O}_{E^*}(1)$ is not pseudo-effective by Theorem \ref{E tr-nega then O_E(1) is also tr-nega} and \ref{characterization of tr-posi, not psef}. %This is a contradiction.
    %From Theorem \ref{E tr-nega then O_E(1) is also tr-nega}, there exists a Hermitian metric $\widetilde{\omega}$ on $\mathbb{P}(E^*)$ such that the tautological line bundle $\mathcal{O}_{E^*}(1)$ on $\mathbb{P}(E^*)$ has the induced singular Hermitian metric satisfying $\widetilde{\omega}$-trace negativity.
    %Then, $\mathcal{O}_{E^*}(1)$ is not pseudo-effective by Theorem \ref{characterization of tr-posi, not psef}. This is a contradiction. %to pseudo-effectivity of $\mathcal{O}_{E^*}(1)$.
\end{proof}

\section{Definition and properties of $\mathrm{deg}\,_\omega$-maximum}

In this section, we introduce the following notion in order to $0$-th cohomology vanishing and show that this notion follows from $\omega$-trace positivity.

\begin{definition}\label{def of deg-max}
    Let $X$ be a compact \kah manifold of dimension $n$ equipped with a \kah metric $\omega$.
    Let $\mathcal{F}$ be a torsion-free coherent sheaf over $X$. 
    We say that $\mathcal{F}$ is $\mathrm{deg}\,_\omega$-\textit{maximal} if for any coherent subsheaf $\mathcal{S}$ with positive rank, we have 
    \begin{align*}
        \mathrm{deg}\,_\omega(\mathcal{S})\leq\mathrm{deg}\,_\omega(\mathcal{F}).
    \end{align*}
    
    If moreover the strict inequality $\mathrm{deg}\,_\omega(\mathcal{S})<\mathrm{deg}\,_\omega(\mathcal{F})$
    %\begin{align*}
    %    \mathrm{deg}\,_\omega(\mathcal{S})<\mathrm{deg}\,_\omega(\mathcal{F})
    %\end{align*}
    holds for all coherent subsheaf $\mathcal{S}$ with $0<\mathrm{rank}\,\mathcal{S}<\mathrm{rank}\,\mathcal{F}$, we say that $\mathcal{F}$ is $\mathrm{deg}\,_\omega$-\textit{strictly maximal}.
\end{definition}

Similar to the stability case, we immediately have the following.

\begin{lemma}\label{deg max by quotient}
    Let $X$ be a compact \kah manifold with a \kah metric $\omega$ and $\mathcal{F}$ be a torsion-free coherent sheaf. %over $X$. 
    Then $\mathcal{F}$ is $\mathrm{deg}\,_\omega$ $($resp. strictly$)$ maximal if and only if 
    $\mathrm{deg}\,_\omega(\mathcal{Q})\geq0$ $($resp. $>0)$ for any quotient sheaf $\mathcal{Q}$ with $0<\mathrm{rank}\,\mathcal{Q}$ $($resp. $<\mathrm{rank}\,\mathcal{F})$. % $0<\mathrm{rank}\,\mathcal{Q}<\mathrm{rank}\,\mathcal{F})$.
\end{lemma}

\begin{proposition}\label{characterization of deg max}
    Let $X$ be a compact \kah manifold equipped with a \kah metric $\omega$ and $\mathcal{F}$ be a torsion-free coherent sheaf. %over $X$. 
    Then the following conditions are equivalent.
    \begin{itemize}
        \item [$(a)$] $\mathcal{F}$ is $\mathrm{deg}\,_\omega$-maximal,
        \item [$(b)$] $\mathrm{deg}\,_\omega(\mathcal{S})\leq\mathrm{deg}\,_\omega(\mathcal{F})$, for any subsheaf $\mathcal{S}$ with torsion-free quotient,
        \item [$(c)$] $\mathrm{deg}\,_\omega(\mathcal{Q})\geq0$ for any torsion-free quotient sheaf $\mathcal{Q}$.
    \end{itemize}
    The same equivalence relationship holds for $\mathrm{deg}_\omega$-strictly maximal, corresponding to the strict inequality, except that $0<\mathrm{rank}\,\mathcal{S},\,\mathrm{rank}\,\mathcal{Q}<\mathrm{rank}\,\mathcal{F}$ is required.
\end{proposition}

\begin{proof}
    $(a) \Longrightarrow (b)$ and $(c)$ are trivial. $(b) \iff (c)$ follows from Lemma \ref{deg max by quotient}. We show that $(b) \Longrightarrow (a)$. 
    For any coherent subsheaf $\mathcal{S}'$, there exists an exact sequence 
    \begin{align*}
        0\longrightarrow\mathcal{S}'\longrightarrow\mathcal{F}\longrightarrow\mathcal{Q}'\longrightarrow0.
    \end{align*}
    Let $\mathcal{T}'$ be the torsion subsheaf of $\mathcal{Q}'$. Set $\mathcal{Q}=\mathcal{Q}'/\mathcal{T}'$ which is torsion-free and define $\mathcal{S}$ 
    by the exact sequence $0\longrightarrow\mathcal{S}\longrightarrow\mathcal{F}\longrightarrow\mathcal{Q}\longrightarrow0.$
    %\begin{align*}
    %    0\longrightarrow\mathcal{S}\longrightarrow\mathcal{F}\longrightarrow\mathcal{Q}\longrightarrow0.
    %\end{align*}
    
    Then $\mathcal{S}'$ is a subsheaf of $\mathcal{S}$ and the quotient sheaf $\mathcal{S}/\mathcal{S}'$ is isomorphic to the torsion sheaf $\mathcal{T}'$.
    By the assumption, we have $\mathrm{deg}\,_\omega(\mathcal{S})\leq\mathrm{deg}\,_\omega(\mathcal{F})$.
    Here, $\mathrm{det}\,\mathcal{T}'$ admits a non-trivial holomorphic global section $\sigma$ from torsion-ness of $\mathcal{T}'$.
    Therefore, we have %that 
    \begin{align*}
        \mathrm{deg}\,_\omega(\mathcal{S})-\mathrm{deg}\,_\omega(\mathcal{S}')=\mathrm{deg}\,_\omega(\mathcal{T}')=\int_D\omega^{n-1}\geq0,
    \end{align*}
    where $D$ is divisor of $\sigma$, and that $\mathrm{deg}\,_\omega(\mathcal{S}')\leq\mathrm{deg}\,_\omega(\mathcal{F})$.
\end{proof}

\begin{proposition}
    Let $X$ be a compact \kah manifold with a \kah metric $\omega$ and $\mathcal{F}$ be a torsion-free coherent sheaf over $X$. 
    We have the following.
    \begin{itemize}
        \item if $\mathcal{F}$ is $\omega$-semistable with $\mathrm{deg}\,_\omega(\mathcal{F})\geq0$ then $\mathcal{F}$ is $\mathrm{deg}\,_\omega$-maximal,
        %\item if $\mathcal{F}$ is $\omega$-semistable with $\mathrm{deg}\,_\omega(\mathcal{F})>0$ or is $\omega$-stable with $\mathrm{deg}\,_\omega(\mathcal{F})\geq0$ then $\mathcal{F}$ is $\mathrm{deg}\,_\omega$-strictly maximal,
        \item if $\mathcal{F}$ is $\mathrm{deg}\,_\omega$-maximal with $\mathrm{deg}\,_\omega(\mathcal{F})\leq0$ then $\mathcal{F}$ is $\omega$-semistable.
        %\item if $\mathcal{F}$ is $\mathrm{deg}\,_\omega$-maximal with $\mathrm{deg}\,_\omega(\mathcal{F})<0$ or is $\mathrm{deg}\,_\omega$-strictly maximal with $\mathrm{deg}\,_\omega(\mathcal{F})\leq0$ then $\mathcal{F}$ is $\omega$-semistable.
    \end{itemize}
    Similar properties also hold for stability and strictly-maximal.
\end{proposition}

For any torsion-free sheaf $\mathcal{F}$ on a compact \kah manifold $X$, we say that $\mathcal{F}$ is \textit{dual} $\mathrm{deg}\,_\omega$-\textit{maximal} (resp. \textit{strictly maximal}) for a \kah metric $\omega$ if the dual sheaf $\mathcal{F}^*$ is $\mathrm{deg}\,_\omega$-maximal (resp. strictly maximal).

\begin{proposition}\label{characterization of dual deg max}
    Let $X$ be a compact \kah manifold equipped with a \kah metric $\omega$ and $\mathcal{F}$ be a torsion-free coherent sheaf. %over $X$. 
    Then the following conditions are equivalent.
    \begin{itemize}
        \item [$(a)$] $\mathcal{F}$ is dual $\mathrm{deg}\,_\omega$-maximal,
        \item [$(b)$] $\mathrm{deg}\,_\omega(\mathcal{S})\leq0$, for any subsheaf $\mathcal{S}$ with torsion-free quotient,
        \item [$(c)$] $\mathrm{deg}\,_\omega(\mathcal{Q})\geq\mathrm{deg}\,_\omega(\mathcal{F})$ for any torsion-free quotient sheaf $\mathcal{Q}$.
    \end{itemize}
    The same equivalence relationship holds for $\mathrm{deg}_\omega$-strictly maximal, corresponding to the strict inequality, except that $0<\mathrm{rank}\,\mathcal{S},\,\mathrm{rank}\,\mathcal{Q}<\mathrm{rank}\,\mathcal{F}$ is required.
\end{proposition}

We obtain the following $0$-th cohomology vanishing.

\begin{proposition}\label{dual deg max then H^0=0}
    Let $X$ be a compact \kah manifold with a \kah metric $\omega$.
    If a torsion-free sheaf $\mathcal{F}$ on $X$ %is $\mathrm{deg}\,_\omega$-maximal with $\mathrm{deg}\,_\omega(\mathcal{F})<0$ or 
    is dual $\mathrm{deg}\,_\omega$-strictly maximal, then $\mathcal{F}$ admits no nonzero holomorphic section, i.e. $H^0(X,\mathcal{F})=0$.
\end{proposition}

\begin{proof}
    We assume that $0\ne f\in H^0(X,\mathcal{F})$, i.e. $0\ne f:\mathcal{O}_X\to \mathcal{F}$. Let $\mathcal{S}:=f(\mathcal{O}_X)$, then $\mathcal{S}$ is a torsion-free quotient sheaf of $\mathcal{O}_X$ and is a torsion-free subsheaf of $\mathcal{F}$. %, i.e. 

    Since $\mathcal{O}_X$ is $\omega$-semistable, we get $0=\mathrm{deg}\,_\omega(\mathcal{O}_X)\leq\mathrm{deg}\,_\omega(\mathcal{S})$. 
    By the assumption and Proposition \ref{characterization of dual deg max}, we have that $\mathrm{deg}\,_\omega(\mathcal{S})<0$. This is a contradiction.
\end{proof}

This proposition is a generalization of the already known Corollary \ref{V-thm if semistable in Kob87}.
In fact, if $\mathcal{F}$ is $\omega$-semistable sheaf on $X$ with $\mathrm{deg}\,_\omega(\mathcal{F})<0$, then $\mathcal{F}$ is dual $\mathrm{deg}\,_\omega$-strictly maximal.

\begin{corollary}$($\textnormal{cf.\,\cite[Corollary\,5.7.12]{Kob87}}$)$\label{V-thm if semistable in Kob87}
    Let $X$ be a compact \kah manifold with a \kah metric $\omega$.
    If $\mathcal{F}$ is $\omega$-semistable sheaf on $X$ with $\mathrm{deg}\,_\omega(\mathcal{F})<0$, then $\mathcal{F}$ admits no nonzero holomorphic section, i.e. $H^0(X,\mathcal{F})=0$.
\end{corollary}

Finally, we clarify the relationship between $\omega$-trace positivity and $\mathrm{deg}\,_\omega$-maximal by using the following lemmas.

\begin{lemma}$($\textnormal{cf.\,\cite[Proposition\,3]{Jac14}}$)$\label{Jac14 Lemma 1}
    Let $E$ be a holomorphic vector bundle on a compact complex manifold $X$ and $\mathcal{S}$ be a torsion-free subsheaf of $E$ with torsion-free quotient $\mathcal{Q}$.
    Then after a finite number of blown ups $\pi:\widetilde{X}\to X$, there exists a holomorphic subbundle $\widetilde{\mathcal{S}}$ of $\pi^*E$ containing $\pi^*\mathcal{S}$ with a holomorphic quotient bundle $\widetilde{\mathcal{Q}}$
    such that $\pi_*\widetilde{\mathcal{S}}=\mathcal{S}$ in codimension $1$.
\end{lemma}

\begin{lemma}$($\textnormal{cf.\,\cite[Lemma\,2]{Jac14}}$)$\label{Jac14 Lemma 2}
    Under the same situation as above, for a \kah metric $\omega$ on $X$ we have that 
    \begin{align*}
        \mathrm{deg}\,_\omega(\mathcal{Q})=\mathrm{deg}\,_{\pi^*\omega}(\widetilde{\mathcal{Q}}):=\int_{\widetilde{X}}c_1(\widetilde{\mathcal{Q}})\wedge\pi^*\omega^{n-1}.
    \end{align*}
\end{lemma}

\begin{theorem}\label{tr-posi then deg max}
    Let $X$ be a compact \kah manifold with a \kah metric $\omega$ and $E$ be a holomorphic vector bundle on $X$ equipped with a singular Hermitian metric $h$.
    If $h$ is $\omega$-trace semi-positive $($resp. quasi-positive$)$ with approximation, then $E$ is $\mathrm{deg}\,_\omega$-maximal $($resp. strictly maximal\,$)$ and we have that $\mathrm{deg}\,_\omega(E)\geq0$ $($resp. $>0)$.
\end{theorem}

\begin{proof}
    From Proposition \ref{characterization of deg max}, it is sufficient to show that $\mathrm{deg}\,_\omega(\mathcal{Q})\geq0$ for any torsion-free quotient sheaf $\mathcal{Q}$.
    We define a subsheaf $\mathcal{S}$ of $E$ by the exact sequence 
    \begin{align*}
        0\longrightarrow\mathcal{S}\longrightarrow E\longrightarrow\mathcal{Q}\longrightarrow0,
    \end{align*}
    then $\mathcal{S}$ is torsion-free. By Lemma \ref{Jac14 Lemma 1} and \ref{Jac14 Lemma 2}, there is blown ups $\pi:\widetilde{X}\to X$ and a holomorphic subbundle $\widetilde{\mathcal{S}}$ of $\pi^*E$ with a holomorphic quotient bundle $\widetilde{\mathcal{Q}}$ such that 
    \begin{align*}
        \mathrm{deg}\,_\omega(\mathcal{Q})=\mathrm{deg}\,_{\pi^*\omega}(\widetilde{\mathcal{Q}})=\int_{\widetilde{X}}c_1(\widetilde{\mathcal{Q}})\wedge\pi^*\omega^{n-1}.
    \end{align*}

    Here, $\pi$ is blown ups along an analytic subset $Z$ of codimension $\geq2$. %all quotient are "coh" torsion free より singular part は codim 2 以上でこれに沿った blown up である.
    Let $D$ be the exceptional divisor of $\pi$. Since $\pi$ is biholomorphic on $\widetilde{X}\setminus D$ and Proposition \ref{pull back omega-SH for hol immersion map}, the singular Hermitian metric $\pi^*h$ on $\pi^*E$ is also $\pi^*\omega$-trace semi-positive over $\widetilde{X}\setminus D$.

    Let $x_0\in Z$ then there exist an open neighborhood $U$ of $x_0$ and a sequence of smooth $\omega$-trace semi-positive Hermitian metrics $(h_\nu)_{\nu\in\mathbb{N}}$ increasing to $h$ a.e. on $E|_U$ which is trivial.
    We consider the short exact sequence of holomorphic vector bundles 
    \begin{align*}
        0\longrightarrow \widetilde{\mathcal{Q}}^* \overset{\iota}{\longrightarrow} (\pi^*E)^*\longrightarrow \widetilde{\mathcal{S}}^*\longrightarrow 0.
    \end{align*}
    Here, for any smooth Hermitian metric $\pi^*h_\nu$ on $\pi^*E|_U$, the dual metric $(\pi^*h_\nu)^*$ is $\pi^*\omega$-trace semi-negative on $\pi^{-1}(U)\setminus D$.
    Let $h^{\widetilde{Q}^*}$ be the induced singular Hermitian metric on $\widetilde{\mathcal{Q}}^*$ by $\iota$ and $(\pi^*h)^*$, which is defined by $|s|_{h^{\widetilde{Q}^*}}:=|\iota s|_{(\pi^*h)^*}$ for any $s\in \mathcal{O}_{\widetilde{X}}(\widetilde{\mathcal{Q}}^*)$.
    In the same manner, $(\pi^*h_\nu)^*$ induces smooth Hermitian metrics $h_\nu^{\widetilde{Q}^*}$ on $\widetilde{\mathcal{Q}}^*|_{\pi^{-1}(U)}$ decreasing to $h^{\widetilde{Q}^*}$ a.e., which are also $\pi^*\omega$-trace semi-negative on $\pi^{-1}(U)\setminus D$.
    Therefore, dual metrics $h_\nu^{\widetilde{Q}}:=(h_\nu^{\widetilde{Q}^*})^*$ on $\widetilde{\mathcal{Q}}|_{\pi^{-1}(U)}$ are $\pi^*\omega$-trace semi-positive on $\pi^{-1}(U)\setminus D$ and increasing to $h^{\widetilde{Q}}:=(h^{\widetilde{Q}^*})^*$. 
    Then smooth Hermitian metrics $\mathrm{det}\,h^{\widetilde{Q}}_\nu$ on $\mathrm{det}\,\widetilde{\mathcal{Q}}|_{\pi^{-1}(U)}$ increasing to $\mathrm{det}\,h^{\widetilde{Q}}$ a.e. are also $\pi^*\omega$-trace semi-positive on $\pi^{-1}(U)\setminus D$ by the proof in Theorem \ref{tr-posi then RC-posi}.
    In particular, $\mathrm{det}\,h^{\widetilde{Q}}$ is also $\pi^*\omega$-trace semi-positive on $\pi^{-1}(U)\setminus D$ by Proposition \ref{exists smoothing then (strictly) omega-SH as currents}.

    Define $\varphi^{\widetilde{Q}}:=-\log\mathrm{det}\,h^{\widetilde{Q}}$ and $\varphi_\nu^{\widetilde{Q}}:=-\log\mathrm{det}\,h^{\widetilde{Q}}_\nu$. For any fiexd test function $\rho_U\in\mathscr{D}(\pi^{-1}(U))_{\geq0}$, we have that 
    \begin{align*}
        \int_{\pi^{-1}(U)}i\Theta_{\mathrm{det}\,\widetilde{\mathcal{Q}},\,\mathrm{det}\,h^{\widetilde{Q}}}\wedge\rho_U\pi^*\omega^{n-1}&=\int_{\pi^{-1}(U)}\idd\varphi^{\widetilde{Q}}\wedge\rho_U\pi^*\omega^{n-1}
        =\int_{\pi^{-1}(U)}\varphi^{\widetilde{Q}}\cdot\idd(\rho_U\pi^*\omega^{n-1})\\
        &=\int_{\pi^{-1}(U)}\lim_{\nu\to+\infty}\varphi^{\widetilde{Q}}_\nu\cdot\idd(\rho_U\pi^*\omega^{n-1})\\
        &\geq\limsup_{\nu\to+\infty}\int_{\pi^{-1}(U)}\varphi^{\widetilde{Q}}_\nu\cdot\idd(\rho_U\pi^*\omega^{n-1})\\
        %&=\limsup_{\nu\to+\infty}\int_{\pi^{-1}(U)}\idd\varphi^Q_\nu\wedge\rho_U\pi^*\omega^{n-1}\\
        &\geq0.
    \end{align*}
    In fact, since $\varphi^{\widetilde{Q}}_\nu$ is smooth on $\pi^{-1}(U)$ and is $\pi^*\omega$-subharmonic on $\pi^{-1}(U)\setminus D$, we get % and $\pi^*\omega$-trace semi-positivity of $\mathrm{det}\,h^{\widetilde{Q}}_\nu$, i.e. $\pi^*\omega$-subharmonicity of $\varphi^{\widetilde{Q}}_\nu$, on $\pi^{-1}(U)\setminus D$ we get
    \begin{align*}
        \int_{\pi^{-1}(U)}\varphi^{\widetilde{Q}}_\nu\cdot\idd(\rho_U\pi^*\omega^{n-1})&=\int_{\pi^{-1}(U)\setminus D}\varphi^{\widetilde{Q}}_\nu\cdot\idd(\rho_U\pi^*\omega^{n-1})=\int_{\pi^{-1}(U)\setminus D}\idd\varphi^{\widetilde{Q}}_\nu\wedge\rho_U\pi^*\omega^{n-1}\\
        &=(n-1)!\int_{\pi^{-1}(U)\setminus D}\rho_U\frac{n\idd\varphi^{\widetilde{Q}}_\nu\wedge\pi^*\omega^{n-1}}{\pi^*\omega^n}dV_{\pi^*\omega}\\
        &=(n-1)!\int_{\pi^{-1}(U)\setminus D}\rho_U\Delta^C_{\pi^*\omega}\varphi^{\widetilde{Q}}_\nu \,dV_{\pi^*\omega}\geq0.
        %&\geq0.
    \end{align*}

    In the case of $x_0\in X\setminus Z$, there exists an approximation of $h$ for an open subset $V\subset X\setminus Z$ such that $E|_V$ is trivial. 
    And as above, for the naturally induced singular Hermitian metric $h^{\widetilde{Q}}$ on $\widetilde{\mathcal{Q}}|_{\pi^{-1}(V)}$, the singular Hermitian metrics $h^{\widetilde{Q}}$ and $\mathrm{det}\,h^{\widetilde{Q}}$ are $\pi^*\omega$-trace semi-positive on $\pi^{-1}(V)$ and we have that
    \begin{align*}
        \int_{\pi^{-1}(V)}i\Theta_{\mathrm{det}\,\widetilde{\mathcal{Q}},\,\mathrm{det}\,h^{\widetilde{Q}}}\wedge\rho_V\pi^*\omega^{n-1}\geq0
    \end{align*}
    for any test function $\rho_V\in\mathscr{D}(V)_{\geq0}$.
    
    By compact-ness of $X$, there exist a finite points $x_1,\cdots,x_s$ in $Z$ and $x_{s+1},\cdots,x_N$ in $X\setminus Z$ and open neighborhoods $U_k$ of $x_k$ such that $h$ has approximation on each $U_k$, $E|_{U_k}$ is trivial and the sets $\{U_k\}_{k\in I_N}$ is the open over of $X$. %, i.e. $\bigcup_{k\in I_N}U_k=X$.
    Let $\{\rho_k\}_{k\in I_N}$ be a partitions of unity subordinate to the open over $\{\pi^{-1}(U_k)\}_{k\in I_N}$ of $\widetilde{X}$, i.e. $\rho_k\in\mathscr{D}(\pi^{-1}(U_k))_{\geq0}$ satisfies $\mathrm{supp}\,\rho_k\subset\subset\pi^{-1}(U_k)$. 
    Here, it is well-known that there exist a smooth Hermitian metric $\mathrm{det}\,h_0^{\widetilde{Q}}$ on $\mathrm{det}\,\widetilde{\mathcal{Q}}$ and a real valued function $\psi\in\mathscr{L}^1_{loc}(\widetilde{X},\mathbb{R})$ such that %$i\Theta_{\mathrm{det}\,(\pi^*E/\widetilde{E}_j),\,\mathrm{det}\,h^Q}=i\Theta_{\mathrm{det}\,(\pi^*E/\widetilde{E}_j),\,\mathrm{det}\,h_0}+\idd\psi$.
    \begin{align*}
        i\Theta_{\mathrm{det}\,\widetilde{\mathcal{Q}},\,\mathrm{det}\,h^{\widetilde{Q}}}=i\Theta_{\mathrm{det}\,\widetilde{\mathcal{Q}},\,\mathrm{det}\,h_0^{\widetilde{Q}}}+\idd\psi.
    \end{align*}

    Hance, we have that 
    \begin{align*}
        \mathrm{deg}\,_{\pi^*\omega}(\widetilde{\mathcal{Q}})=\int_{\widetilde{X}}c_1(\widetilde{\mathcal{Q}})\wedge\pi^*\omega^{n-1}&=\int_{\widetilde{X}}i\Theta_{\mathrm{det}\,\widetilde{\mathcal{Q}},\,\mathrm{det}\,h_0^{\widetilde{Q}}}\wedge\pi^*\omega^{n-1}
        =\int_{\widetilde{X}}i\Theta_{\mathrm{det}\,\widetilde{\mathcal{Q}},\,\mathrm{det}\,h^{\widetilde{Q}}}\wedge\pi^*\omega^{n-1}\\
        &=\int_{\widetilde{X}}i\Theta_{\mathrm{det}\,\widetilde{\mathcal{Q}},\,\mathrm{det}\,h^{\widetilde{Q}}}\wedge\sum_{j\in I_N}\rho_k\pi^*\omega^{n-1}\\
        &=\sum_{k\in I_N}\int_{\pi^{-1}(U_k)}i\Theta_{\mathrm{det}\,\widetilde{\mathcal{Q}},\,\mathrm{det}\,h^{\widetilde{Q}}}\wedge\rho_k\pi^*\omega^{n-1}\\
        &\geq0,
    \end{align*}
    where $\idd\pi^*\omega^{n-1}=0$. It is shown similarly with respect to $\omega$-trace quasi-positivity.
\end{proof}

\section{Vanishing theorems for singular quasi-positivity and applications}
\subsection{Vanishing theorems for singular positivity weaker than Griffiths}

In this subsection, we present the $0$-th cohomology vanishing for various singular positivity. %if singular Hermitian metrics have various positivity.
We already know the following vanishing theorem for singular Griffiths positivity.

\begin{theorem}$($\textnormal{cf.\,\cite{Ina20}}$)$\label{V-thm of Inayama}
    Let $X$ be a projective manifold and $E$ be a holomorphic vector bundle equipped with a singular Hermitian metric $h$. 
    If $h$ is Griffiths positive and the Lelong number $\nu(-\log\mathrm{det}\,h,x)<2$ for all points $x\in X$ then we have that % the $n$-th cohomology group vanishing 
    \begin{align*}
        H^n(X,K_X\otimes E)=0.
    \end{align*}
\end{theorem}

For any singular Hermitian metric on $E$, we have the following relationship.
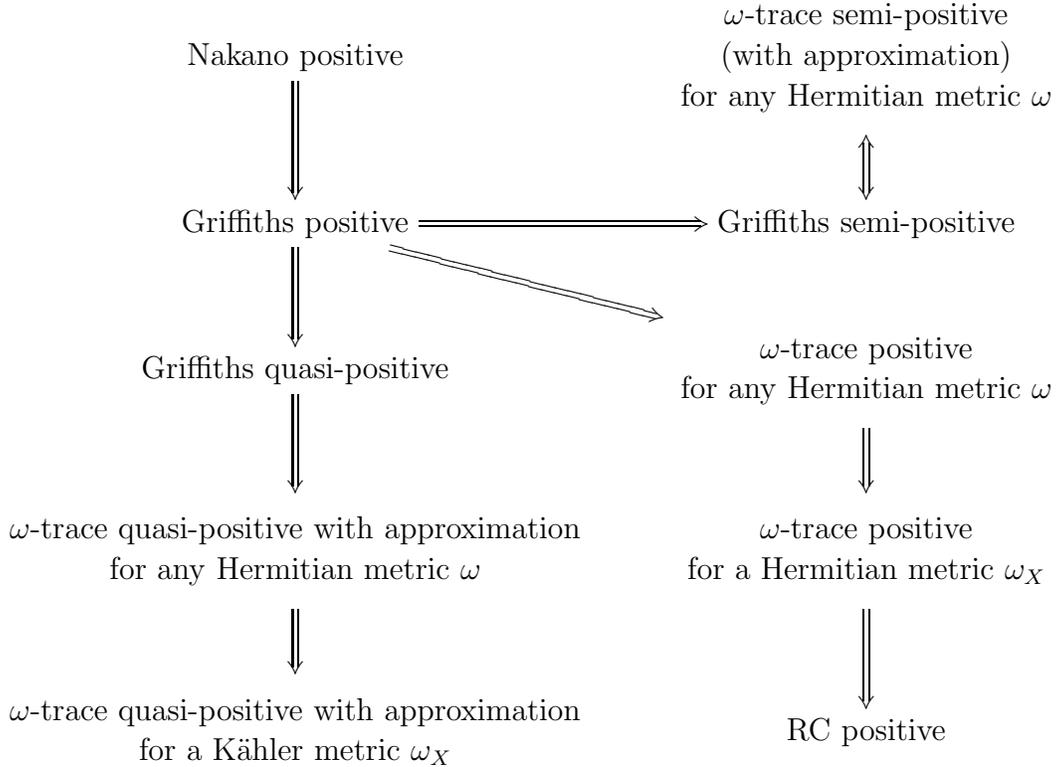
\vspace*{-2mm}
\begin{align*}
    \xymatrix{\mathrm{Nakano~positive}\ar@{=>}[d] & \colorbox{white}{$\begin{matrix}
        \omega\mathrm{\mbox{-}trace~ semi\mbox{-}positive}\\ 
        \mathrm{(with ~approximation)} \\
        \mathrm{for~ any~ Hermitian~ metric~} \omega \\
        \end{matrix}
    $}\ar@{<=>}[d] \\
        \mathrm{Griffiths~positive} \ar@{=>}[d]\ar@{=>}[dr]\ar@{=>}[r]& \mathrm{Griffiths~semi\mbox{-}positive} \\
        \mathrm{Griffiths~ quasi\mbox{-}positive} \ar@{=>}[d]  &
        \colorbox{white}{$\begin{matrix}
            \omega\mathrm{\mbox{-}trace~ positive} \\
            \mathrm{for~ any~ Hermitian~ metric~} \omega \\
            \end{matrix}
        $} \ar@{=>}[d] \\ %\omega\mathrm{\mbox{-}trace~ quasi\mbox{-}positive~with~approximation},~\forall\,\omega:\mathrm{K\ddot{a}hler} \\ %\ar@{=>}[d] \\
        \colorbox{white}{$\begin{matrix}
            \omega\mathrm{\mbox{-}trace~ quasi\mbox{-}positive~with~approximation} \\
            \mathrm{for~ any~ Hermitian~ metric~} \omega \\
            \end{matrix}
        $} \ar@{=>}[d] & \colorbox{white}{$\begin{matrix}
            \omega\mathrm{\mbox{-}trace~ positive} \\
            \mathrm{for~ a~ Hermitian~ metric~} \omega_X \\
            \end{matrix}
        $}\ar@{=>}[d] \\
        \colorbox{white}{$\begin{matrix}
            \omega\mathrm{\mbox{-}trace~ quasi\mbox{-}positive~with~approximation} \\
            \mathrm{for~ a~ K\ddot{a}hler~ metric~} \omega_X \\
            \end{matrix}
        $} & \mathrm{RC~positive}
    }
\end{align*}

We weaken singular positivity of Theorem \ref{V-thm of Inayama} as much as possible
and generalize it from projective manifolds to compact \kah manifolds.
%and obtain the vanishing of $n$-th cohomology group with respect to $E$ larger than the $L^2$-subsheaf $\mathscr{E}(h)$.
%正値性を quasi にまで弱めるので, 既知の$L^2$-評価を用いた手法を用いる事が出来なくなり、$\omega$-trace positivity にまで弱めるので大域的な approximation も存在しない.
%Since positivity is weakened to quasi-positivity, it is no longer possible to use the known $L^2$-estimate method, and the Hodge decomposition cannot be used as in the smooth case.
From Proposition \ref{tr-quasi-posi with appro then h^m is also tr-quasi-posi with appro} and \ref{dual deg max then H^0=0} and Theorem \ref{tr-posi then deg max}, we have the following.

\begin{theorem}\label{V-thm for tr-quasi-posi with appro}
    Let $X$ be a compact \kah manifold and $E$ be a holomorphic vector bundle over $X$ equipped with a singular Hermitian metric $h$. 
    If there exists a \kah metric $\omega$ such that $h$ is $\omega$-trace quasi-positive with approximation, 
    then %$E$ is generically $\omega^{n-1}$ strictly positive and 
    %for any $m\in\mathbb{N}$ 
    we have that 
    \begin{align*}
        H^0(X,(E^*)^{\otimes m})&=0,\\
        H^0(X,\Lambda^pE^*)&=0,
    \end{align*}
    for any $m\in\mathbb{N}$ and any $1\leq p\leq \mathrm{rank}\,E$.
\end{theorem}

\begin{corollary}\label{V-thm for Grif quasi-posi}
    Let $X$ be a compact \kah manifold and $E$ be a holomorphic vector bundle equipped with a singular Hermitian metric $h$. 
    If $h$ is Griffiths quasi-positive then %we have that 
    \begin{align*}
        H^0(X,(E^*)^{\otimes m})&=0,\\
        H^0(X,\Lambda^pE^*)&=0,
    \end{align*}
    for any $m\in\mathbb{N}$ and any $1\leq p\leq \mathrm{rank}\,E$.
\end{corollary}

Furthermore, we have the following theorem by Theorem \ref{tr-posi then RC-posi} and Proposition \ref{RC-posi then H^0=0}.

\begin{theorem}
    Let $X$ be a compact manifold and $E$ be a holomorphic vector bundle equipped with a singular Hermitian metric $h$. 
    If there exists a Hermitian metric $\omega$ such that $h$ is $\omega$-trace positive, 
    then %$E$ is generically $\omega^{n-1}$ strictly positive and 
    %for any $m\in\mathbb{N}$ 
    we have that 
    \begin{align*}
        H^0(X,(E^*)^{\otimes m})&=0,\\
        H^0(X,\Lambda^pE^*)&=0,
    \end{align*}
    for any $m\in\mathbb{N}$ and any $1\leq p\leq \mathrm{rank}\,E$.
\end{theorem}

%Finally, 
We consider the cohomology vanishing of direct image sheaves as an application.

\begin{corollary}
    Let $f:X\to Y$ be a smooth fibration of compact \kah manifolds and $L$ be a holomorphic line bundle on $X$ equipped with a singular Hermitian metric $h$.
    We assume that $h$ is singular semi-positive and there exist $t_0\in Y$ and a neighborhood $U$ of $t_0$ satisfying the following % such that
    \begin{itemize}
        \item [$(a)$] $h$ is singular positive on $f^{-1}(U)$,
        \item [$(b)$] $\mathscr{I}(h|_{X_t})=\mathcal{O}_{X_t}$ for any $t\in U$ where $X_t:=f^{-1}(t)$,
        \item [$(c)$] there exists an analytic subset $A$ which intersects $X_{t_0}$ transversely and satisfies that $f^{-1}(U)\setminus A\Subset S$ for a Stein subset $S\subset X$. %and that $H$ は $X_{t_0}$ と横断的に交わる.
    \end{itemize}
    %$h$ is singular positive on $f^{-1}(U)$ and that $\mathscr{I}(h|_{X_t})=\mathcal{O}_{X_t}$ for any $t\in U$, 
    Then the torsion-free coherent sheaf $\mathscr{F}:=f_*(K_{X/Y}\otimes L\otimes\mathscr{I}(h))$ has the canonical singular Hermitian metric $H$ defined by fiber integrals and this metric $H$ is Griffiths quasi-positive,
    and we have that $H^0(Y,\mathscr{F}^*)=0$. 
    %Moreover, we have that $H^0(Y,\mathscr{F}^*)=0$. 
    
    In particular, if $Y$ is a scheme then we have that 
    \begin{align*}
        H^{\,\mathrm{dim}\,Y}\!(Y,K_Y\otimes\mathscr{F})=H^{\,\mathrm{dim}\,Y}\!(Y,f_*(K_X\otimes L\otimes\mathscr{I}(h)))=0.
    \end{align*}
\end{corollary}

\begin{proof}
    We already know that $H$ is Griffiths semi-positive (see \cite{PT18,HPS18}).
    Let $X(\mathscr{F})\subseteq X$ denote the maximal open subset where $\mathscr{F}$ is locally free, 
    then $Z_\mathscr{F}:=X\setminus X(\mathscr{F})$ is a closed analytic subset of codimension $\geq2$. 
    Here, the restriction of $\mathscr{F}$ to the open subset $X(\mathscr{F})$ is a holomorphic vector bundle. 
    By the assumption $(b)$ and \cite[Lemma\,22.1]{HPS18}, we get $f^{-1}(U)\subseteq X(\mathscr{F})$.

    By the assumption $(a)$ and $(c)$, there exists a strictly plurisubharmonic function $\psi$ on $f^{-1}(U)\setminus A$ and we can write $h=e^{-\psi}$ a.e.
    Let $\omega_0$ be the standard \kah metric on $U$, then $i\Theta_{L,h}=\idd\psi\geq\delta\cdot f^*\omega_0$ on $f^{-1}(U)$ in the sense of currents for some $\delta>0$. 
    From \cite[Theorem\,3.8]{Bou17}, we can find a sequence of smooth functions $(\psi_j)_{j\in\mathbb{N}}$ decreasing to $\psi$ pointwise a.e. such that $\idd\psi_j\geq\delta\cdot f^*\omega_0$ on $f^{-1}(U)\setminus A$ for any $j\in\mathbb{N}$. 
    Define the canonical Hermitian metric $H_j$ on $\mathscr{F}|_{f^{-1}(U)}$ induced by $\psi_j$. Here, $\psi_j$ is extended as $0$ on $A$.
    Then, by replacing $U$ small enough if necessary, $A$ intersects each fiber $X_t$ transversely for any $t\in U$ and $H_j$ is smooth. 
     
    We have that $(H_j)_{j\in\mathbb{N}}$ increases and pointwise convergence to $H$ a.e. Let $c_f=i^{(\mathrm{dim}\,X-\mathrm{dim}\,Y)^2}$.
    In fact, $H$ is bounded a.e. and $H(t)<+\infty$ for any $t\in U$ and any $u,v\in\mathscr{F}_t=H^0(X_t,K_{X_t}\otimes L|_{X_t})$, we already define and obtain
    \begin{align*}
        (u,v)_H(t)=\int_{X_t}c_f u\wedge\overline{v}e^{-\psi}<+\infty, \quad (u,v)_{H_j}(t)=\int_{X_t}c_f u\wedge\overline{v}e^{-\psi_j}. %\!\quad c_f=i^{(\mathrm{dim}\,X-\mathrm{dim}\,Y)^2}.
    \end{align*}
    %where $c_f=i^{(\mathrm{dim}\,X-\mathrm{dim}\,Y)^2}$. 

    Hance, by Fatou's lemma and $(\psi_j)_{j\in\mathbb{N}}$ decreasing to $\psi$ pointwise a.e., we have that 
    \begin{align*}
        \limsup_{j\to+\infty}\int_{X_t}c_f u\wedge\overline{v}e^{-\psi_j}&\leq\int_{X_t}\limsup_{j\to+\infty}c_f u\wedge\overline{v}e^{-\psi_j}=\int_{X_t}c_f u\wedge\overline{v}e^{-\psi}\\
        &=(u,v)_H(t)=\int_{X_t}\liminf_{j\to+\infty}c_f u\wedge\overline{v}e^{-\psi_j}\\
        &\leq\liminf_{j\to+\infty}\int_{X_t}c_f u\wedge\overline{v}e^{-\psi_j},
    \end{align*}
    i.e. $\lim_{j\to+\infty}(u,v)_{H_j}(t)=(u,v)_H(t)$. 

    Then we have that $i\Theta_{\mathscr{F},H_j}\geq \delta\cdot\omega_0$ on $U$ in the sense of Nakano by Deng-Ning-Wang-Zhou's approach (see \cite{DNWZ23}, \cite[Theorem\,5.10 and 5.11]{WZ24}).
    % これは f^{-1}(U)\setminus A で L^2-estimate を解いて解を拡張すれば良い.
    In particular, $\idd|u|^2_{H_j^*}\geq\delta|u|^2_{H_j^*}\idd|z|^2$ for any local holomorphic section $u\in H^0(U,\mathscr{F}^*)$ and for a local coordinate $(z_1,\cdots,z_n)$ of $U$ satisfying $\omega_0\geq\idd|z|^2$.
    %Here, $\mathscr{F}^*=\mathcal{H}om(\mathscr{F},\mathcal{O}_Y)$.
    By Proposition \ref{characterization of Grif nega by smoothing}, $H^*$ is Griffiths negative at $t_0$ and $H$ is Griffiths quasi-positive on $Y$.

    Here, $\mathscr{F}$ is holomorphic vector bundle on $X(\mathscr{F})=X\setminus Z_{\mathscr{F}}$. 
    By Griffiths semi-positivity of $H$ on $Y$, for any point $p\in Z_{\mathscr{F}}$ there exist an open neighborhood $V$ of $p$ and a sequence of Griffiths semi-positive Hermitian metrics $(H_\nu)_{\nu\in\mathbb{N}}$ on $V$ increasing to $H$ pointwise a.e. such that $H_\nu|_{V\setminus Z_{\mathscr{F}}}$ is smooth.
    Thus, even if $\mathscr{F}$ is not a holomorphic vector bundle on $Y$, similarly to Theorem \ref{tr-posi then deg max} we have that $\mathscr{F}$ is $\mathrm{deg}\,_\omega$-strictly maximal and $\mathrm{deg}\,_\omega(\mathscr{F})>0$ for any \kah metric $\omega$ on $Y$. 
    From Proposition \ref{dual deg max then H^0=0}, we obtain $H^0(Y,\mathscr{F}^*)=0$.
    Finally, if $Y$ is a scheme then we have that  %we have $H^0(Y,\mathscr{F}^*)=\mathrm{Ext}^0_Y(\mathscr{F},\mathcal{O}_Y)\cong\mathrm{Ext}^0_Y(K_Y\otimes\mathscr{F},K_Y)\cong H^{\,\mathrm{dim}\,Y}\!(Y,K_Y\otimes\mathscr{F})$
    \begin{align*}
        H^0(Y,\mathscr{F}^*)%=\mathcal{H}om_Y(\mathscr{F},\mathcal{O}_Y)&
        =\mathrm{Ext}^0_Y(\mathscr{F},\mathcal{O}_Y)%\\
        \cong\mathrm{Ext}^0_Y(K_Y\otimes\mathscr{F},K_Y)\cong H^{\,\mathrm{dim}\,Y}\!(Y,K_Y\otimes\mathscr{F})
    \end{align*}
    by the Serre duality for Cohen-Macaulay schemes. %も Cohen-Macaulay schemes で成り立つもので cpt \kah でも良いのか???
\end{proof}

In particular, if $X$ and $Y$ are projective and further assume nef-ness of $L$, then $L$ is big from singular quasi-positivity (see \cite[Corollary\,6.19]{Dem10}), and more strongly the following is known.

\begin{theorem}$($\textnormal{cf.\,\cite[Theorem\,1.3 and Corollary\,1.4]{WZ24}}$)$
    Let $f:X\to Y$ be a smooth fibration of smooth projective varieties and $L$ be a holomorphic line bundle on $X$. 
    If $L$ is nef and big, then the direct image sheaf $\mathcal{F}:=f_*(K_{X/Y}\otimes L)$ is also nef and Viehweg-big, and has a induced singular Hermitian metric $H$ satisfying $\mathscr{E}(H)=\mathcal{O}_Y(\mathcal{F})$, and there exists a proper analytic subset $Z$ such that $H$ is smooth and Nakano positive on $X\setminus Z$. 
    
    Here, $\mathscr{E}(H)$ is the $L^2$-subsheaf defined by $\mathscr{E}(H)_y:= \{s_y \in \mathcal{O}(\mathcal{F})_y \mid |s|^2_{H}\,\, \text{is locally}$ $\text{integrable around} \,\, y \}$.
    Moreover, we have the following cohomology vanishing
    \begin{align*}
        H^q(Y,f_*(K_X\otimes L))=0
    \end{align*}
    for any integers $q>0$. 
    %Here, the holomorphic vector bundle $f_*(K_{X/Y}\otimes L)$ has a induced singular Hermitian metric $H$ which satisfies $\mathscr{E}(H)=\mathcal{O}_Y(\mathcal{F})$, and there exists a proper analytic subset $Z$ such that $H$ is smooth and Nakano positive on $X\setminus Z$.
\end{theorem}

\subsection{global generation of adjoint linear series for Griffiths quasi-positivity}

In \cite{Fuj88}, Fujita proposed the following conjecture which is a open question in classical algebraic geometry.
%Recall that, $X$ is an $n$-dimensional complex manifold.
In this subsection, we provide a global generation theorem of the Fujita conjecture type, involving $L^2$-subsheaves.

\begin{conjecture}
    Let $X$ be a smooth projective variety and $L$ be an ample line bundle. %Then 
    \begin{itemize}
        \item $K_X\otimes L^{\otimes(\mathrm{dim}\,X+1)}$ is globally generated;
        \item $K_X\otimes L^{\otimes(\mathrm{dim}\,X+2)}$ is very ample.
    \end{itemize}
\end{conjecture}

The global generation conjecture has been proved (cf. \cite{YZ20,Kaw97}) %[12,\,22,\,40]) % \cite{YZ20}, \cite{EL93}, \cite{Kaw97}) 
up to dimension $5$.
Recently, Fujita's conjecture type theorems was obtained in \cite{SY19} for the case of pseudo-effective involving the multiplier ideal sheaf. % and for the case of nef involving Nakano semi-positive vector bundles, as follows.

\begin{theorem}\label{Fujita Conj thm for psef}$($\textnormal{cf.\,\cite[Theorem\,1.3]{SY19}}$)$ %$(\mathrm{cf.\,[35,\,Theorem\,1.3]})$ %[SY19]
    Let $X$ be a compact \kah manifold, $L$ be an ample and globally generated line bundle and $(B,h)$ be a pseudo-effective line bundle.
    If the numerical dimension of $(B,h)$ is not zero, i.e. $\mathrm{nd}(B,h)\ne0$. then
    \begin{align*}
        K_X\otimes L^{\otimes n}\otimes B\otimes \mathscr{I}(h)
    \end{align*}
    is globally generated.
\end{theorem}

After that, the following global generation theorem involving $L^2$-subsheaves was obtained in \cite{Wat23} for singular Griffiths semi-positivity.
%the Fujita's conjecture type theorem involving $L^2$-subsheaves in \cite{Wat23} for singular Griffiths semi-positivity. 文になってないww

\begin{theorem}$($\textnormal{cf.\,\cite[Theorem\,1.9]{Wat23}}$)$\label{Fujita Conj for Grif}
    Let $X$ be a compact \kah manifold and $E$ be a holomorphic vector bundle equipped with a singular Hermitian metric $h$. 
    Let $L$ be an ample and globally generated line bundle and $N$ be a nef but not numerically trivial line bundle.
    If $h$ is Griffiths semi-positive and there exists a smooth effective ample divisor $A$ such that the Lelong number $\nu(-\log\mathrm{det}\,h|_A,x)<1$ for all points in $A$ and 
    that the equality of numerical dimensions $\mathrm{nd}(N|_A)=\mathrm{nd}(N)$, then the coherent sheaf
    \begin{align*}
        K_X\otimes L^{\otimes n}\otimes N\otimes \mathscr{E}(h\otimes\mathrm{det}\,h)
    \end{align*}
    is globally generated.
\end{theorem}

Here, for any nef line bundle $N$, if $\mathrm{nd}(N)\ne n$, i.e. $N$ is not big, then we can always select a nonsingular ample divisor $A$ satisfying $\mathrm{nd}(N|_A)=\mathrm{nd}(N)$.
However, this theorem is complicated by requiring conditions of the Lelong number and of the numerically dimension.
As an application of Corollary \ref{V-thm for Grif quasi-posi}, we obtain the following theorem for the simpler assumption of singular Griffiths quasi-positivity.

\begin{theorem}\label{Fujita Conj for Grif quai-posi}
    Let $X$ be a compact \kah manifold and $L$ be an ample and globally generated line bundle. 
    Let $E$ be a holomorphic vector bundle equipped with a singular Hermitian metric $h$. 
    If $h$ is Griffiths quasi-positive then the coherent sheaf 
    \begin{align*}
        K_X\otimes L^{\otimes n}\otimes \mathscr{E}(h\otimes\mathrm{det}\,h)
    \end{align*}
    is globally generated.
\end{theorem}

\begin{proof}
    By the Mumford regularity, we already know that if $\mathcal{F}$ is a $0$-regular coherent sheaf on $X$ with respect to $L$, i.e. $H^0(X,\mathcal{F}\otimes L^{\otimes(n-q)})=0$ for any $q>0$, then $\mathcal{F}$ is generated by its global sections (see \cite[Theorem\,1.8.4]{Laz04}).
    Then, we only need to prove $K_X\otimes L^{\otimes n}\otimes \mathscr{E}(h\otimes\mathrm{det}\,h)$ is $0$-regular with respect to $L$.
    For $0<q<n$, we have the following cohomologies vanishing 
    \begin{align*}
        H^q(X,K_X\otimes L^{\otimes(n-q)}\otimes \mathscr{E}(h\otimes\mathrm{det}\,h))=0
    \end{align*}
    by the positivity of $L^{\otimes(n-q)}$ and the already known Nakano vanishing theorem \cite[Theorem\,1.6]{Wat23}.
    Obviously, the $q$-th cohomology always vanishes in the case of $q>n$.

    Finally, we show the case of $q=n$. By the natural inclusion $\mathscr{E}(h\otimes\mathrm{det}\,h)\hookrightarrow E\otimes\mathrm{det}\,E$ and Corollary \ref{V-thm for Grif quasi-posi}, 
    we obtain
    \begin{align*}
        H^n(X,K_X\otimes \mathscr{E}(h\otimes\mathrm{det}\,h))\hookrightarrow H^n(X,K_X\otimes E\otimes\mathrm{det}\,E)=0.
    \end{align*}
    Hance, we have that $H^n(X,K_X\otimes \mathscr{E}(h\otimes\mathrm{det}\,h))=0$ and this proof is completed.
\end{proof}

\section{$\omega$-trace positivity, rational conected-ness and generically $\omega_X^{n-1}$ positivity}
\subsection{$\omega$-trace positivity and rational conected-ness}

A projective manifold $X$ is called $\it{rationally}$ $\it{connected}$ if any two points of $X$ can be connected by some rational curve.
In this subsection, we show that rational conected-ness follows from $\omega$-trace quasi-positivity with approximation.
Yau's conjecture was solved in \cite{HW20,Mat22}, and in \cite{Yan18} by introducing the notion of RC-positivity.

\begin{conjecture}$($\textnormal{Yau's conjecture,\,\cite[Problem\,47]{Yau82}}$)$\label{Yau's conjecture}
    Let $X$ be a compact \kah manifold. If $X$ has a \kah metric with positive holomorphic sectional curvature, 
    then $X$ is projective and rationally connected manifold.
\end{conjecture}

After that, Yau's conjecture was conjectured by weakening the assumption to quasi-positivity (see \cite[Conjecture\,1.9]{Yan20}) and was recently solved in \cite{ZZ23}.

\begin{theorem}$($\textnormal{cf.\,\cite[Theorem\,1.6]{ZZ23}}$)$
    Let $(X,\omega_g)$ be a compact \kah manifold with nonnegative holomorphic sectional curvature. 
    Then the following statements hold:
    \begin{itemize}
        \item For any holomorphic $(p,0)$-form $\eta$, we have $|\eta|^2_{\omega_g}\equiv C$ for some constant $C\geq0$.
        \item If $(X,\omega)$ has no nonzero truly flat tangent vector at some point (which is satisfied when the holomorphic sectional curvature is quasi-positive), 
        then we have that $H^0(X,\Omega^p_X)=0$ for every $1\leq p\leq\mathrm{dim}\,X$. In particular, $X$ is projective and rationally connected. 
    \end{itemize}
\end{theorem}

Here, a \kah metric $\Omega$ on $T_X$ is RC-positive if $\Omega$ has positive holomorphic sectional curvature or is $\omega$-trace positive for a Hermitian metric $\omega$ on $X$.
Then, we obtain the following theorem as a generalization to singular Hermitian metrics and quasi-positivity with respect to $\omega$-trace positivity.
In fact, this theorem is already known if $h$ is smooth and $\omega$-trace positive (see \cite[Corollary\,1.5]{Yan18}).

\begin{theorem}\label{rationally conected if tr-quasi posi}
    Let $X$ be a compact \kah manifold. %and $h$ be a singular Hermitian metric on $T_X$.
    If there exist a Hermitian metric $\omega$ $($not necessarily K\"ahler$)$ on $X$ and a singular Hermitian metric $h$ on $T_X$ such that $h$ is $\omega$-trace quasi-positive with approximation, then $X$ is projective and rationally connected.
\end{theorem}

\begin{proof}
    We have that $H^0(X,\Omega^p_X)=0$ for any $1\leq p\leq n$ by Theorem \ref{V-thm for tr-quasi-posi with appro}.
    From $H^0(X,\Omega^2_X)=0$ and Kodaira's theorem (see \cite[Proposition\,3.3.2 and Corollary\,5.3.3]{Huy05}), the \kah manifold $X$ is projective.
    By Proposition \ref{tr-quasi-posi with appro then h^m is also tr-quasi-posi with appro}, the singular Hermitian metric $\bigwedge^ph$ on $\bigwedge^pT_X$ is also $\omega$-trace quasi-positive with approximation for any $1\leq p\leq n$.
    For any invertible subsheaf $\mathcal{F}\subset \Omega^p_X=\mathcal{O}_X(\bigwedge^pT^*_X)$ has the induced singular metric $h_{\mathcal{F}}$ which is also $\omega$-trace quasi-negative (with approximation) from Proposition \ref{subbdl is also tr-negative}.
    Hance, $\mathcal{F}$ is not pseudo-effective by Theorem \ref{characterization of tr-posi, not psef}. 
    Thus, the proof is completed from \cite[Theorem\,1.1]{CDP14}.
\end{proof}

\subsection{generically $\omega_X^{n-1}$ positivity}

In this subsection, we clarify the relationship between the following notion of generically semi-positivity and $\omega$-trace positivity.

\begin{definition}$($\textnormal{cf.\,\cite[Section\,6]{Miy87}}$)$\label{def of generically omega posi}
    Let $X$ be a compact \kah manifold and $E$ be a holomorphic vector bundle on $X$. 
    Let $\omega_1,\cdots,\omega_{n-1}$ be \kah classes. Let 
    \begin{align*}
        0=\mathscr{E}_0\subset\mathscr{E}_1\subset\cdots\subset\mathscr{E}_s=E
    \end{align*}
    be the Harder-Narasimhan semi-stable filtration with respect to $(\omega_1,\cdots,\omega_{n-1})$.
    We say that $E$ is \textit{generically} $(\omega_1,\cdots,\omega_{n-1})$ \textit{semi}-\textit{positive} (resp. \textit{strictly} \textit{positive}), if 
    \begin{align*}
        \int_Xc_1(\mathscr{E}_j/\mathscr{E}_{j-1})\wedge\omega_1\wedge\cdots\wedge\omega_{n-1}\geq0 \quad (\mathrm{resp}.\,\,>0) \quad \mathrm{for~all}~j.
    \end{align*}
    If $\omega_1=\cdots=\omega_{n-1}$, we write the polarization as $\omega^{n-1}_1$ for simplicity.
\end{definition}

The following conjecture and theorem are known for generically positivity.

\begin{conjecture}$($\textnormal{cf.\,\cite[Conjecture\,1.3]{Pet12}}$)$
    Let $X$ be a projective manifold with nef anti-canonical bundle $-K_X$. Then $T_X$ is generically $(H_1,\cdots,H_{n-1})$ semi-positive for any $(n-1)$-tuple of ample divisors $H_1,\cdots,H_{n-1}$.
\end{conjecture}

\begin{theorem}$($\textnormal{cf.\,\cite[Theorem\,0.2]{Cao14}}$)$\label{Cao's thm}
    Let $X$ be a compact \kah manifold with nef anti-canonical bundle $-K_X$ $($resp. nef canonical bundle $K_X)$. 
    Then $T_X$ $($resp. $\Omega_X^1)$ is generically $\omega^{n-1}_X$ semi-positive for any \kah class $\omega_X$.
\end{theorem}

More generally, we show that $E$ has generically $\omega^{n-1}$ positivity if a singular Hermitian metric of $E$ has $\omega$-trace positivity with approximation, and obtain analogous results to Theorem \ref{Cao's thm}.
%Let $E$ be a holomorphic vector bundle on a projective manifold. これは cpt \kah に ext できるか???
Here, it is a sufficient condition for $\mathrm{det}\,E$ to be pseudo-effective that $\mathrm{det}\,E$ is nef or that $E$ is Griffiths semi-positive for singular Hermitian metrics.
%Here, $\mathrm{det}\,E$ is pseudo-effective if $\mathrm{det}\,E$ is nef or $E$ has a singular Hermitian metric with Griffiths semi-positivity.
Theorem \ref{Cao's thm} is the case where $E=T_X$ or $T^*_X$ and $\mathrm{det}\,E$ is nef.

\begin{theorem}\label{Grif posi then generically omega posi}
    Let $X$ be a compact \kah manifold and $E$ be a holomorphic vector bundle equipped with a singular Hermitian metric $h$.
    If $h$ is Griffiths semi-positive $($resp. quasi-positive$)$, then $E$ is generically $\omega^{n-1}_X$ semi-positive $($resp. strictly positive$)$ for any \kah class $\omega_X$.
\end{theorem}

Theorem \ref{Grif posi then generically omega posi} follows immediately from the following theorem and the characterization of singular Griffiths semi-positivity, i.e. Theorem \ref{characterization of sing Grif semi-posi and tr-omega semi-posi}.

\begin{theorem}
    Let $X$ be a compact \kah manifold, $\omega$ be a \kah metric on $X$ and $E$ be a holomorphic vector bundle equipped with a singular Hermitian metric $h$. 
    If $h$ is $\omega$-trace semi-positive $($resp. quasi-positive$)$ with approximation then $E$ is generically $\omega^{n-1}$ semi-positive $($resp. strictly positive$)$.
\end{theorem}

\begin{proof}
    Let $ 0=\mathscr{E}_0\subset\mathscr{E}_1\subset\cdots\subset\mathscr{E}_s=E$
    %\begin{align*}
    %    0=\mathscr{E}_0\subset\mathscr{E}_1\subset\cdots\subset\mathscr{E}_s=E
    %\end{align*}
    be the Harder-Narasimhan filtration of torsion-free subsheaves such that $\mathscr{E}_j/\mathscr{E}_{j-1}$ is the maximal $\omega$-semistable torsion-free subsheaf of $E/\mathscr{E}_{j-1}$ for $1\leq j\leq s$.
    By Lemma \ref{Jac14 Lemma 1}, there exists a desingularization $\pi:\widetilde{X}\to X$ such that $\pi^*E$ admits a filtration:
    \begin{align*}
        0\subset \widetilde{E}_1\subset\widetilde{E}_2\subset\cdots\subset\widetilde{E}_s=\pi^*(E),
    \end{align*}
    where $\widetilde{E}_j,\widetilde{E}_j/\widetilde{E}_{j-1}$ are holomorphic vector bundles and $\pi_*\widetilde{E}_j=\mathscr{E}_j$ outside an analytic subset of codimension at least $2$.
    From Lemma \ref{Jac14 Lemma 2}, we have that 
    \begin{align*}
        \mu_\omega(\mathscr{E}_j/\mathscr{E}_{j-1})=\mu_{\pi^*\omega}(\widetilde{E}_j/\widetilde{E}_{j-1})
    \end{align*}
    and $\widetilde{E}_j/\widetilde{E}_{j-1}$ is a $\pi^*\omega$-semistable of $\pi^*E/\widetilde{E}_{j-1}$ of maximal slope.

    Hance, it is sufficient to prove $\mu_{\pi^*\omega}(\pi^*E/\widetilde{E}_j)\geq0$ (resp. $>0$) for any $1\leq j\leq s-1$.
    This follows from the proof of Theorem \ref{tr-posi then deg max}.
\end{proof}

Finally, we give an application to the Mumford conjecture that a projective manifold $X$ satisfying the cohomology vanishing $H^0(X,(T^*_X)^{\otimes m})=0$ for any $m\geq1$ is rationally connected.

\begin{corollary}\label{Grif semi-posi and Mumford conj}
    Let $X$ be a compact \kah manifold. We assume that the tangent holomoprhic bundle $T_X$ has a singular Hermitian metric with Griffiths semi-positivity. 
    Then the following conditions are equivalent
    \begin{itemize}
        \item [$(a)$] $H^0(X,(T_X^*)^{\otimes m})=0$ for all $m\geq1$,
        \item [$(b)$] $X$ is rationally connected,
        \item [$(c)$] $T_X$ is generically $\omega^{n-1}_X$ strictly positive for some \kah class $\omega_X$,
        \item [$(d)$] $T_X$ is generically $\omega^{n-1}_X$ strictly positive for any \kah class $\omega_X$.
    \end{itemize}
\end{corollary}

This corollary is shown similarly to \cite[Proposition\,0.4]{Cao14}, which changes the assumption of Corollary \ref{Grif semi-posi and Mumford conj} to nef-ness of the anti-canonical bundle $-K_X$. 
In particular, $(a)\Longrightarrow (b)$ follows from the structure theorem (see \cite[Theorem\,1.1]{HIM22}) for singular Griffiths semi-positivity of $T_X$.

%\section{Acknowledgement}
\vspace*{5mm}
{\bf Acknowledgement. } 
The author would like to thank my supervisor Professor Shigeharu Takayama for guidance and helpful advice. 
The author would also like to thank Professor Takahiro Inayama and Professor Shin-ichi Matsumura for useful suggestions and discussions.
%The author also thanks Yoshiaki Suzuki for the helpful discussion about Kodaira-Spencer maps.

%\section{Reference}

\end{document}